\documentclass[11pt]{article}

\usepackage{amsmath,amssymb,latexsym,amsthm,graphicx}
\usepackage{epstopdf}

\usepackage[lofdepth,lofdepth]{subfig}

\newtheorem{theorem}{Theorem}[section]

\newtheorem{lemma}[theorem]{Lemma}

\newtheorem{prop}[theorem]{Proposition}

\newtheorem{rem}[theorem]{Remark}

\newtheorem{defi}[theorem]{Definition}

\newcommand{\mref}[1]{{(\ref{#1})}}

\newcommand{\reftheo}[1]{Theorem~\ref{#1}}

\newcommand{\supp}{\operatorname{supp}}

\newcommand{\mS}{\mathcal{S}}

\newcommand{\mB}{\mathcal{B}}

\newcommand{\mC}{\mathcal{C}}

\newcommand{\mE}{\mathcal{E}}

\newcommand{\mR}{\mathcal{R}}

\newcommand{\mF}{\mathcal{F}}

\newcommand{\mP}{\mathcal{P}}

\newcommand{\mD}{\mathcal{D}}

\newcommand{\mQ}{\mathcal{Q}}

\newcommand{\mL}{\mathcal{L}}

\newcommand{\Int}{\mbox{Int}}

\newcommand{\bd}{{\bf d}}

\newcommand{\rN}{\mathbb{R}}

\newcommand{\mT}{\mathcal{T}}

\newcommand{\mO}{\mathcal{O}}

\newcommand{\uS}{\mathbb{S}}

\newcommand{\intl}{\int\limits}

\newcommand{\cT}{\mathbb{T}}

\newcommand{\ma}{{\bf a}}

\newcommand{\mb}{{\bf b}}

\newcommand{\mc}{{\bf c}}

\newcommand{\mV}{\mathcal{V}}

\newcommand{\mI}{\mathcal{I}}

\newcommand{\wf}{\mbox{WF}}
\newcommand{\eg}{\varepsilon}

\newcommand{\llg}{\lambda}

\newcommand{\Llg}{\Lambda}

\newcommand{\ag}{\alpha}

\newcommand{\bg}{\beta}

\newcommand{\sg}{\sigma}

\newcommand{\Og}{\Omega}

\newcommand{\Ga}{\Gamma}

\newcommand{\ga}{\gamma}

\newcommand{\pdh}{\partial}

\title{On Artifacts in Limited Data Spherical Radon Transform: Curved Observation Surface}

\author{ Lyudmyla L. Barannyk\footnote{Department of Mathematics, University of Idaho, 875 Perimeter Drive, Moscow, Idaho 83844, USA. Email:~barannyk@uidaho.edu.}, J\"urgen Frikel\footnote{Department of Applied Mathematics and Computer Science, Technical University of Denmark, Matematiktorvet 303, 2800 Kgs. Lyngby, Denmark. Email: jyfr@dtu.dk.}, and Linh V. Nguyen\footnote{Department of Mathematics, University of Idaho, 875 Perimeter Drive, Moscow, Idaho 83844, USA. 
Email:~lnguyen@uidaho.edu.}}

\begin{document}

\maketitle

\begin{abstract} 
We study the limited data problem of the spherical Radon transform in two and three dimensional spaces with general acquisition surfaces. In such situations, it is known that the application of filtered-backprojection reconstruction formulas might generate added artifacts and degrade the quality of reconstructions. In this article, we explicitly analyze a family of such inversion formulas, depending on a smoothing function that vanishes to order $k$ on the boundary of the acquisition surfaces. We show that the artifacts are $k$ orders smoother than their generating singularity. Moreover, in two dimensional space, if the generating singularity is conormal satisfying a generic condition then the artifacts are even $k+\frac{1}{2}$ orders smoother than the generating singularity. Our analysis for three dimensional space contains an important idea of lifting up a space. We also explore the theoretical findings in a series of numerical experiments. Our experiments show that a good choice of the smoothing function might lead to a significant improvement of reconstruction quality.
%In this article, we consider the limited data problem for spherical mean transform. We characterize the generation and strength of the artifacts in a reconstruction formula.  In contrast to the third's author work \cite{Artifact-sphere-flat},  the observation surface considered in this article is not flat. Our results are comparable to those obtained in \cite{Artifact-sphere-flat} for flat observation surface. For the two dimensional problem, we show that the artifacts are $k$ orders smoother than the original singularities, where $k$ is vanishing order of the smoothing function. Moreover, if the original singularity is conormal, then the artifacts are $k+\frac{1}{2}$ order smoother than the original singularity. We provide some numerical examples and discuss how the smoothing effects the artifacts visually. For three dimensional case, although the result is similar to that \cite{Artifact-sphere-flat}, the proof is significantly different. We introduce a new idea of lifting the space. 
\end{abstract}

\section{Introduction}\label{S:intro} Let $\mS \subset \rN^n$ be a convex closed smooth hyper-surface. We consider the following spherical Radon transform $\mR f$ of a function $f$ defined in $\rN^n$
$$\mR f (z, r) = \int\limits_{\uS(z,r)} f(y) \, d\sg(y),\quad (z,r) \in \mS \times (0,\infty).$$
Here, $\uS(z,r)$ is the sphere centered at $z$ of radius $r$, and $d\sg$ is its surface measure. This transform appears in several imaging modalities, such as thermo/photoacoustic tomography (e.g., \cite{FPR,FHR,KKun}), ultrasound imaging (e.g., \cite{norton79,norton81}), SONAR (e.g., \cite{QuintoSONAR}), and inverse elasticity (e.g., \cite{BuKar}). For example, in thermo/photoacoustic tomography (TAT/PAT), $f$ is the initial ultrasound pressure generated by the thermo/photo-elastic effect. It contains useful information about the inner structure of the tissue, which can be used, e.g., for cancer detection. On the other hand, the knowledge of $\mR(f)(z,.)$ can be extracted from the ultrasound signals collected by a transducer located at $z \in \mS$, which is called the {\bf observation} surface. One, therefore, can concentrate on finding $f$ given $\mR(f)$. The same problem also arises in other aforementioned image modalities.

%\medskip

%In all  above imaging modalities, the main problem is to find $f$ from $\mR f$. That is, one needs to invert $\mR$.

\medskip

It is commonly assumed that $f$ is supported inside the bounded domain $\Og$ whose boundary is $\mS$. Let us discuss an inversion formula under this assumption. Let $\mP: C_0^\infty(\rN_+) \to C^\infty(\rN_+)$ be the pseudo-differential operator defined by \begin{equation} \label{E:P0} \mP(h)(r) =  \int\limits_{\rN} \int\limits_{\rN_+} e^{i(s^2-r^2) \llg} \, |\llg|^{n-1} \, h(s) \, ds \, d\llg,\end{equation}
and $\mB:C^\infty(\mS \times \rN_+) \to C^\infty(\Og)$ be the back-projection type operator \begin{eqnarray*}\mB(g)(x) = \frac{1}{2 \pi^n} \int\limits_{\mS} \left<z-x,\nu_z \right> \, g(z,|x-z|) \,d\sg(z).\end{eqnarray*}
When $\mS$ is an $(n-1)$-dimensional ellipsoid, one has the following inversion formula \cite{Kun07,natterer2012photo,Halt-Inv} \footnote{The reader is referred to, e.g., \cite{FPR,FHR, palamodov2012uniform,Salman,palamodov2014time} for other inversion formulas.}
 \begin{equation} \label{E:inversion} f(x) = \mB \mP \mR f (x), \quad \mbox{ for all } x \in \Og. \end{equation}
We note here that formula (\ref{E:inversion}) was written in other forms in the above references. The above form, presented in \cite{AMP}, is convenient to analyze from the microlocal point of view. Another advantage of the above form is that it can be implemented straight forwardly: 1) $\mP$ can be computed fast by using Fast Fourier Transform (FFT) and 2) $\mB$ only involves a simple integration on $\mS$. Fig.~\ref{fig:full-circle} is the result of our implementation when $n=2$ and $\mS$ is the unit circle. The image size is $N=2048$ pixels. The sampling data has the resolution of $n_a=n_r=2048$ for the spatial (angular) variable $z=(\cos \theta,\sin \theta)$ and radial variable $r \in [0,2]$. The reconstruction is almost perfect.

\begin{figure}[ht]
\centering
 \subfloat[short for lof][Original phantom]{
   \includegraphics[width=0.25\linewidth]{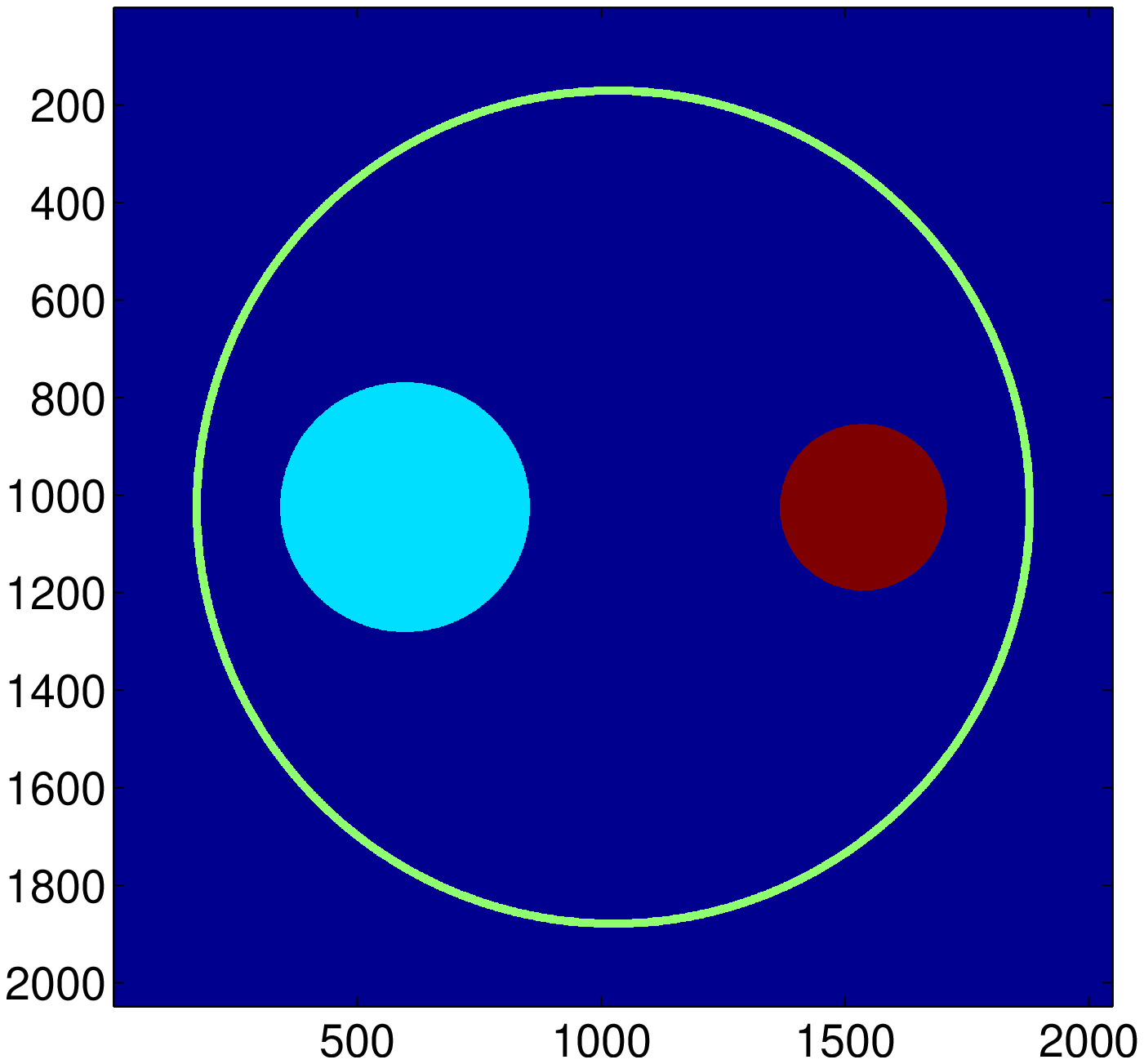}  
   \label{subfig:fig1}
 }
  \hspace{15pt}
 \subfloat[short for lof][Reconstruction]{
   \includegraphics[width=0.25\linewidth]{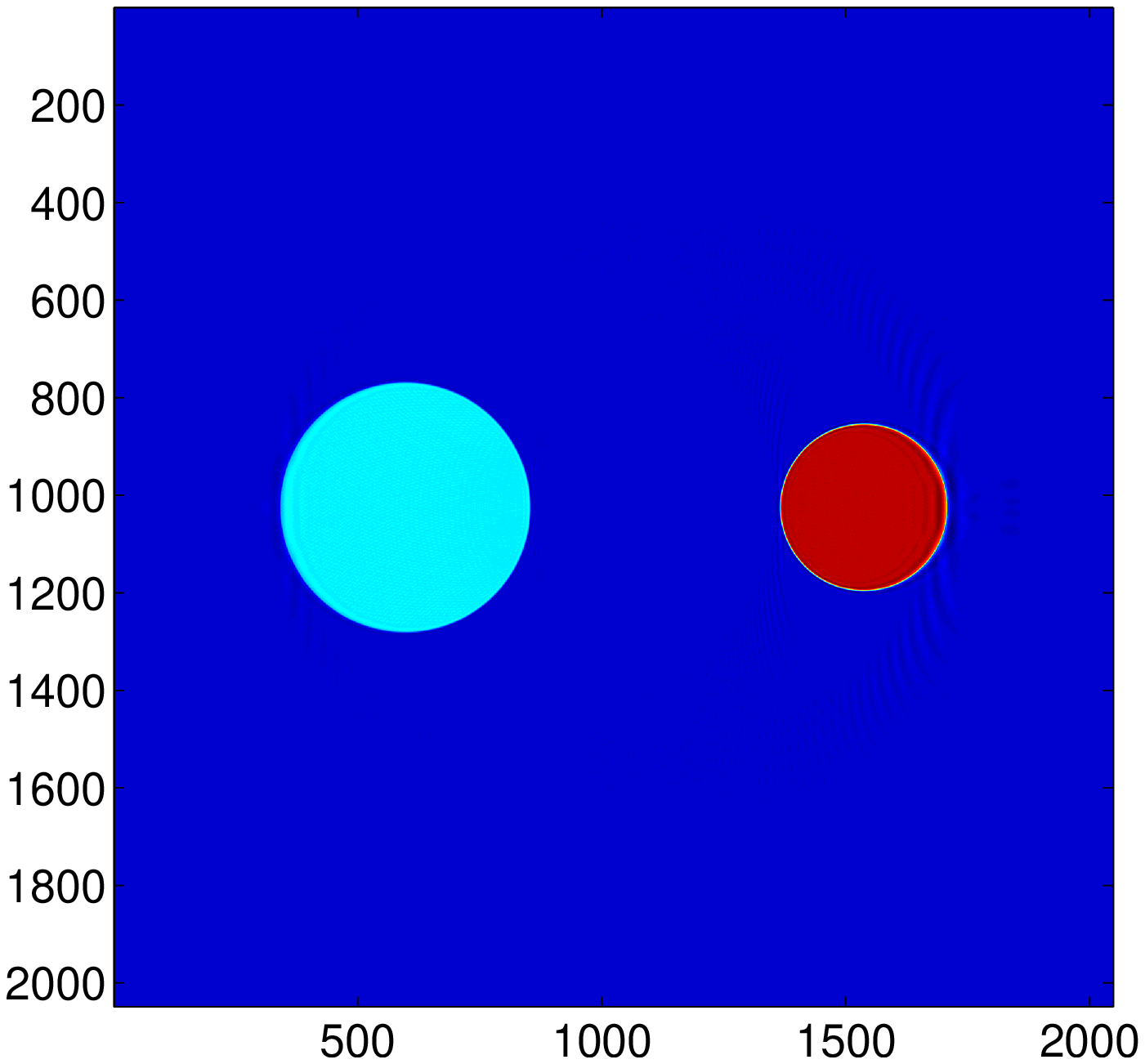}
   \label{subfig:fig2}
}
\hspace{15pt}
 \subfloat[short for lof][Line profile along the central horizontal line: the phantom (left) and reconstruction (right) ]{
   \includegraphics[width=0.30 \linewidth]{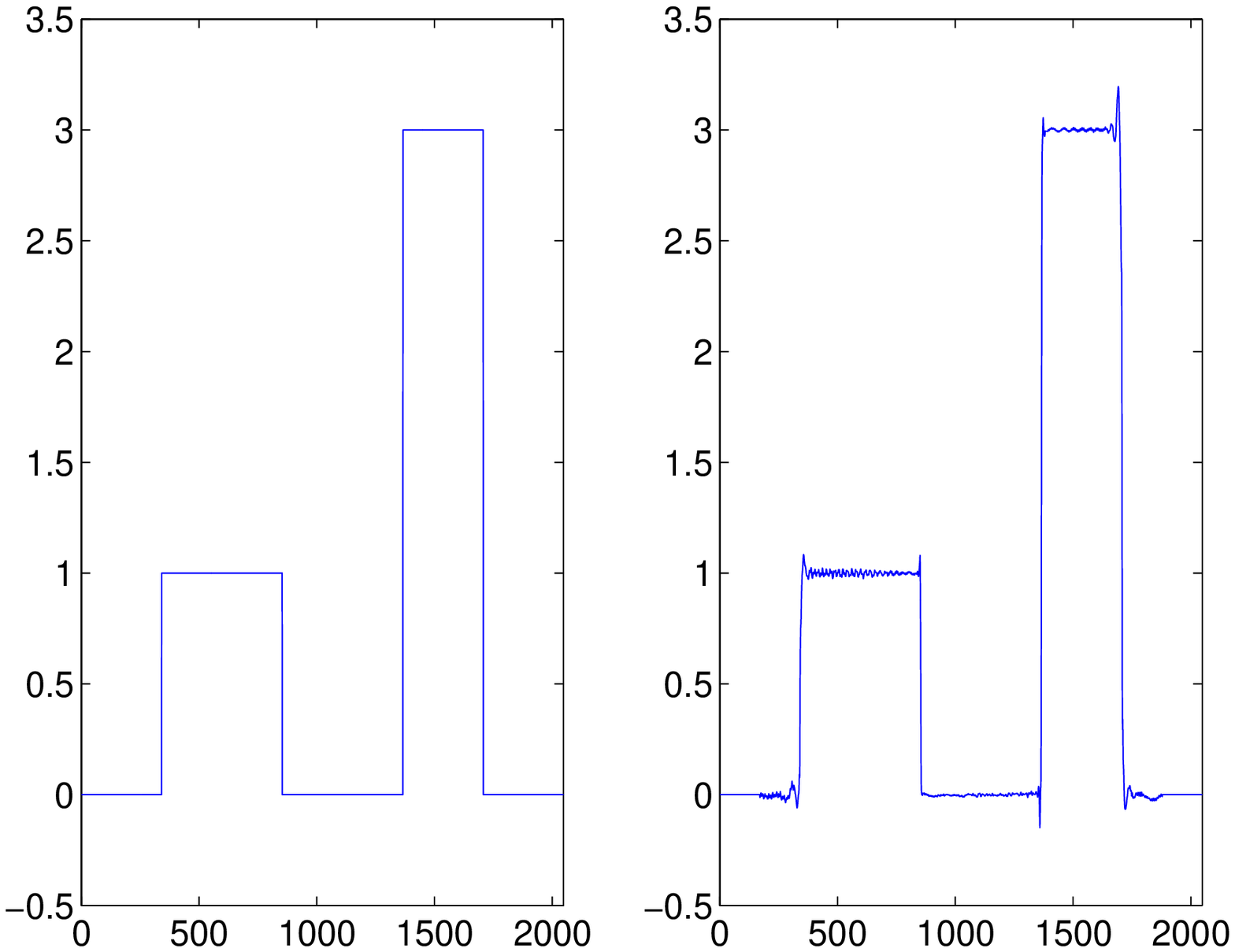}
   \label{subfig:fig2}
}

\caption[short for lof]{Reconstruction using the operator $\mB \mP \mR$ (for $n=2$) when $\mS$ is the \emph{unit circle}. Here, the number of angular samples as well as the number of radial samples of the data $\mR f$ was chosen as $n_a=n_r=2048$.}
\label{fig:full-circle}

\end{figure}

\medskip

\noindent When $\mS$ is a general convex surface, the operator on the right hand side of (\ref{E:inversion}) might not be the identity. However, it only differs from the identity by a compact operator (see \cite{natterer2012photo,Halt-Inv}), which is a pseudo-differential operator of order $-1$ (as shown in \cite{AMP}). Moreover, from the numerical experiments, we observe that $\mB \mP \mR f$ and $f$ are very close, even when $\mS$ is not any ellipse/ellipsoid. For example, in Fig.~\ref{fig:full-polar} we present the reconstruction using $\mB \mP \mR$ for $\mS$ being the polar curve defined by $$\left\{(x,y): x=\frac{1}{2}\big((2 +  \cos \theta) \cos \theta -1\big), \, y= \frac{1}{2}(2 +  \cos \theta) \sin \theta),\, 0 \leq \theta \leq 2 \pi \right\}.$$ The reconstruction is almost perfect. 

\begin{figure}[ht]
\centering
 \subfloat[short for lof][Original phantom]{
   \includegraphics[width=0.25\linewidth]{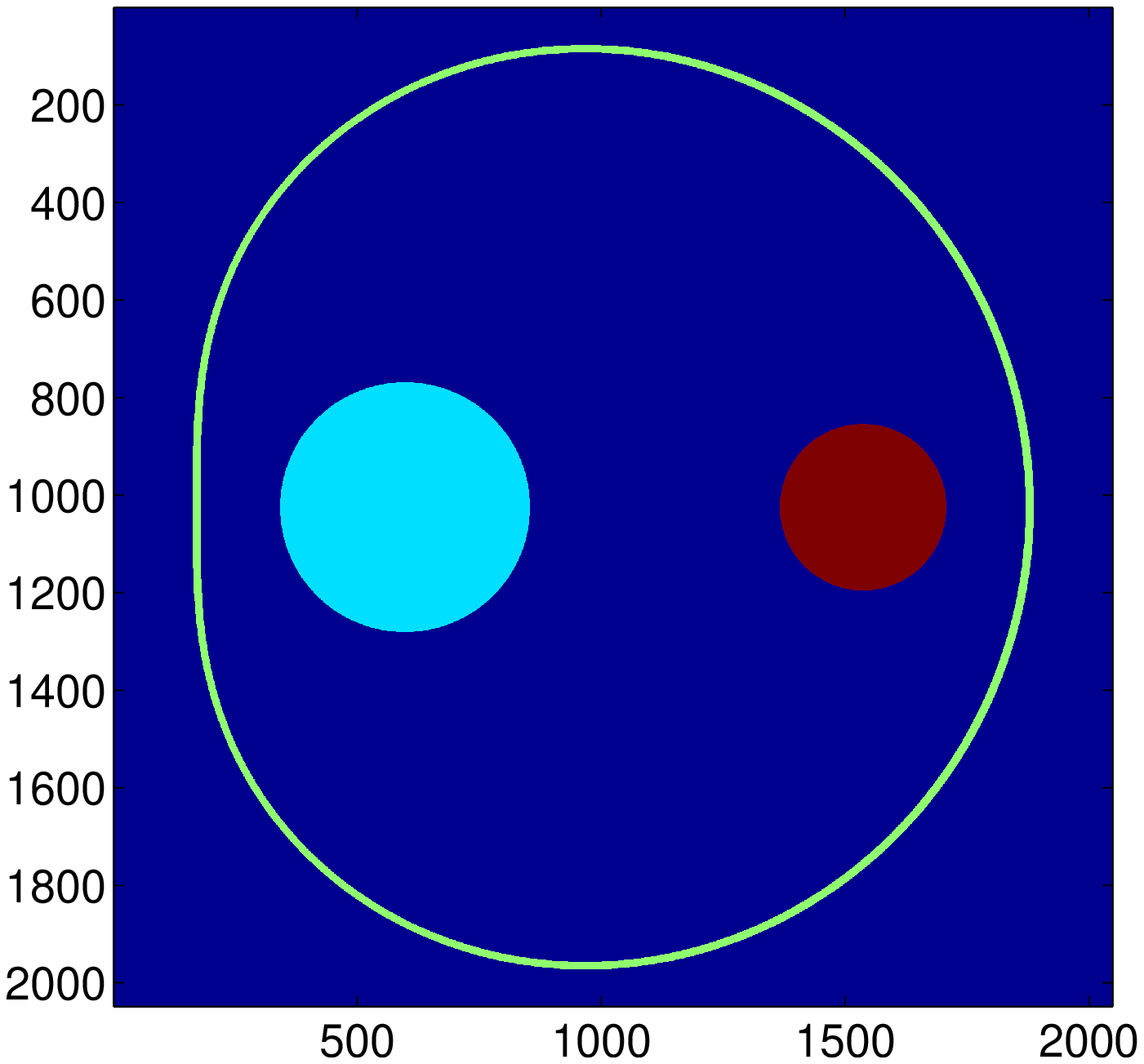}
   \label{subfig:fig1}
 }
  \hspace{15pt}
 \subfloat[short for lof][Reconstruction]{
   \includegraphics[width=0.25\linewidth]{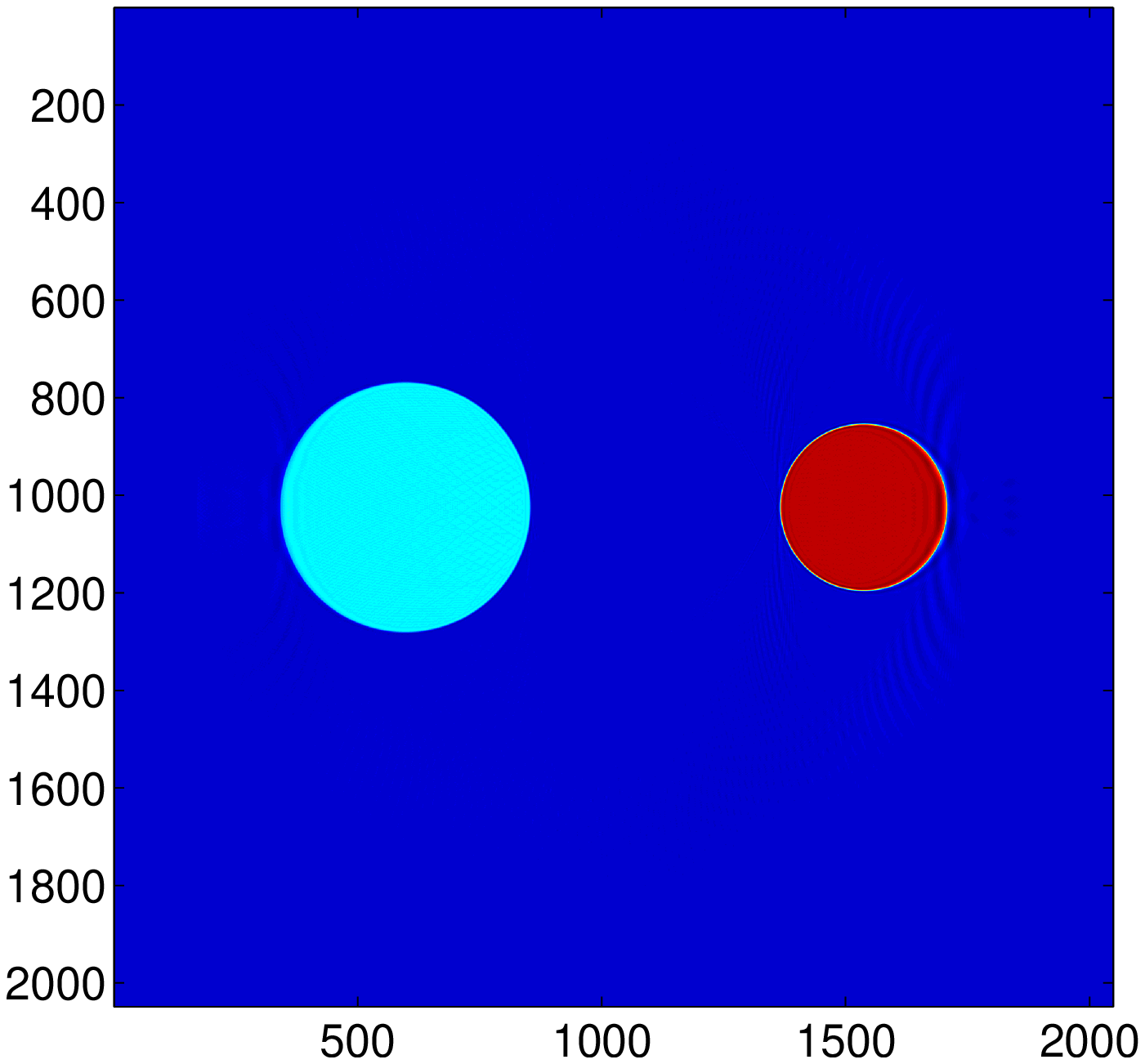} 
   \label{subfig:fig2}
}
\hspace{15pt}
 \subfloat[short for lof][Line profile along the central horizontal line: the phantom (left) and reconstruction (right) ]{
   \includegraphics[width=0.30 \linewidth]{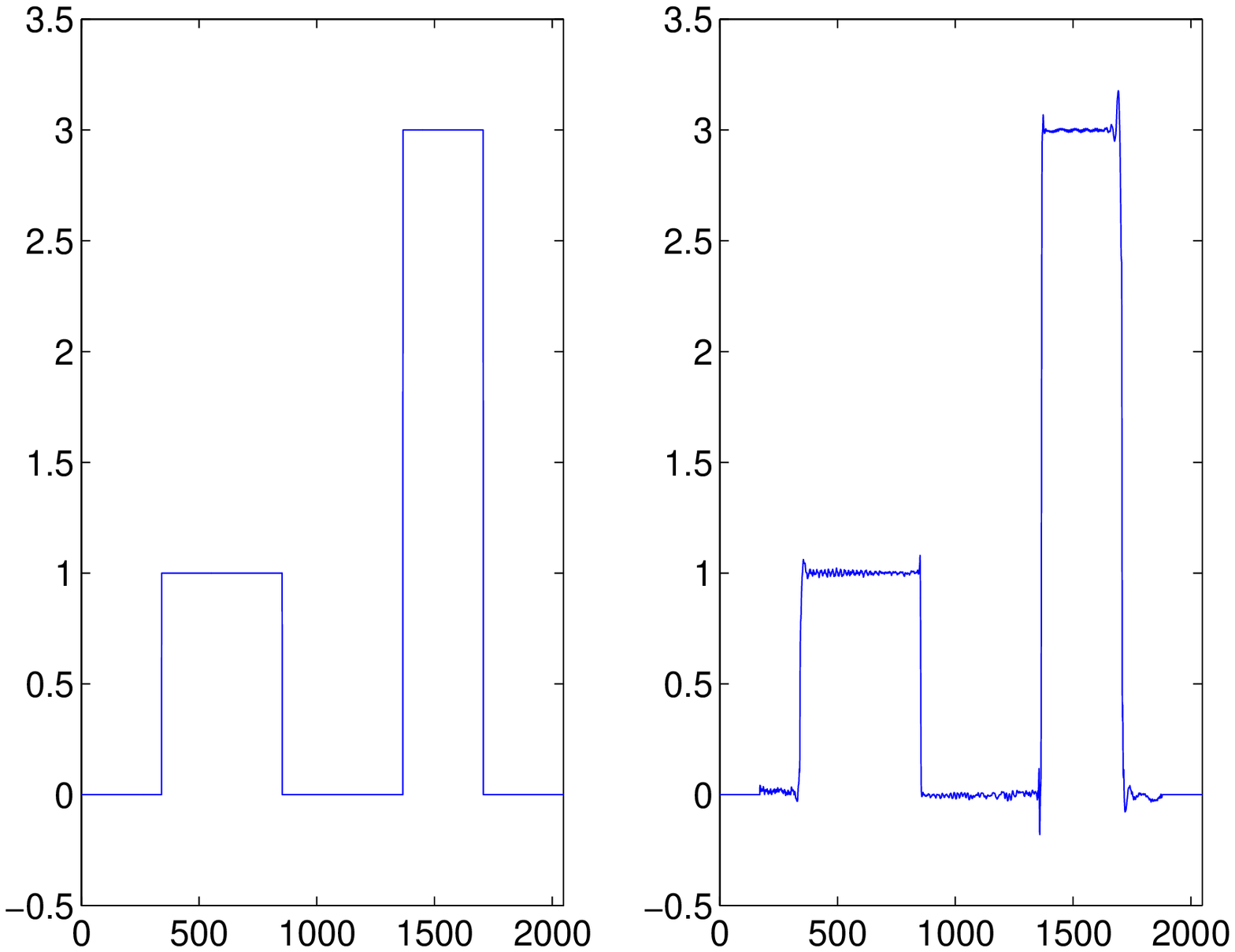} 
   \label{subfig:fig2}
}

\caption[short for lof]{Reconstruction using the operator $\mB \mP \mR$ (for $n=2$) when $\mS$ is a \emph{polar curve}. Here, the number of angular samples as well as the number of radial samples of the data $\mR f$ was chosen as $n_a=n_r=2048$.}

\label{fig:full-polar}
\end{figure}

%\begin{figure}[ht]
%\centering   
%\includegraphics[width=0.25\linewidth]{Figures_paper/original_phantom_polar_curve_two_circles_N2048_Rmax16} \label{Original phantom}
%   \hspace{10pt}
%  \includegraphics[width=0.25\linewidth]{Figures_paper/reconstr2_polar_curve_full_N1024_na8192_nr2048_s0} 
%   \hspace{10pt}
%   \includegraphics[width=0.30\linewidth]{Figures_paper/slices2_polar_curve_full_N1024_na8192_nr2048_s0} 
%%
%\caption[short for lof]{Reconstruction for $n=2$ and $\mS$ is the polar curve. Left panel: phantom; middle panel: reconstruction with $N=1024$, $n_a=n_r=2048$; right panel: slices. UPDATE RECONSTRUCTION AND SLICES WITH HIGHER RESOLUTION}
%\label{fig:full-polar}
%\end{figure}
%
%
%

\medskip

In this article, we are interested in the limited data problem. That is, the knowledge of $\mR f$ is only available on a closed proper subset $\Gamma \subset \mS$ with smooth boundary $\pdh \Ga$ and nontrivial interior $\Int(\Ga)$ (see, e.g., \cite{XWAK08,QuintoSONAR}).  It is natural to consider the following formula
\begin{equation} \label{no_smoothing}
 \mT_0 f (x) := \mB \chi_\Ga \mP \mR f(x),
\end{equation}
where $\chi_\Ga$ is the characteristic function of $\Ga$. As we observe from Fig.~\ref{fig:limited-cir}, formula (\ref{no_smoothing}) reconstructs some singularities and also smoothens out some singularities of the original image. Moreover, it also creates some artifacts (added singularities) into the picture. 
\begin{figure}[ht]
\centering
 \subfloat[short for lof][Original phantom]{
   \includegraphics[width=0.25\linewidth]{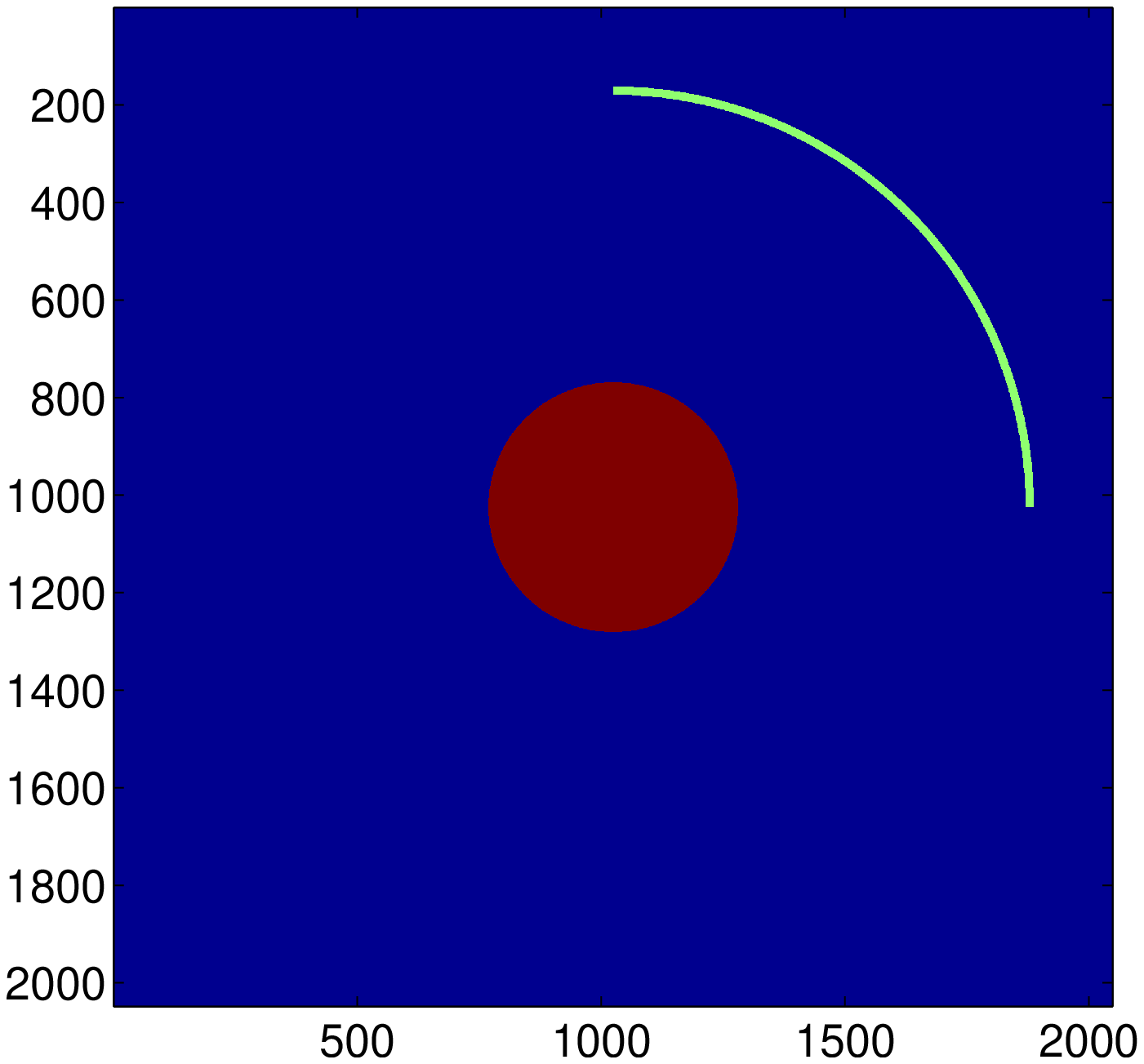}
   \label{subfig:fig1}
 }
  \hspace{20pt}
 \subfloat[short for lof][Reconstruction]{
   \includegraphics[width=0.25\linewidth]{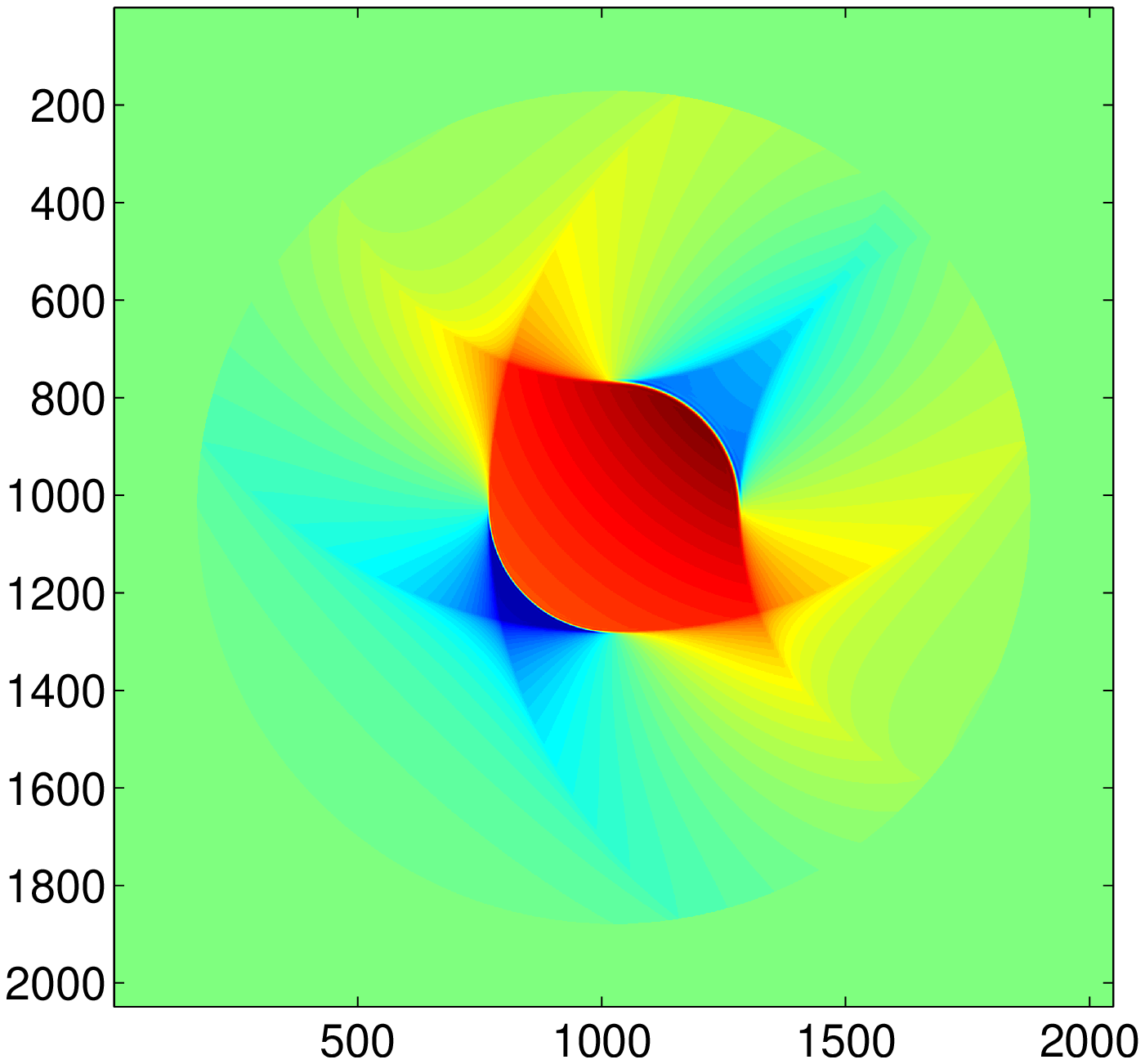}
   \label{subfig:fig2}
}
%\\
%\centering
%\hspace*{8pt}
% \subfloat[short for lof][Original phantom]{
%   \includegraphics[width=0.23\linewidth]{Figures_paper/quart-original}
%   \label{subfig:fig2}
%}
%  \hspace{27pt}
% \subfloat[short for lof][Reconstruction without smoothing]{
%%   \includegraphics[width=0.2\linewidth]{Figures_paper/quart-non}
%   \includegraphics[width=0.23\linewidth]{Figures_paper/reconstr_one_circle_grey_N2048_na8192_nr2048_p25_pi2_s0}  
%   \label{subfig:fig2}
%}
\caption[short for lof]{Reconstruction from limited view data, collected on a \emph{quarter of a unit circle}, using the standard reconstruction operator $\mT_{0}$ with no artifact reduction. The acquisition surface is illustrated by the green line in the phantom image.}
\label{fig:limited-cir}
\end{figure}
More detailed discussion will be presented in Sections~\ref{S:2D}~and ~\ref{S:3D} (see also Section~\ref{S:Num} for numerical demonstrations). The operator $\mT_0$ reconstructs all the visible singularities. However, the artifacts it generates are quite strong. We now introduce a generalization of $\mT_0$  in order to reduce the artifacts. Let us consider
\begin{equation} \label{E:mT}
 \mT f (x) := \mB \chi \mP \mR f (x).
\end{equation}
Here, $\chi \in C^\infty(\Ga)$ and $\chi = 0$ on $\mS \setminus \Ga$. Moreover, we assume further that $\chi > 0$ on the interior $\Int(\Ga)$ of $\Ga$ and $\chi$ vanishes to order $k$ on the boundary $\pdh \Ga$. Of course, $\mT=\mT_0$ if $\chi =\chi_\Ga$.

\medskip

In this article, we will study formula (\ref{E:mT}). Namely, we will analyze which singularities are reconstructed and how strong the reconstructed singularities are, compared to the original ones. Moreover, we will describe how the artifacts are generated by (\ref{E:mT}) and how strong they are.

\medskip

Let us discuss here the main idea of our analysis. We first follow the approach in \cite{AMP} to represent $\mT$ as an oscillatory integral. To that end, let $\mu$ be the Schwartz kernel of $\mT$. Then, it can be written as (see \cite{AMP}) 
\begin{equation} \label{E:mu-ori} \mu(x,y) = \frac{1}{2 \, \pi^2}\int\limits_{\Ga} \int\limits_{\rN} \int\limits_{\rN^n} e^{i \left( |y-z|^2-|x-z|^2 \right) \llg} \; |\llg|^{n-1} \,\left<z-x,\nu_z\right> \chi(z) \; d\llg \; d\sg(z).\end{equation}
By the simple change of variables $(z,\llg) \to \xi=2 (z-x) \, \llg$,  \begin{eqnarray}\label{E:mu} \nonumber
\mu(x,y) &=& \frac{1}{ 2 \, (2 \pi)^2} \intl_{\rN^n} e^{i \big(\left<x-y, \xi \right> +  \, |x-y|^2 \frac{|\xi|}{2 |x- z_+(x,\xi)| } \big)} \chi(z_+) d \xi \\ \nonumber &+&  \frac{1}{ 2 \, (2 \pi)^2} \intl_{\rN^n} e^{i \big(\left<x-y, \xi \right> - \, |x-y|^2 \frac{|\xi|}{2 |x- z_-(x,\xi)|} \big)}  \chi(z_+) \,d \xi \\ &=& \mu_+(x,y) + \mu_-(x,y) . 
\end{eqnarray}
\noindent Here, $z_\pm = z_\pm(x,\xi)$ are the intersections of $\mS$ with the positive and negative rays 
\begin{eqnarray*}
R_+(x,\xi)=\{(x+t \xi): t \geq 0\},\\
R_-(x,\xi) =\{(x+t \xi): t \leq 0\}.
\end{eqnarray*}

To illuminate the idea, let us at the moment assume that $\chi$ vanishes to {\bf infinite} order at the boundary of $\Ga \subset \mS$, then the functions $a_\pm(x,\xi) =\chi(z_\pm) $ are (smooth) symbols of order zero. Then, due to \cite[Theorem~3.2.1]{Soggeb}, $\mT$ is a pseudo-differential operator ($\Psi$DO) of order zero whose principal symbol is (see \cite{US}, and also \cite{SUSAR}, for the same result in a more general framework): $$\sg_0(x,\xi) = \frac{1}{2} \big[\chi(z_+) + \chi(z_-) \big].$$
One important consequence of $\mT$ being a $\Psi$DO is that it does not generate the artifacts into the picture. Moreover, let us denote $\ell(x,\xi) = R_+(x,\xi) \cup R_-(x,\xi)$, the line passing through $x$ along direction $\xi$. Then, the following discussions hold.
\begin{itemize}

\item[a)] Let us denote \begin{eqnarray*} \mV &=& \{(x,\xi) \in \cT^* \Og \setminus 0: \ell(x,\xi) \mbox{ intersects } \Int(\Ga) \}.\end{eqnarray*} 
Then, any singularities $(x,\xi)$ of $f$ in this zone generates a corresponding singularity of on the observed data $g:= \mR(f)|_\Ga$. It is, therefore, called the {\bf visible zone} (see, e.g., \cite{louis00local,palamodov00reconstruction,XWAK08}). Let $(x,\xi) \in \mV$, then either $\chi(z_+)>0$ or $\chi(z_-)>0$. Therefore, $\sg_0(x,\xi)>0$. That is, $\mT$ is an elliptic pseudo-differential operator of order $0$ near $(x,\xi)$. Due to the standard theory of pseudo-differential operators (see Lemma~\ref{L:Pet}), $(x,\xi) \in \wf_s(f)$ if and only if $(x,\xi) \in \wf_s(\mT f)$. That is,  all the visible singularities of $f$ are reconstructed by $\mT$ with the correct order. We notice that one visible singularity may be either visible in two directions, when both $z_+(x,\xi) $ and $z_-(x,\xi)$ belong to $\Int(\Ga)$, or in one direction, when only one of $z_+(x,\xi)$ or $z_-(x,\xi)$ belongs to $\Ga$. We will say that they are {\bf doubly} visible  and {\bf singly} visible, respectively. 

%\medskip

\item[b)] On the other hand, let us denote \begin{eqnarray*}  \mI &=& \{(x,\xi) \in \cT^* \Og \setminus 0: \ell(x,\xi) \mbox{ does not intersect } \Ga \}.\end{eqnarray*} Then, $\mI$ is called {\bf invisible zone}, since any singularity of $f$ in this zone does not generate any singularity of $g$. A singularity of $f$ in this zone is called {\bf invisible}.
We notice that for each $(x,\xi) \in \mI$, then $\chi(z_\pm)=0$. That is, the full symbol of $\mu$ is zero near $(x,\xi)$. Therefore, due to the standard theory of pseudo-differential operators, $(x,\xi) \not \in \wf(\mT f)$, even if $(x,\xi) \in \wf(f)$. That is, all the invisible singularities are completely smoothen out by $\mT$.

\end{itemize}

Let us now consider the case of our interest: $\chi$ only vanishes to a {\bf finite} order $k$ on the boundary of $\Gamma$ ($k$ can be zero as in case of $\mT_0$). Then, $\mT$ is not a $\Psi$DO anymore, since $a_\pm$ are not smooth. Therefore, $\mT$ may generate artifacts into the picture, as shown Fig.~\ref{fig:limited-cir} where $\mT=\mT_0$. 

\medskip

Our results (see Sections~\ref{S:2D} and \ref{S:3D}), show that the effect of $\mT$ on the zones $\mV$ and $\mI$ is exactly what we discuss above for the case of infinitely smooth $\chi$. This can be seen from the facts that $a_\pm(x,\xi)$ are smooth on these two zones. The artifacts, on the other hand, come from the {\bf boundary zone} where at least one of $a_+$ or $a_-$ is non-smooth:
\begin{eqnarray*} \pdh \mV = \{(x,\xi) \in \cT^* \Og \setminus 0: \ell(x,\xi) \mbox{ intersects } \pdh(\Ga)\}.\end{eqnarray*}
We will characterize how these artifacts are generated and how strong they are. To that end, we will make use of the formula (\ref{E:mu-ori}). Compared to (\ref{E:mu}), the formula \mref{E:mu-ori} still keeps track of the geometric information of $\mR$, which is helpful to understand the generation artifacts. Moreover, we will study (\ref{E:mu-ori}) as a Fourier integral operator, whose order determines the strength of artifacts (compared to the original singularity generating them).

\medskip

In term of geometric description of the artifacts, we will take advantage of the technique developed in \cite{FQ14} (see also \cite{FQ-Para}). In \cite{Artifact-sphere-flat}, the third author studied the strength of the artifacts in $\mS$ is a hyperplane (that is, $\Ga$ is flat). In this article, we study the problem for any convex smooth surface $\mS$. We will restrict ourselves to the two and three dimensional problems (i.e., $n=2,3$), since they are the most practical ones. We will follow the microlocal analysis technique in \cite{Artifact-sphere-flat}. However, due to the generality of the geometry involved, the arguments are more sophisticated. Moreover, for the three dimensional problem, we introduce a new idea of lifting up the space, which is simple but interesting. It is worth mentioning that similar problem has been studied for the X-ray (or classical Radon) transform \cite{Kat-JMAA,RamKat-AML1,RamKat-AML2,FQ13,Streak-Artifacts}. 

\medskip

The article is organized as follows. In Section~\ref{S:2D}, we state and prove the main results for the two dimensional problem. In Section~\ref{S:3D}, we state and prove the results for three dimensional problem. We then present some numerical experiments for the two dimensional problem in Section~\ref{S:Num}. Finally, we recall some essential background in microlocally analysis in the Appendix.

\section{Two dimensional problems and some numerical demonstrations} \label{S:2D} 

%The intersection $z$ of $(x,\xi) \in \pdh \mV$ with $\pdh \Ga$ is called the boundary point {\bf corresponding} to $(x,\xi)$. 

Let us consider $n=2$. We assume that $\mS$ is a smooth convex curve, and $\Ga$ is a connected piece of $\mS$. For our convenience, we assume that $\Ga$ is arc-length parametrized by the smooth function $$z=z(s):[a,b] \to \rN^2,$$ and its end points are $\ma=z(a)$ and $\mb=z(b)$. 
Let us define the function $h=h(s):[a,b] \to \rN$ by 
$$h(s) = \chi(z(s)).$$
Then, $h(s)$ vanishes to order $k$ at $s=a$ and $b$, and $h(s)>0$ for $a<s<b$. 

\medskip

%\noindent Let us denote by $\cT^* \Og$ the cotangent bundle of $\Og$. It can be considered as $\Og \times \rN^2$. We also denote
%$$\cT^* \Og \setminus 0=\{(x,\xi) \in \cT^* \Og: \xi \neq 0 \}. $$
%Roughly speaking, an element $(x,\xi) \in \cT^* \Og$ consists of a location $x \in \Og$ and a direction $\xi \in \rN^2 \setminus 0$. 
\noindent We define the following canonical relations in $(\cT^* \Og \setminus 0) \times (\cT^* \Og \setminus 0)$
$$\Delta_\mV = \{(x,\xi; x, \xi): (x,\xi) \in \overline \mV\},$$
and
$$\Llg_\ma= \{(x,\tau(x-\ma);\, y,\tau(y-\ma)) \in \cT^* \Og \times \cT^* \Og: |x-\ma| = |y-\ma|, \tau \neq 0\},$$
$$\Llg_\mb= \{(x,\tau(x-\mb);\, y,\tau(y-\mb)) \in \cT^* \Og \times \cT^* \Og: |x-\mb| = |y-\mb|, \tau \neq 0\}.$$
We notice that $(x,\xi;\, y, \eta) \in \Llg_\ma$ iff $(y,\eta)$ is in the boundary zone, corresponding to $\ma$, and $(x,\xi)$  is obtained from $(y,\eta)$ by rotating around the corresponding boundary point $\ma$. Similar description holds for $\Llg_\mb$. 

%\includegraphics{Rota-edge.jpg}

%\begin{figure} \begin{subfigure}   \centering  \includegraphics[width=.4\linewidth]{smoothing0.pdf}  \caption{1a}  \label{fig:sfig1} \end{subfigure}% \begin{subfigure}  \centering  \includegraphics[width=.8\linewidth]{smoothing1.eps}  \caption{1b}  \label{fig:sfig2}\end{subfigure} \caption{plots of....} \label{fig:fig}\end{figure}

%\noindent Let us denote by $\wf(f)$ the wave front set of $\mu$. The {\bf twisted} wave front set of $f$ is defined by
%$$\wf(\mu)' =\{(x,\xi; y, -\eta): (x,\xi; y, \eta) \in \wf(\mu) \}.$$
\begin{prop} \label{P:wavefront}
We have $$\wf(\mu)' \subset \Delta_\mV \cup \Llg_\ma \cup \Llg_\mb.$$
\end{prop}
\noindent Here, $\wf(\mu)'$ is the twisted wave front set of $\mu$:
$$\wf(\mu)' =\{(x,\xi; y, -\eta): (x,\xi; y, \eta) \in \wf(\mu) \}.$$
Proposition \ref{P:wavefront} was proved in \cite{Artifact-sphere-flat} when $\Ga$ is a line segment. The proof carries naturally to the general curve $\Ga$ without any major changes. We present it here for the sake of completeness and convenience in later discussion.

\begin{proof}[\bf Proof of Proposition~\ref{P:wavefront}] Due to the composition rule for wave front sets (see Theorem~\ref{T:compose} in Appendix), we obtain \footnote{Since $\mP$ is a pseudo-differential operator, it does not generate new wave front set elements. This well-known pseudo-locality property of a pseudo-differential operator is recalled in Appendix~\ref{A:PDO}, see (\ref{E:wave-front-PDO}).}
\begin{equation} \label{E:wfmu} \wf(\mu)' \subset \wf(\mu_\mB)' \circ \wf(\chi \mu_\mR)',\end{equation}
where $\mu_\mB$ and $\mu_\mR$ are the Schwartz kernel of $\mB$ and $\mR$, respectively. Let us now proceed to analyze the right hand side of the above inclusion.

\medskip

%For simplicity, we will identify $\Ga$ with $[a,b]$ by the mapping
%$$z(s) \to s.$$

We note that $\mR$ is an FIO with the phase function (see, e.g., \cite{louis00local,Pal-IPI,AMP})
$$\phi(z,r,x,\llg) = (|x-z|^2 - r^2)  \, \llg.$$ 

For the sake of simplicity, for an each $(z,r) = (z(s),r) \in \mS \times \rN_+$, we use the notation $(z,r,p,q)$ for $(z,r,p \, ds, q \, dr) \in \cT^* (\mS \times \rN_+)$. Due to Theorem~\ref{T:wave-FIO} (see Appendix), we obtain, by letting $\tau = 2\llg$ 
\begin{multline*} \wf(\mu_\mR) \subset \mC_\mR: = \{\big(z=z(s),r,\, \tau  \left<x-z(s),z'(s) \right>, - \tau  r ; \, x,\tau \, (x-z) \big): \\ s \in \rN, r \in \rN_+, x \in  \Og, |x-z|=r,\, \tau \neq 0\}.
\end{multline*}
%We note that $\mC_R$ is a canonical relation in $\cT^* \Og \times \cT^*(\mS \times \rN_+)$ whose left projection (into the space of $(x,\xi)$) is $\mV$. From the theory of FIO, $\mR$ translates a wave front set element $(x,\xi) \in \wf(f) \cap \mV$ into a wave front set element $(z,r,\eta) \in \cT^* (\mS \times \rN_+)$ of $\mR f$ via the canonical relation $\mC_R$. That is $(x,\xi)$ is detected by $\mR$. Therefore, $\mV$ has been termed as {\bf visible} or {\bf audible} zone.

\medskip

\noindent Also considering $\chi(z)$ as a function of $(z,r,x)$, we have
$$\wf(\chi) \subset \{\big(\ma,r, \, \tau', 0; \, x,{\bf 0}\big): \tau' \neq 0\} \cup \{\big(\mb,r, \, \tau', 0; \, x,{\bf 0}\big): \tau' \neq 0\}.$$

\medskip

\noindent Applying the product rule for wave front sets (see Theorem~\ref{T:product} in Appendix), we obtain
\begin{equation} \label{E:wf-pro} \wf(\chi \mu_\mR) \subset \mC_{\mR,\chi} \cup \mC_a \cup \mC_b,\end{equation}
where
\begin{eqnarray*}\mC_{\mR,\chi} &=& \mC_{\mR} \cap \{z \in \supp(\chi) \}, \\[4 pt] \mC_a &=& \{\big(\ma,r, \, \tau \, \left<x-\ma, z'(a) \right>+ \tau', - \tau \, r; \, x,\tau \, (x-\ma)): |x-\ma|=r,\,\tau' \neq 0\}, \\[4 pt] \mC_b &=& \{\big(\mb,r, \, \tau \, \left<x-\mb, z'(b) \right>+ \tau', - \tau \, r; \, x,\tau \, (x-\mb)): |x-\mb| =r, \, \tau' \neq 0\}.
\end{eqnarray*}

\medskip
On the other hand, we notice that $\mB$ is a FIO with the same phase function $\phi$ as $\mR$ (but the order of variables is switched), see e.g., \cite{AMP}. Therefore,
$$\wf(\mu_\mB) \subset \mC_\mR^t,$$
where $\mC^t$ is the transpose relation of $\mC$ $$\mC^t=\{(x,\xi;s,r,p,q): (s,r,p,q; x,\xi) \in \mC\}.$$

\medskip

From (\ref{E:wfmu}), we arrive to
$$\wf(\mu)  \subset (\mC^t_\mR)' \circ (\mC_{\mR,\chi} \cup \mC_\ma \cup \mC_\mb)' = (\mC^t_\mR) \circ (\mC_{\mR,\chi} \cup \mC_\ma \cup \mC_\mb) = (\mC^t_\mR \circ \mC_{\mR,\chi}) \cup (\mC^t_\mR \circ \mC_\ma) \cup (\mC^t_\mR \circ \mC_\mb).$$
We notice that
\begin{eqnarray*} \mC_\mR^t \circ \mC_{\mR,\chi} = \Delta_\mV,\quad \mC_\mR^t \circ \mC_a = \Llg_\ma, \quad  \mC_\mR^t \circ \mC_b = \Llg_\mb.\end{eqnarray*}
Therefore,
$$\wf(\mu)'  \subset \Delta_\mV \cup \Llg_\ma \cup \Llg_\mb.$$

\end{proof}

\begin{rem} \label{R:wave} Assume that $\chi(z) =0$ in a neighborhood of $z_0 \in Int(\Ga)$ and $(x,\xi) \in \cT^*\Og$ such that $\xi= \tau (x-z_0)$. Then, due to (\ref{E:wf-pro}), $$(x,\xi) \not \in \pi_R(\wf(\chi \mu_\mR)),$$ where $\pi_R$ is the right projection operator. From (\ref{E:wfmu}), we obtain:
$$(x,\xi;x,\xi) \not \in \wf(\mu).$$
This observation will be used later in the proof of Theorem~\ref{T:Main1}~a).
\end{rem}

Let us now employ Proposition \ref{P:wavefront} to describe the geometric effects of $\mT$ on the wave front set of $f$ (see also the discussion in \cite{FQ14}). We first keep in mind the following inclusion, coming from Theorem~\ref{T:cal-wave}: $$\wf(\mT f) \subset \wf(\mu)' \circ \wf(f).$$
Therefore, due to Proposition \ref{P:wavefront}, 
\begin{equation} \label{E:poss} \wf(\mT f) \subset \big[ \Delta_\mV \circ \wf(f) \big] \cup \big[\Llg_\ma \circ \wf(f)\big] \cup \big[\Llg_\mb \circ \wf(f) \big].\end{equation}
The first part on the right hand side contains all the singularities that may be possibly reconstructed by $\mT$. The other two contain all the possible artifacts generated by $\mT$. We now discuss the implications of (\ref{E:poss}) in more details.
\begin{itemize}
\item[a)] Let $(x,\xi) \in \wf(f)$ be an invisible singularity. We observe that $$\Delta_\mV \circ (x,\xi) = \emptyset, \quad \Llg_\ma \circ (x,\xi) = \emptyset, \quad \Llg_\mb \circ (x,\xi) = \emptyset.$$
 Therefore, from inclusion (\ref{E:poss}), $(x,\xi)$  is not reconstructed and does not generate any artifacts. 

%\medskip

%\noindent On the other hand, for any $(x,\xi) \in \mI$, we have 
%$$ (x,\xi) \not \in \pi_L(\Delta_\mV), \quad  (x,\xi) \not \in \pi_L(\Llg_\ma), \quad (x,\xi) \not \in \pi_L(\Llg_\mb).$$
%Due to the inclusion (\ref{E:poss}), $$(x,\xi) \not \in \wf(\mu)' \circ \wf(f).$$Therefore, $(x,\xi) \not \in \wf(\mT f)$. That is, invisible singularity is not reconstructed (totally smoothened out). 

%In summary, an invisible singularity is not reconstructed and does not generates any artifacts.

\item[b)] Let $(x,\xi) \in \wf(f)$ be a visible singularity. Then,
$$\Delta_\mV \circ (x,\xi) = (x,\xi), \quad \Llg_\ma \circ (x,\xi) = \emptyset, \quad \Llg_\mb \circ (x,\xi) = \emptyset.$$

%Due to \refprop{P:wavefront}, $$\wf(\mu)' \circ (x,\xi) \subset (x,\xi)$$ Therefore, the only singularity it can cause for $\wf(\mT f)$ is at $(x,\xi)$. 

From inclusion (\ref{E:poss}), $(x,\xi)$ may be reconstructed and does not generate any artifacts.

\item[c)] Let $(x,\xi) \in \wf(f)$ be a boundary singularity pointing through $\ma$, that is $\xi=\tau(x-\ma)$ for some $\tau \neq 0$. Then,  $$\Llg_\ma \circ (x,\xi)=\{(y,\eta): |x-\ma| = |y-\ma|,~\eta = \tau(y-\ma)\}.$$ From the inclusion (\ref{E:poss}), we observe that $(x,\xi)$ may generate artifacts $(y,\eta)$ by rotating around $\ma$. %We notice that all the generated artifacts also point through $\ma$.

Conversely, assume that $(y,\eta=\tau(y-\ma))$ be an artifact. Then, there is $(x,\xi=\tau(x-\ma)) \in \wf(f)$ that generates $(y,\eta)$. 

Similar description holds for a boundary singularity pointing through $\mb$.

\medskip

%Due to \reftheo{T:Main1}~b) and Lemma~\ref{L:Ho}, the above artifacts are at least $k$ orders smoother than the original singularities at $(x,\xi)$.  

%If we assume further that $(x,\xi)=(x,\tau(x-z_\pm)) \in \wf(f)$ is a conormal singularity of order $m$ along a curve $S$ with curvature $\kappa\neq \frac{1}{|x-z_\pm|}$. Then, the artifacts are conormal singularities of order $m-k-\frac{1}{2}$ along the circle $\{y:|y-z_\pm| = |x-z_\pm|\}$. 

%On the other hand, assume that $(x,\xi)$ is in the boundary zone (whether it is a singularity of $f$ or not). Then, the only possibility that $(x,\xi) \in \wf(\mT f)$ is that there is boundary singularity $(y,\eta)=(y,\tau(y-z_\pm)) \in \wf(f)$ such that $(x,\xi)$ is obtained from $(x,\xi)$ by a rotation around $z_+$.  

\end{itemize}

\noindent The strength of the reconstructed singularities, described in b), will be obtained by analyzing the singularities of $\mu$ near $\Delta \setminus (\Llg_\ma \cup \Llg_\mb)$. This will be done by the standard theory of pseudo-differential operators. In order to analyze the strength of the artifacts, described in c), we will make use of a class of FIOs associated to a point, introduced in Section~\ref{S:FIO-point}. 

%We notice that if $a \in S^{m}(\Og \times \Og \times \rN)$, we then $\mu_\ma \in I^{m-\frac{1}{2}}(\Llg_\ma)$ and $\mu_\mb \in I^{m-\frac{1}{2}}(\Llg_\mb)$. Let $\mT_\ma$ and $\mT_\mb$ be the linear operators whose Schwartz kernel are $\mu_\ma$ and $\mu_\mb$. Then, see \cite[Theorem 4.3.2]{hormander71fourier}, $\mT_\ma$ and $\mT_\mb$ define continuous maps from $H^{s}_{comp}(\Og) \to H^{s-m}_{loc}(\Og)$. 

\medskip

Here is our main result of this section:
\begin{theorem} \label{T:Main1} We have 
\begin{itemize}
\item[a)] Microlocally on $\Delta \setminus (\Llg_\ma \cup \Llg_\mb)$, we have $\mu \in I^0(\Delta)$. Moreover, its principal symbol is
$$\sg_0(x,\xi) = \frac{1}{2} \left[\chi(z_+)  +\chi(z_-) \right],$$ 
 where $z_\pm$ is the intersection of the ray $R_\pm(x,\xi)$ with $\mS$.
\item[b)] We can write $\mu=\mu_\ma + \mu_\mb$ such that:
\begin{itemize}
\item[i)] $\wf(\mu_\ma) \subset \Delta \cup \Llg_\ma$. Moreover, microlocally near $\Llg_\ma \setminus \Delta$, $\mu_\ma \in I^{-k-\frac{1}{2}}(\Llg_\ma)$. 
\item[ii)] $\wf(\mu_\mb) \subset \Delta \cup \Llg_\mb$. Moreover, microlocally near $\Llg_\mb \setminus \Delta$, $\mu_\mb \in I^{-k-\frac{1}{2}}(\Llg_\mb)$. 

\end{itemize}%Moreover, the principal symbols of $\mu$ on $\Llg_\ma \setminus \Delta$ and $\Llg_\mb \setminus \Delta$ are proportional to $h^{(k)}(a)$ and $h^{(k)}(b)$, respectively. 
\end{itemize}
\end{theorem}

\noindent Will present the proof of Theorem~\ref{T:Main1} in Section~\ref{S:Proof-2D}. We now describe some of its consequences.
\begin{itemize}
\item[1)] Let $(x,\xi) \in \wf(f)$ be a visible singularity. Then, due to Theorem~\ref{T:Main1}~b), microlocally near $(x,\xi;x,\xi)$, $\mu$ is a Fourier distribution of order zero with positive principal symbol. Applying Lemma~\ref{L:Pet}, we obtain $(x,\xi) \in \wf_s(f)$ if and only if $(x,\xi) \in \wf_s(\mT f)$.  That is, all the visible singularities are reconstructed with the correct order.

Moreover, the formula $\sg_0(x,\xi)$ provides the magnitude of the main part of reconstructed singularities. For example, if $(x,\xi)$ is a jump singularity across a curve $S$ with the jump equal to $1$, then $(x,\xi)$ is also a jump singularity across $S$ with the jump equal to $\sg_0(x,\xi)$. This explains the difference in the magnitude of the reconstructed singularities, that is demonstrated in Section~\ref{S:Num}.
\item[2)] Now, assume that $(x,\xi) \in \wf_s(\mT f)$ is an artifact pointing through $\ma$. Then, each of its generating singularities $(y,\eta) \in \wf(f)$ satisfies
$$(x,\xi;y,\eta) \in \Llg_\ma.$$
Then, due to Theorem~\ref{T:Main1}~b) and Lemma~\ref{L:Ho}, at least one generating singularity $(y,\eta)$ satisfies $(y,\eta) \in \wf_{s-k}(f)$. That is, all the artifacts are at least $k$ order(s) smoother then their strongest generating singularities. 

If we assume further that $(x,\xi)$ has finitely many generating singularities $(y,\eta)$, each of them is a conormal singularity of order $r$ along a curve whose order of contact with the circle $\uS(\ma,|y-\ma|=|x-\ma|)$ is exactly $1$ \footnote{We note here that the contact order is always at least $1$, since both curves are perpendicular to $\eta$ at $y$. Therefore, the condition on the contact order is quite generic.}. Then, due to Theorem~\ref{T:Main1}~b) and Lemma~\ref{L:Spread}, $(x,\xi) \in \wf(\mT f)$ is a conormal singularity of order $r+k+\frac{1}{2}$ along the circle $\uS(\ma,|x-\ma|)$. That is, the artifacts are at least $(k+\frac{1}{2})$ order(s) smoother than the strongest generating singularity. In our numerical experiments in Section~\ref{S:Num}, we will demonstrate this fact.
\end{itemize}
%We will provide some numerical demonstrations in Section~\ref{S:Num}. 

\subsection{Proof of \reftheo{T:Main1}} \label{S:Proof-2D}

\noindent Let us now discuss the proof of Theorem \ref{T:Main1}. It is similar to that of \cite[Theorem~2.2]{Artifact-sphere-flat}. However, we need to employ more sophisticated microlocal arguments due to the generality of the geometry involved. As in \cite[Theorem~2.2]{Artifact-sphere-flat}, the proofs for a) and b) require two different oscillatory integral representations for $\mu$.

\begin{proof}[\bf Proof of a).] The idea is similar to the case of infinitely smooth $\chi$ presented in \cite{AMP} (see also \cite{US} for more general framework). The main point here is to microlocalize the argument to stay away from $\Delta \cap (\Llg_\ma \cap \Llg_\mb)$. Let $(x^*,\xi^*) \in \cT^* \Og$ be such that $(x^*,\xi^*;x^*,\xi^*) \in \Delta \setminus (\Llg_\ma \cup \Llg_\mb)$. We need to prove that there is $\mu_0 \in I^{-k - \frac{1}{2}}(\Delta)$ such that 
\begin{equation} \label{E: diff} (x^*,\xi^*;x^*,\xi^*) \not \in \wf(\mu - \mu_0)\end{equation}
 and the principall symbol of $\mu_0$ at $(x^*,\xi^*)$ equals $\frac{1}{2} \big[\chi(z_+^*)+ \chi(z_-^*) \big]$, where $z_\pm^* = z_\pm(x^*,\xi^*)$.  

\medskip

Since $(x^*,\xi^*;x^*,\xi^*) \in \Delta \setminus (\Llg_\ma \cup \Llg_\mb)$, $\chi$ is smooth near both $z=z_+^*$ and $z_-^*$. Let us define a cut-off function $\mc$ such that $\mc(z) = 1$ near $z=z_\pm^*$ and zero elsewhere such that $\mc \, \chi \in C^\infty (\mS)$. 
Let us write \begin{equation} \label{with_smoothing}
 \mT := \mT_\mc + \mT',
\end{equation}
where
\begin{equation*}
 \mT_\mc := \mB_\Ga \mc \, \chi \,  \mP \mR,\quad  \mT' := \mB_\Ga \,(1-\mc)  \, \chi \mP \mR.
\end{equation*}
 
Let $\mu'$ be the Schwartz kernel of $\mT'$. Since $\chi (1-\mc) =0$ at $z=z_\pm^*$, using Remark~\ref{R:wave}, we obtain $$(x^*,\xi^*;x^*,\xi^*) \not \in \wf(\mu').$$

\medskip
 
On the other hand, similarly to  (\ref{E:mu}) (see also \cite{AMP} for the derivation), we obtain the formula for the Schwartz kernel of $\mT_\mc$
\begin{eqnarray*}\label{E:mu-22_2}
\mu_{\mc} (x,y) &=& \frac{1}{ 2 \, (2 \pi)^2} \intl_{\rN^2} e^{i \big(\left<x-y, \xi \right> +  \, |x-y|^2 \frac{|\xi|}{2 |x- z_+(x,\xi)| } \big)} \mc(z_+) \chi(z_+) \, d \xi \\&+&  \frac{1}{ 2 \, (2 \pi)^2} \intl_{\rN^2} e^{i \big(\left<x-y, \xi \right> - \, |x-y|^2 \frac{|\xi|}{2 |x- z_-(x,\xi)|} \big)} \mc(z_-) \, \chi(z_-)\, d \xi. 
\end{eqnarray*}
We notice that $a_\pm(x,\xi) :=\mc(z_\pm) \, \chi(z_\pm)$ is smooth and homogenous of degree $0$ in $\xi$. Therefore, $$\mu_\mc \in I^{0}(\Delta).$$ Moreover, the principal symbol of $\mu$ is (see, e.g., \cite[Theorem 3.2.1]{Soggeb}) $$\sg_{0,\mc}(x,\xi)= \mc(z_+) \, \chi(z_+) +\mc(z_-) \, \chi(z_-).$$  
Since $\mc(z)=1$ at $z=z_\pm^*$, we obtain 
$$\sg_{0,\mc}(x^*,\xi^*)= \chi(z_+^*) + \chi(z_-^*).$$  
This finishes the proof (\ref{E: diff}) where $\mu_0 =\mu_\mc$.

\medskip

\noindent{\bf Proof of b).} Let us decompose $\chi$ into the form $$\chi=\chi_\ma + \chi_\mb,$$ where $\chi_\ma, \chi_\mb \in C^\infty(\Ga)$ such that $\chi_\ma$ vanishes near $\mb$ and $\chi_\mb$ vanishes near $\ma$.  We then can write $\mT=\mT_\ma + \mT_\mb$ where 
\begin{equation*}
 \mT_\ma := \mB_\Ga \, \chi_\ma \,  \mP \mR,\quad  \mT_\mb := \mB_\Ga \, \chi_\mb \, \mP \mR.
\end{equation*}
Repeating the argument in Proposition~\ref{P:wavefront} (see also Remark~\ref{R:wave}), we obtain 
\begin{equation*}
\wf(\mu_\ma) \subset \Delta \cup \Llg_\ma, \quad \wf(\mu_\mb) \subset \Delta \cup \Llg_\mb.
\end{equation*}

We now prove that microlocally near $\Delta \setminus \Llg_\ma$, $\mu_\ma \in I^{-k-\frac{1}{2}}(\Llg_\ma)$. Let $(x^*,\xi^*;y^*,\eta^*) \in \Llg_\ma \setminus \Delta$. We need to prove that there is $\mu_0 \in I^{-k - \frac{1}{2}}(\Llg_\ma)$ such that 
\begin{equation} \label{E: diffb} (x^*,\xi^*;y^*,\eta^*) \not \in \wf(\mu_\ma - \mu_0). \end{equation}
Indeed, we notice that $x^*$ and $y^*$ belong to the same circle centered at $\ma$. Therefore, $$\left<x^*-y^*, z'(a) \right> \neq 0.$$ Otherwise, the tangent line of $\Ga$ at $\ma$ passes through the midpoint of the line segment connecting $x^*$ and $y^*$. This would be a contradiction to the assumption that $\Og$ is convex and $x^*, y^* \in \Og$.

\medskip

Thus, there is a neighborhood $U$ of $(x^*, y^*)$ and $\eg>0$ such that 
\begin{eqnarray*} \left<x-y, z'(s) \right> \neq 0,~\mbox{ for all } (x,y) \in U,~s \in [a,a+\eg]. \end{eqnarray*}

Let $\mc \in C^\infty(\mS)$ such that $\mc=1$ near $z=\ma$ and $\mc=0$ for $z=z(s)$ such that $|s-a|>\eg$. Let us write \begin{equation} \label{with_smoothing}
 \mT_\ma := \mT_\mc + \mT',
\end{equation}
where
\begin{equation*}
 \mT_\mc := \mB_\Ga \mc \, \chi_\ma \,  \mP \mR,\quad  \mT' := \mB_\Ga \,(1-\mc)  \, \chi_\ma \mP \mR.
\end{equation*}
We denote by $\mu_\mc$ and $\mu'$  the Schwartz kernels of $\mT_\mc$ and $\mT'$, respectively. An argument similar to the proof of Proposition~\ref{P:wavefront} shows that
$$\wf(\mu_\mc) \subset \Delta \cup \Llg_\ma, \quad \wf(\mu') \subset \Delta.$$
Therefore, $$(x^*,\xi^*;y^*,\eta^*) \not \in \wf(\mu').$$

Similarly to \mref{E:mu-ori}, we can write:
$$\mu_\mc(x,y) = \frac{1}{2 \, \pi^2}\int\limits_{\Ga} \int\limits_{\rN} \int\limits_{\rN^n} e^{i \left(|y-z|^2-|x-z|^2 \right) \llg} \; |\llg| \,\left<z-x,\nu_z\right> \mc(z) \,  \chi_\ma(z) \; d\llg \; d\sg(z).$$
The phase function of $\mu_\mc$ can be written as
$$(|x-z|^2 -|y-z|^2) \llg = \left<x-y,\, x+y-2z \right> \llg.$$
Therefore,
\begin{eqnarray*} \mu_\mc(x,y) &=&  \frac{1}{2 \, \pi^2} \intl_{\rN} \intl_{\Ga}e^{i\left<x-y,\, x+y-2z\right> \llg}\,|\llg| \, \left<z-x,\nu_z \right> \,\mc(z) \, \chi_\ma(z) \, dz \, d\,\llg\\
&=&  \frac{1}{2 \, \pi^2} \intl_{\rN} e^{i\left<x-y,\, x+y \right> \llg}\,|\llg| \, \intl_{\Ga}e^{- i\left<x-y,2z\right> \llg}\,\left< z-x,\nu_z \right> \, \mc(z) \, \chi_\ma(z) \, dz \, d\,\llg. \end{eqnarray*}
Let us denote \begin{eqnarray*} A (x,y,\llg) = \intl_{\Ga}e^{- i\left<x-y,2z\right> \llg}\,\left< z-x,\nu_z \right> \, \mc(z) \, \chi_\ma(z) \, dz.\end{eqnarray*}
Applying Lemma \ref{L:sym}, we obtain
\begin{eqnarray*} \mu_\mc(x,y) &=& \frac{1}{2 \, \pi^2}  \intl_{\rN} e^{i \big[\left<x-y,\, x+y \right>- \left<x-y, 2 \, \ma \right> \big] \llg}\,|\llg| \, \left[ e^{- i\left<x-y,\, 2 \,\ma \right> \llg}  \,A(x,y,\llg) \right] \, d\,\llg
\\&=&  \frac{1}{2 \, \pi^2}  \intl_{\rN} e^{i (|x-\ma|^2 - |y-\ma|^2) \llg} \,B(x,y,\llg)\, d\,\llg,\end{eqnarray*}
where $B(x,y,\llg)$ is a symbol of order $-(k+1)$ on $U$. We notice that $\Llg_\ma$ is the canonical relation associated to the phase function of $\mu_\mc$ (see Appendix~\ref{A:FIOs}). Therefore,  $$\mu_\mc|_{U} \in I^{-k - \frac{1}{2}}(\Llg_\ma).$$

\medskip

\noindent Let $\mu_0$ be obtained from $\mu_\mc$ by multiplying with a cut-off function near $(x^*,\xi^*;y^*,\eta^*)$. We arrive to $\mu_0 \in I^{-k-\frac{1}{2}}(\Llg_\ma)$ and $$(x^*,\eta^*;y^*,\eta^*) \not \in \wf(\mu_\ma -\mu_0).$$
We conclude that $\mu_\ma \in I^{-k-\frac{1}{2}}(\Llg_\ma)$ microlocally near $(x^*,\xi^*; y^*,\eta^*)$. 

\medskip

The proof for $\mu_\mb$ is similar.

\end{proof}

\noindent In the above proof, we have used the following result:
\begin{lemma}\label{L:sym}  Let $h \in C^\infty[a, b]$ and define
\begin{eqnarray*} A(x,y,\llg) = \intl_{a}^b e^{- i\left<x-y,\, 2z(s) \right> \llg} \, h(s) \, ds.\end{eqnarray*}
We have $A \in C^\infty(\Og \times \Og \times \rN)$. Moreover, 

\begin{itemize}
\item[a)] Assume that $h(s) =0$ near $s=b$ and $h$ vanishes to order $k$ at $\tau = a$ . Then,  on the set $$\{(x,y) \in \Og \times \Og: \left<x- y, z'(s) \right>  \neq 0,~ \mbox{ for all } s \in [a,b] \mbox{ such that } z(s) \in \supp(h)  \},$$
we have
\begin{eqnarray*}  e^{- i \, \left<x-y,\, 2\, \ma \right> \llg} \,  A(x,y,\llg) &\sim& \frac{h^{(k)}(a)}{[2\, i \,\left<x-y, z'(a) \right> \, \llg]^{(k +1)}} \, \big[1+ r(x,y,\llg)\big]. 
\end{eqnarray*}
Here, $\nu_\ma=\nu(a)$ is the normal vector of $\mS$ at $\ma=z(a)$.
\item[b)] Assume that $h(s) =0$ near $s=a$, and $h$ vanishes at $\tau = b$ to order $k$. On the set $$\{(x,y,\llg) \in \Og \times \Og \times \rN: \left<x- y, z'(\tau) \right>  \neq 0,~ \mbox{ for all } \tau \in [a,b],~ z(s) \in \supp(h)  \},$$ we have
\begin{eqnarray*} e^{- i \, \left<x-y,\, 2 \, \mb \right> \llg} \, A(x,y,\llg) &\sim& \frac{-h^{(k)}(b)}{[2\, i \,\left<x-y, z'(b) \right> \, \llg]^{(k +1)}} \, \big[1+ r(x,y,\llg)\big].
\end{eqnarray*}
Here, $\nu_\mb=\nu(b)$ is the normal vector of $\mS$ at $\mb=z(b)$.
\end{itemize}
In both a) and b), $r(x,y,\llg)$ is a symbol of order at most $-1$. 
\end{lemma}
\begin{proof} The Lemma is proved by successive integration by parts. It is very similar to \cite[Lemma 2.3]{Artifact-sphere-flat}. We skip it here for the sake of brevity.
\end{proof}

%%%%%%%%%%%%%%%%%%%%%%%%%%%%%%%%%%%%%%%%%%%%%%%%%%%
%
%%%%%%%%%%%%%%%%%%%%%%%%%%%%%%%%%%%%%%%%%%%%%%%%%%%

\section{Three dimensional problem}\label{S:3D}
Let us now consider $n=3$. We assume that $\Og \subset \rN^3$ is a convex domain with the smooth boundary $\mS$. We assume that $\Ga$ is a connected and simply connected subset of $\mS$ with the smooth boundary $\ga$. We will analyze $\mT$ when $\chi$ vanishes to a finite order $k$ on $\ga$. With a slight abuse of notation, we arc-length parametrize $\ga$ by the function $\ga:\rN \to \rN^3$ (with $\ga(s+L) = \ga(s)$, where $L$ is the length of $\ga$). 

\medskip

\noindent Similarly to the case $n=2$, we define
$$\Delta_\mV = \{(x,\xi; x, \xi): (x,\xi) \in \mV\}.$$

\medskip

\noindent We denote by $\Llg$ the following canonical relation in $(\cT^* \Og \setminus 0)\times (\cT^* \Og \setminus 0)$:
\begin{multline*}\Llg = \{(x,\tau \, (x-\ga(s)); \, y,\tau \, (y-\ga(s))): |x-\ga(s)| = |y-\ga(s)|,\\ \left<x-\ga(s),\ga'(s) \right> = \left<y-\ga(s),\ga'(s) \right>, \mbox{ for some } s, \tau \in \rN \mbox{ such that } \tau \neq0 \}.\end{multline*}

\noindent We notice that $(x,\xi; y, \eta) \in \Llg$ if and only if $(x,\xi)$ and $(y,\eta)$ are boundary elements corresponding to a common boundary point $z \in \ga$ and they are obtained from the other by a rotation around the tangent line of $\ga$ at $z$. 

\medskip
\noindent The following result gives us a geometric description of the singularities of the Schwartz kernel $\mu$ of $\mT$.
\begin{prop} \label{P:wavefront-3} We have
$$\wf(\mu)' \subset \Delta_\mV \, \cup \, \Llg.$$
\end{prop}
\noindent The proof of Proposition \ref{P:wavefront-3} is similar to that of Proposition~\ref{P:wavefront} (see also \cite[Proposition 3.1]{Artifact-sphere-flat} and \cite{FQ14}). We skip it for the sake of brevity. Similarly to Proposition~\ref{P:wavefront} we obtain the following implications of Proposition \ref{P:wavefront-3}:
\begin{itemize}
\item[a)] $\mT$ smoothens out all the visible singularities.
\item[b)] $\mT$ may reconstruct the visible singularities, and
\item[c)] $\mT$ may generate artifacts by rotating a boundary singularity around the tangent line of $\ga$ at the corresponding boundary point.
\end{itemize}

The following result tells us the strength of the reconstructed singularities, explained in b), and artifacts, described in c):
\begin{theorem}  \label{T:Main2} The following statements hold:
\begin{itemize}
\item[a)]Microlocally on $\Delta \setminus \Llg$, we have $\mu \in I^0(\Delta)$ with the principal symbol 
$$\sg_0(x,\xi) =\frac{1}{2} \left[ \chi(z_+) + \chi(z_-)\right],$$ where $z_\pm$ is the intersection of the ray $R_\pm(x,\xi)$ with $\mS$.
\item[b)] Microlocally on $\Llg \setminus \Delta$, we have $\mu \in I^{-k-\frac{1}{2}}(\Llg)$. 
\end{itemize}
\end{theorem}

\medskip

\noindent Similarly to Theorem~\ref{T:Main1}, we obtain the following implications of Theorem \ref{T:Main2} (see also Lemma~\ref{L:Ho3D} for the discussion on artifacts):
\begin{itemize}
\item[a)] If $(x,\xi) \in \wf(f)$ is a visible singularity, then $(x,\xi)$ is reconstructed with correct order. %Moreover, it does not generate artifacts.

\item[b)] The artifacts are at least $k$ order(s) smoother than the strongest generating singularity.
\end{itemize}

Let us now proceed to prove Theorem~\ref{T:Main2}. We will need the following two lemmas:
%\begin{lemma} \label{L:lift}
%Let us define the operator $\mL: \mE'(\Og) \to \mE'(\Og \times \ga)$ by $$\mL(f)(x,z_\ga) = f(x),\quad (x,z_\ga ) \in \Og \times \ga.$$
%Then $\mL$ defines a continuous map $H^{s}_{comp}(\Og) \to H^{s+\frac{1}{2}}_{comp}(\Og \times \ga)$ with the canonical relation
%$$\Llg_\mL=\{(x,z_\ga,\xi,0; x,\xi): (x,\xi) \in \cT^* \Og \setminus 0\}.$$
%\end{lemma}
%The proof of this lemma is elementary. We skip it for the sake of simplicity.
\medskip

\begin{lemma} \label{L:F} Let $\mF: \mE'(\Og) \to \mD'(\Og)$ be defined by
$$\mF(f)(x) = \intl_{\Og} \intl_{\ga} e^{i \, (|x-z_\ga|^2 - |y-z_\ga|^2) \llg} a(x,y,z_\ga,\llg) \,f(y) \, dz_\ga \, dy \, d\llg,$$
where $a \in S^{m}((\Og \times \Og \times \ga) \times \rN)$. Then, $\mF \in I^{m-\frac{3}{2}}(\Llg)$.
\end{lemma}
We first notice that the above formula $\mF(f)$ does not directly define an FIO, since the ``phase" function $\phi = (|x-z_\ga|^2 - |y-z_\ga|^2) \llg$ involves an extra variable $z_\ga \in \ga$, which is neither a variable of $f$ nor a phase variable. In \cite{Artifact-sphere-flat}, where $\ga$ is a line segment, the above result was proved by a change of variables. When $\ga$ is a general curve, such change of variable seems to be complicated. We, instead, introduce a simple idea of lifting up the space. %This idea is also used in \cite{}, where the integral transform associated with thermo-acoustic tomography with line detectors is considered.

\begin{proof} For the notational ease, let us denote $X = \Og$ and $Y = \Og \times \ga$. Then, $X$ and $Y$ are smooth manifolds of dimensions $n_X=3$ and $n_Y=4$, respectively.

\medskip

Let us define the operator $\mL: \mE'(\Og) \to \mE'(\Og \times \ga)$ by  \footnote{This operator is just to lift up the space by one dimension.}
$$\mL(f)(x,z_\ga) = f(x),\quad (x,z_\ga ) \in \Og \times \ga.$$
Then, $\mL$ can be written in the following form
$$\mL(f)(x,z_\ga) =\frac{1}{(2 \pi)^3} \intl_{\rN^3} \intl_{\rN^3} e^{(x-y) \cdot \xi } \, f(y) \, dy \, d \xi, \quad (x,z_\ga ) \in \Og \times \ga.$$
That is, $\mL$ is an FIO of order  (see (\ref{E:Ord}) in Appendix~\ref{A:FIOs})\footnote{Here, $n_\xi=3$ is the dimension of the phase variable $\xi$.}
$$m_\mL = (2 n_\xi - n_X - n_Y)/4=-\frac{1}{4},$$ 
with the canonical relation
$$\Llg_\mL=\{(x,z_\ga,\xi,0; x,\xi): (x,\xi) \in \cT^* \Og \setminus 0\}.$$

Let us define $\mF_0: \mE'(Y) \to \mD'(X)$ by the formula $$\mF_0(g)(x) = \intl_{\Og} \intl_{\ga} e^{i \, (|x-z_\ga|^2 - |y-z_\ga|^2) \llg} a(x,y,z_\ga,\llg) \,g(y,z_\ga) \, dz_\ga \, dy \, d\llg.$$
Then, $\mF_0$ is an FIO of order (see (\ref{E:Ord}) in Appendix~\ref{A:FIOs}) \footnote{Here, $n_\llg=1$ is the dimension of the phase variable $\llg$.}$$m_0= m+(2 n_\llg - n_X-n_Y)/4= m-\frac{5}{4},$$ with the canonical relation
\begin{multline*}\Llg_0 = \{(x,\tau \, (x-z_\ga); \, y,z_\ga, \tau \, (y-z_\ga), \tau \left<y-x,\dot \ga \right>): |x-\ga(s)| = |y-\ga(s)|, \mbox{ for some } \tau \neq0 \}.\end{multline*}
%It is easy to observe that the projection $\pi_L: \Llg_0 \to \cT^* X$ and $\pi_R: \Llg_0 \to \cT^* Y$ has the rank $2+\dim X$ and $2+ \dim Y$, respectively. Applying \cite[Theorem~4.3.2]{hormander71fourier}, we have $\mF_0$ maps continuously from $H^{s}_{comp} (Y)$ to $H_{loc}^{s-r+(2-n_X-n_Y)/4}(X)= H_{loc}^{s-m- \frac{1}{2}}(X)$. 

\medskip

We observe that $\mF = \mF_0 \circ \mL$, and $\Llg = \Llg_0 \circ \Llg_\mL$. Therefore, see \cite[pp. 96]{hormander71fourier}, $$\mF \in I^{m_\mL + m_0}(\Llg)= I^{m-\frac{3}{2}}(\Llg).$$ 

\end{proof}

\medskip

\begin{proof}[\bf Proof of Theorem~\ref{T:Main2}] The  proof of a) is similar to that of Theorem~\ref{T:Main1}~a). We skip it for the sake of brevity. We now proceed to prove b). 

%We will start our analysis with the following formula (see \cite{AMP})
%\begin{equation} \label{E:mu-3} \mu(x,y) = \frac{1}{2 \, \pi^3}\intl_{\rN} \intl_{\Ga}e^{i(|x-z|^2-|y-z|^2) \, \llg}\,\llg^2 \,\left<z-x,\nu_z \right> \, \mc(z) \, \chi(z) \, dz \, d\,\llg.\end{equation}

\medskip

Let $(x^*,\xi^*; y^*, \eta^*) \in \Llg \setminus \Delta$. That is $x^* \neq y^*$ and there is $z^*=\ga(t^*) \in \ga$ such that 
\begin{equation} \label{E:xy} |x^*-z^*| = |y^*-z^*|,\quad \left<x^*-y^*,\tau^*_1 \right> = 0. \end{equation}
Here, $\tau^*_1=\ga'(t^*)$ is the unit tangent vector of $\ga$ at $z^*$. Let $\tau^*_2$ be the unit normal vector of $\ga$ which is tangent to $\mS$ and points inward to $\Ga$. 
Let $\bd$ be the metric on $\mS$ and $\mO \subset \mS$ be a small (bounded) neighborhood of $z^*$ such that for each $z \in \mO$ there exists uniquely $z_\ga \in \ga$ such that $$\bd(z,z_\ga) = \min\{d(z,z'): z' \in \ga\}.$$
That is, each $z \in \mO$ can be unique parametrized by $(z_\ga, \delta=\bd(z_\ga,z))$. By narrowing down $\mO$, if necessary, we may assume $$z  \to (z_\ga,\delta)$$ defines a smooth map from $\mO \cap \Ga$ to $\ga \times [0,\delta_0]$, for some $\delta_0>0$, whose Jacobian $|J(z)|$ is bounded from below. 
 
\medskip

\noindent Let $\mc \in C^\infty_0(\mO)$ such that $\mc(z) =1$ near $z=z^*$. Let us write \begin{equation} \label{with_smoothing}
 \mT := \mT_\mc + \mT',
\end{equation}
where
\begin{equation*}
 \mT_\mc := \mB_\Ga \mc \, \chi \,  \mP \mR,\quad  \mT' := \mB_\Ga \,(1-\mc)  \, \chi \mP \mR.
\end{equation*}
Let us denote by $\mu_\mc, \mu'$  the Schwartz kernels of $\mT_\mc$ and $\mT'$, respectively. Then $\mu = \mu_\mc + \mu'$.

\medskip

An argument similar to the proof of Proposition~\ref{P:wavefront-3} (see also proof of Proposition~\ref{P:wavefront}) shows that
$$(x^*,\xi^*; y^*, \eta^*) \not \in  \wf(\mu') .$$

It now remains to analyze $\mu_\mc$. It can be written in the form:
\begin{eqnarray}\label{E:mu3-bis} \mu_\mc(x,y) =  \frac{1}{2 \, \pi^3}\intl_{\rN} \intl_{\Ga}e^{i(|x-z|^2-|y-z|^2) \, \llg}\,\llg^2 \,\left<z-x,\nu_z \right> \,\mc(z) \, \chi(z)  dz \, d\,\llg.\end{eqnarray}
By changing the variable $z \in \mO \cap \Ga \to (z_\ga,s)$, we obtain
\begin{eqnarray}\label{E:mu3-bis} \mu_\mc(x,y) =  \frac{1}{2 \, \pi^3}\intl_{\rN} \intl_{\ga} \intl_{0}^\delta e^{i(|x-z(z_\ga,\delta)|^2-|y-z(z_\ga,\delta)|^2) \, \llg}\, h(z_\ga,\delta) \, d\delta \,dz_\ga \,\llg^2  \, d\,\llg. \end{eqnarray}
Here, 
\begin{eqnarray*}h(z_\ga,\delta) =  \left<z-x,\nu_z \right> \, \chi(z) \,\mc(z) |J(z)|^{-1}\end{eqnarray*}
satisfies $h(z_\ga,\delta)=0$ for $s \geq \delta$ and $h(z_\ga,\delta)$ vanishes to order $k$ at $s=0$. Moreover, notice that by choosing $\mO$ small enough, we can assume that $\delta$ as small as we wish. Let us show that there is a neighborhood $\mQ$ of $(x^*,y^*) \in \Og \times \Og$ such that 
\begin{multline} \label{C:Claim1} \intl_{0}^\delta e^{i(|x-z(z_\ga,\delta)|^2-|y-z(z_\ga,\delta)|^2) \, \llg}\, h(z_\ga,\delta) \, d\delta = e^{i(|x-z_\ga|^2-|y-z_\ga|^2) \, \llg}\, a(x,y,z_\ga,\llg),\\  \mbox{ where }a(x,y,z_\ga, \llg) \in S^{-k}(\mQ \times \ga \times \rN).\end{multline} Indeed, consider the phase function of the left hand side $$\phi(x,y,z_\ga,\delta,\llg) = (|x-z(z_\ga,\delta)|^2-|y-z(z_\ga,\delta)|^2) \, \llg.$$ Taking the derivative with respect to $s$, we obtain
$$\phi_s(x,y,z_\ga,\delta,\llg) = 2 \left< x-y, z_s(z_\ga,\delta)\right> \, \llg.$$
Therefore, 
$$\phi_s(x,y,z_\ga,0,\llg) = 2 \left< x-y, \tau_2(z_\ga)\right> \, \llg,$$ where $\tau_2(z_\ga)$ is the unit vector tangent to $\mS$ and normal to $\ga$ (pointing inward to $\Ga$). From \eqref{E:xy} and the fact that $\Og$ is convex, we easily observe that 
\begin{equation} \label{E:nozero} \left<x^*-y^*, \tau_2(z^*) \right> \neq 0. \end{equation}
Therefore, by choosing $\delta>0$ small enough, we may assume that $$\phi_s(x,y,z,0,\llg)  \neq 0 \mbox{ for all } (x,y,z_\ga,\delta,\llg) \in \mQ \times \ga \times [0,\delta] \times (\rN\setminus 0),$$ where $\mQ$ is a small neighborhood of $(x,y) \in \Og \times \Og$. Taking integration by parts, we obtain for all $(x,y) \in \mQ$:
\begin{multline*} \intl_{0}^\delta e^{i(|x-z(z_\ga,\delta)|^2-|y-z(z_\ga,\delta)|^2) \, \llg}\, h(z_\ga,\delta) \, d\delta = e^{i(|x-z_\ga|^2-|y-z_\ga|^2) \, \llg}\, \frac{h(z_\ga,0)}{i \phi_s(x,y,z_\ga,0,\llg)} \\ + \intl_{0}^\delta e^{i(|x-z(z_\ga,\delta)|^2-|y-z(z_\ga,\delta)|^2) \, \llg}\, \pdh_s \Big(\frac{h(z_\ga,\delta)}{i \phi_s(x,y,z_\ga,\delta,\llg)}\Big) \, d\delta.
\end{multline*}
Continuing the integration by parts 
\begin{multline*} \intl_{0}^\delta e^{i(|x-z(z_\ga,\delta)|^2-|y-z(z_\ga,\delta)|^2) \, \llg}\, h(z_\ga,\delta) \, d\delta = e^{i(|x-z_\ga|^2-|y-z_\ga|^2) \, \llg}\,\sum_{l=0}^{k} H_{l}(x,y,z_\ga,0,\llg ) \\ + \intl_{0}^\delta e^{i(|x-z(z_\ga,\delta)|^2-|y-z(z_\ga,\delta)|^2) \, \llg}\, \pdh_s H_{k+1}(x,y,z_\ga,\delta) \, d\delta.
\end{multline*}
Here, $H_{l}$ is homogeneous of degree $-l-1$ with respect to $\llg$. That is, 
\begin{eqnarray*} \intl_{0}^\delta e^{i(|x-z(z_\ga,\delta)|^2-|y-z(z_\ga,\delta)|^2) \, \llg}\, h(z_\ga,\delta) \, d\delta = e^{i(|x-z_\ga|^2-|y-z_\ga|^2) \, \llg} a(x,y,z_\ga,\delta,\llg).
\end{eqnarray*}
Here,
\begin{eqnarray*} a(x,y,z_\ga,\llg) = \sum_{l=0}^{k} H_l(x,y,z_\ga,0,\llg )+ R_{k}(x,y,z_\ga,\llg) \end{eqnarray*}
and $$R_{k}(x,y,z_\ga,\llg) =  \intl_{0}^\delta e^{i[(|x-z(z_\ga,\delta)|^2-|y-z(z_\ga,\delta)|^2) - (|x-z_\ga|^2-|y-z_\ga|^2)] \, \llg}\, \pdh_s H_{k+1}(x,y,z_\ga,\delta) \, d\delta.$$
Using standard integration by parts as above, one can obtain $$R_{k} \in S^{-k-1} ((\mQ \times \ga) \times \rN).$$
Moreover, from the definition of $H_{l}$ we get $H_{l} = 0$ for all $0 \leq l \leq k-1$. Therefore, 
\begin{eqnarray*} a(x,y,z_\ga,\llg) = H_k(x,y,z_\ga,0,\llg) + R_{k}(x,y,z_\ga,\llg).\end{eqnarray*}
This finishes the proof of (\ref{C:Claim1}).

\medskip

Now let us write:
\begin{eqnarray*}\mu_\mc(x,y) =  \frac{1}{2 \, \pi^3}\intl_{\rN} \intl_{\ga} e^{i(|x-z_\ga|^2-|y-z_\ga|^2) \, \llg}\, a(x,y,z_\ga,\llg) \,\llg^2 \,dz_\ga  \, d\,\llg.\end{eqnarray*}
Applying Lemma~\ref{L:F}, we obtain $\mu_\mc|_{\mQ} \in I^{-k - \frac{1}{2}}(\Llg)$. 
\end{proof}

%\noindent In the above proof, we have used the following result:
%\begin{lemma} \label{L:asym} Assume that $h \in C_0^\infty(\ga \times [0,\infty))$ such that $h$ vanishes to order $k$ at $s=0$. Then, there is a neighborhood $\mQ$ of $(x^*,y^*) \in \Og \times \Og$ such that 
%$$\intl_{0}^\delta e^{i(|x-z(z_\ga,s)|^2-|y-z(z_\ga,s)|^2) \, \llg}\, h(z_\ga,s) \, ds = e^{i(|x-z_\ga|^2-|y-z_\ga|^2) \, \llg}\, a(x,y,z_\ga,\llg),$$
%where $a(x,y,z_\ga,\llg) \in S^{-k-1}((\mQ \times \ga) \times \rN)$.
%\end{lemma}
%\begin{proof} 
%\end{proof}

\section{Numerical demonstrations} \label{S:Num}

In this section we investigate reconstructions from limited view circular mean data in a series of numerical experiments. Here, we focus on the two dimensional problem for illustration reasons. We will show that our theoretical results from Section~\ref{S:2D} directly translate into practice and, in particular, that they can be used to significantly improve the reconstruction quality.
%To illuminate the results we present here some simple numerical demonstrations in two dimensions. The three dimensional case is more complicated and will be presented elsewhere. 

\medskip

%In our experiments, $\mS$ is the unit circle 
%$$\mS =\{z(s) = (\cos s, \sin s): 0 \leq s \leq 2 \pi\}. $$
%The original phantom consists of a disc (marked in red) of radius $0.3$ centered at the origin. We consider two cases of $\Ga$:

\paragraph{Experimental setup.}  In all of our experiments we consider the phantom consisting of a disc with radius $0.3$ centered at the origin, see Fig.~\ref{fig:fig1} \subref{subfig:OP}, where the image size of the phantom was chosen to be $2048 \times 2048$.  In our experiments, we will numerically illustrate the visible and invisible singularities, as well as boundary singularities for different angular range. We will also numerically investigate the difference between singly and doubly visible singularities. It is therefore useful to keep in mind that all the singularities of the original phantom are located on the circle of radius $0.3$ centered at the origin, and the directions of all singularities are given by normal (orthogonal) directions to the circle at the location of the singularity. 

In what follows, we assume that the limited view data of this phantom are collected on a circular arc of the form
$$
\Ga_{b}=\{z(s) = (\cos s, \sin s): 0 \leq s \leq b\}.
$$
In all of our experiments we computed the circular means of the phantom analytically and, to obtain limited view data, we sampled the data on $\Ga_\frac{\pi}{2}$ (first experimetn) and $\Ga_\frac{3\pi}{2}$ (second experiment).  In each experiment, we chose the number of angular samples $n_{a}$ as well as the number of radial samples $n_{r}$ independently of the angular range as $n_a= n_r=2048$. Given this data we then implemented and applied the reconstruction formula $\mT_0 f = \mB \chi_{\Ga_{b}} \mP \mR f$ in Matlab (cf. \eqref{no_smoothing}) where $\chi_{\Ga_{b}}$ is the characteristic function of  $\Ga_{b}$. In this situation, no artifact reduction was performed. To incorporate artifact reduction into the reconstruction formula, we also implemented the modified reconstruction formula $\mT f = \mB \chi \mP \mR f$ where $\chi$ was constructed to be smooth in the interior $\Int(\Ga_{b})$ of $\Ga_{b}$ and, at the same time, to vanish to some order $k>0$ at the end points of $\Ga_{b}$ (cf.  \eqref{E:mT} and subsequent discussion). More precisely, in our experiments we consider the following construction. Let
\begin{equation}
\label{eq:smooth cutoff}
	H(s) = \frac{s(b -s)}{s(b -s) + \epsilon}, \;\text{ and }\; h(s) = \frac{H(s)}{H(\frac{b}{2})}.
\end{equation}
Then, $h \in C^\infty((0, b))$ and 
\begin{itemize}
	\item[i)] $0<h(s)\leq 1$ for $0<s< b$, $h(\frac{b}{2}) =1$, 
	\item[ii)] $h$ vanishes to order $1$ at  $s=0,b$. 
\end{itemize}
Let $\chi$ be defined as $\chi(z(s)) = h(s)$, then $\mT$ is smoothing of order one. The parameter $\epsilon>0$  controls how close the function $\chi$ is to the constant function $1$. The closer $\epsilon$ is to $0$, the closer $h$ is  to the constant function $1$ on $(0,b)$, see Fig.~\ref{fig:1q-simple_smoothing_eps_var_power_var}~\subref{subfig:h-eps}. That is, the smaller $\epsilon$, the closer is the function $\chi$ to the constant function $1$ on $\Ga$. We also consider smoothing of orders $2$ and $3$. The corresponding smoothing functions are defined through $h^2(s)$ and $h^3(s)$, respectively. In general, we set 
\begin{equation}
\label{eq:smooth cutoff chi}
	\chi^{k}(z(s)) = h^{k}(s),\quad k=1,2,3,\dots\;.
\end{equation}
The graphs of the functions $h^{k}$ for $k=1,2,3$ with $\epsilon=0.2$ are shown in Fig.~\ref{fig:1q-simple_smoothing_eps_var_power_var}~\subref{subfig:h-order}. 

\begin{figure}[ht]
\centering
 \subfloat[short for lof][Smoothing function $h$ for different values of $\epsilon$]{
   \includegraphics[width=0.33\linewidth]{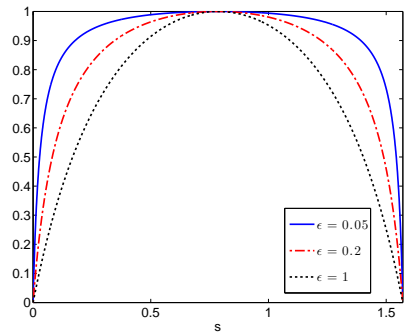}   
   \label{subfig:h-eps}
 }
  \hspace{50pt}
 \subfloat[short for lof][Smoothing functions $h^{k}$ of orders $k=1,2,3$]{
   \includegraphics[width=0.33\linewidth]{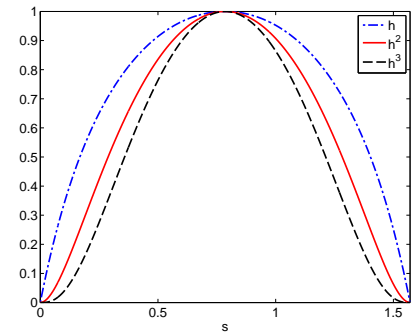}
   \label{subfig:h-order}
}

\caption[short for lof]{Smoothing functions $h^{k}$ for different values of $\epsilon$ and for different smoothing orders at the ends of the angular range, cf. \eqref{eq:smooth cutoff}-\eqref{eq:smooth cutoff chi}. } %{Simple smoothing function $h(s)$  with $\epsilon$ varying from $0.05$ to $1$ and various powers of $h$.}
\label{fig:1q-simple_smoothing_eps_var_power_var}
\end{figure}

%The reconstruction using these smoothing functions are shown in Fig.~\ref{fig:1q-order_varies}. As expected, the artifacts get weaker as the order increases.

%We then pick a smoothing function $\chi$ that vanishes to order $1$ at the endpoints of $\Ga$. Ideally, we would like to have $\chi$ as close to a constant function on $\Ga$ as possible. Then, most of the visible singularities are reconstructed up to (almost) a uniform ratio. It helps to compare the relative strength (magnitude) of the visible singularities. Such a choice, however, will introduce quick changes of $\chi$ at the endpoints. In the discretization, it may behave like a jump (discontinuity). We, therefore, have to balance the choice of $\chi$. Let

\paragraph{Experiment 1:} In our first experiment we assume the data are collected on the acquisition surface $\Ga_\frac{\pi}{2}$ (quarter of the unit circle) and consider reconstructions without as well as with artifact reduction (using the smoothing function \eqref{eq:smooth cutoff} with parameter $b=\frac{\pi}{2}$). We  numerically investigate the effect of changing the reconstruction parameters $\epsilon$ as well as the effect of changing the smoothing order $k$ at the ends of the angular range. The results of this experiment are presented in Fig.~\ref{fig:fig1}-\ref{fig:1q-order_varies}. 

%All the singularities of the phantom are located on the circle of radius $0.3$ centered at the origin. Their direction is orthogonal to the circle. 
Before examining the reconstructions, let us first note that the visible singularities at the considered angular range $[0,\frac{\pi}{2}]$ are located on two open quarters of the circle: first (upper right) and third (lower left). Accordingly, the invisible singularities are located on the other two quarters of the circle, whereas the boundary singularities correspond to the boundary of the acquisition surface $\Ga_\frac{\pi}{2}$ and, therefore, are located at the four points $(0,\pm 0.3)$, $(\pm 0.3,0)$. These characterizations are due to the analysis and explanations that we presented in Section~\ref{S:2D}. To numerically verify those theoretical findings, let us now examine the reconstruction in Fig.~\ref{fig:fig1}. Our observations are as follows:

%The original phantom consists of a disc (marked in red) of radius $0.3$ centered at the origin -- see Fig.~\ref{fig:fig1} \subref{subfig:OP} (and Fig.~\ref{fig:fig1} \subref{subfig:OPBW} for the black \& white image). The limited data are collected on the first quarter of the unit circle:
%$$
%\Ga=\{z(s) = (\cos s, \sin s): 0 \leq s \leq \frac{\pi}{2} \}.
%$$

%\medskip
%
%\noindent We then apply the reconstruction formula (\ref{E:mT}) with $\mT= \mT_0$  for limited data problem (i.e., $\chi=\chi_\Ga$ and the reconstruction is without smoothing). The phantom (original image) is of the size of $2048 \times 2048$. The unit circle (full observation surface) has the resolution of $n_a= 2048$ and the radial resolution is $n_r=2048$. The reconstruction is presented in Fig.~\ref{fig:fig1}~\subref{subfig:RCNS} (see also Fig.~\ref{fig:fig1}~\subref{subfig:RCBW} for the black \& white image). 
\begin{figure}[t]
\centering
 \subfloat[short for lof][Original phantom]{
   \includegraphics[width=0.3\linewidth]{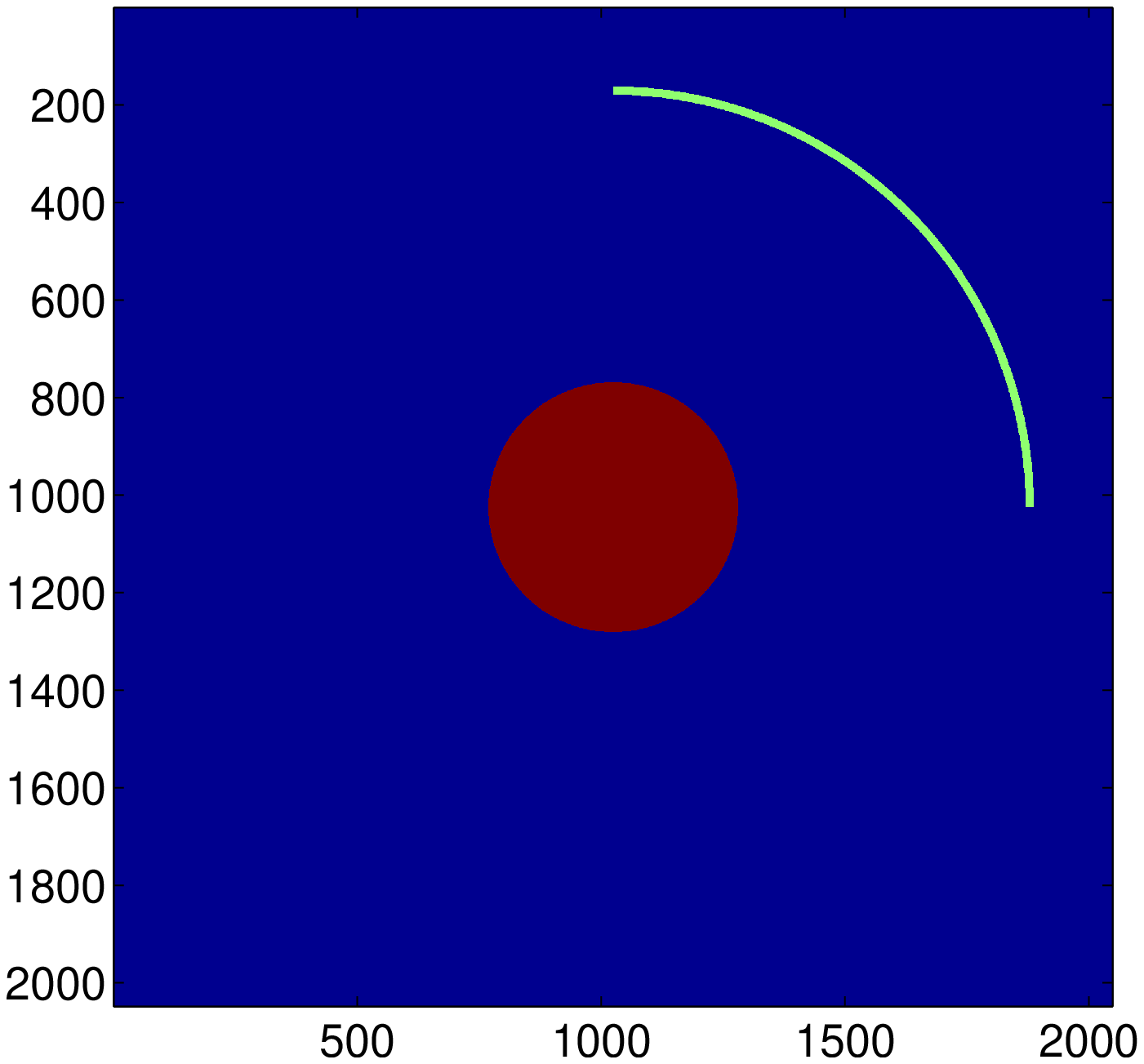}
   \label{subfig:OP}
 }
  \hspace{20pt}
 \subfloat[short for lof][Reconstruction]{
   \includegraphics[width=0.3\linewidth]{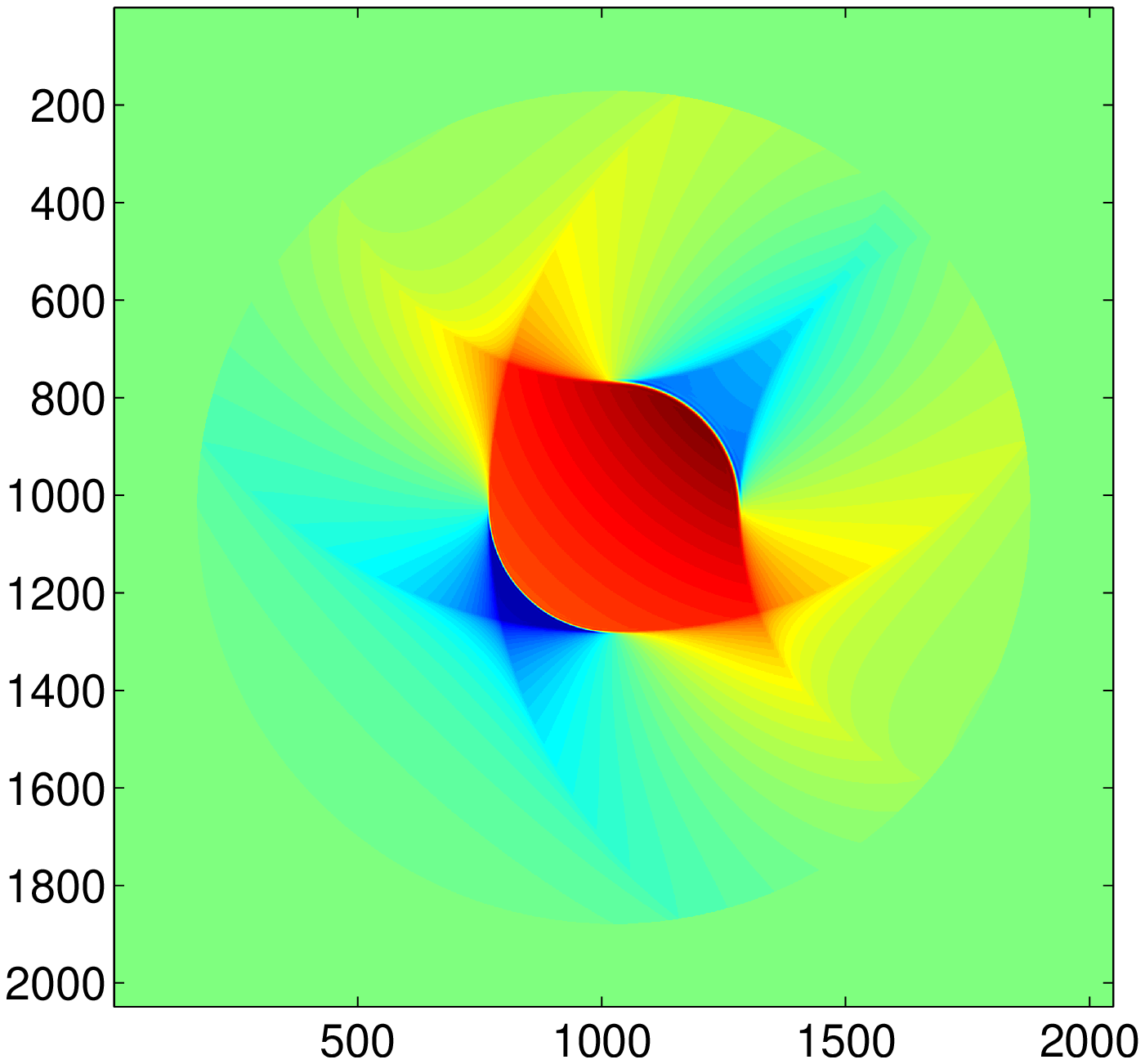}
   \label{subfig:RCNS}
}%\\
%\centering
%\hspace*{8pt}
% \subfloat[short for lof][Original phantom]{
%   \includegraphics[width=0.23\linewidth]{Figures_paper/quart-original}
%   \label{subfig:OPBW}
%}
%  \hspace{27pt}
% \subfloat[short for lof][Reconstruction without smoothing]{
%%   \includegraphics[width=0.2\linewidth]{Figures_paper/quart-non}
%   \includegraphics[width=0.23\linewidth]{Figures_paper/reconstr_one_circle_grey_N2048_na8192_nr2048_p25_pi2_s0}  
%   \label{subfig:RCBW}
%}
\caption[short for lof]{Reconstruction from limited view data, collected on $\Ga_{\frac{\pi}{2}}$ (quarter of the unit circle), using the standard reconstruction operator $\mT_{0}$ with no artifact reduction. The acquisition surface $\Ga_{\frac{\pi}{2}}$ is illustrated by the green line in the phantom image \subref{subfig:OP}. }
\label{fig:fig1}
\end{figure}

\begin{itemize}
\item[a)] All visible singularities are reconstructed sharply. They visually appear to be of the same order as the original singularities (jump from red to blue). %Moreover, the jump are basically half

\item[b)] The invisible singularities are smoothed away and, hence, not present in the reconstruction. This can be seen from the fact that there are no sharp boundaries (intensity jumps) along invisible directions.

\item[c)] Added singularities (artifacts) are generated along four circles, each of them touches the disc (phantom) tangentially. Moreover, we observe that two circles are concentric and centered on the x-axis, and the other two are concentric and centered on the y-axis. More precisely, the added artifacts are located on circles that are centered at the boundary points of the acquisition surface $\Ga_{\frac{\pi}{2}}$ (which is illustrated by the green curve in Fig. \ref{fig:fig1}\subref{subfig:OP}) and that are tangent to a singularity of the original phantom. That is, the artifacts are generated by the boundary singularities at $(0.3,0)$, $(-0.3,0)$, $(0,0.3)$, and $(0,-0.3)$. By further examining the artifacts in Fig. \ref{fig:fig1}\subref{subfig:RCNS}, we also observe that the jumps along the added artifact circles are not as sharp as in the case of visible singularities. This indicates that the added artifacts are weaker than original (generating) singularities. In fact, our theoretical analysis shows that they are $\frac{1}{2}$-order weaker.
%\begin{itemize} 
%\item[1)] Two of those circles are centered at $(1,0)$ of radius $0.7$ \footnote{These artifacts are generated by the boundary singularities at $(0.3,0)$.}  and $1.3$ \footnote{These artifacts are generated by the boundary singularities at $(-0.3,0)$.}, respectively. We notice that the contrast between the left and right of these circle is not as much as those described in a). This comes from the the fact that the generating boundary singularities are conormal to the circle centered at the origin of radius $0.3$. According to the discussion after Theorem~\ref{T:Main1}, we obtain that the artifacts are $\frac{1}{2}$-order weaker than the original singularity.
%
%\item[2)]  similar description holds for the other two circles centered at $(0,1)$. 
%\end{itemize}
\end{itemize}
Summing up, this experiment shows that the above observations correspond to our theoretical findings stated in Propostion \ref{P:wavefront} and  Theorem~\ref{T:Main1}.

\medskip

\noindent In the next step we investigate performance of the modified (artifact reduction) reconstruction operator $\mT$. To that end, we use the operator $\mT$ with the cutoff function $\chi^{k}$ defined in \eqref{eq:smooth cutoff}-\eqref{eq:smooth cutoff chi} (cf. also Fig. \ref{fig:1q-simple_smoothing_eps_var_power_var}), and apply it to the limited view data. Note that the cutoff function $\chi^{k}$ is smooth in the interior of $\Ga_{\frac{\pi}{2}}$ and vanishes to an order $k$ at the end points of $\Ga_{\frac{\pi}{2}}$. According to Theorem~\ref{T:Main1}, the reconstructions obtained through $\mT$, will exhibit added artifacts that are $k+\frac{1}{2}$ orders smoother than the original singularities. Therefore, the degree of artifact reduction is linked to the order $k$ and we expect the operator $\mT$ to mitigate artifacts more when the bigger the order $k$ is. In addition to that, we expect that the strength of artifacts is influenced by the parameter $\epsilon$ (see \eqref{eq:smooth cutoff} and the definition of $h^{k}$). 

To investigate the practical performance, we computed a series of artifact reduction reconstructions by varying the parameters $\epsilon$ and $k$. The results are shown in Fig. \ref{fig:1q-simple} and \ref{fig:1q-order_varies}. First, we observe that in all reconstructions shown in Fig. \ref{fig:1q-simple} and \ref{fig:1q-order_varies} most of the visible singularities are well reconstructed. In Fig. \ref{fig:1q-simple}, we have displayed some reconstructions using smoothing order $k=1$ and varying the parameter $\epsilon$. Here we observe that for $\epsilon=0.05$ almost no artifact reduction happens. This is due to the fact that, in the discretization regime, $\chi$ changes very fast near the endpoints of $\Ga_{\frac{\pi}{2}}$ and, hence, $\chi$ behaves like a discontinuous function. The artifact reduction gets clearer as we increase the value of $\epsilon$.  Next, we consider the effect of varying the smoothing parameter $k$ for a fixed $\epsilon=0.2$. The corresponding reconstructions are shown in Fig.~\ref{fig:1q-simple_smoothing_eps_var_power_var}.  As expected, the artifacts get weaker (are better reduced) as the order increases. This is in accordance with our theoretical characterizations in Theorem~\ref{T:Main1}.

\begin{figure}[t]
\centering
 \subfloat[short for lof][$k=1$, $\epsilon=0.05$]{
   \includegraphics[width=0.3\linewidth]{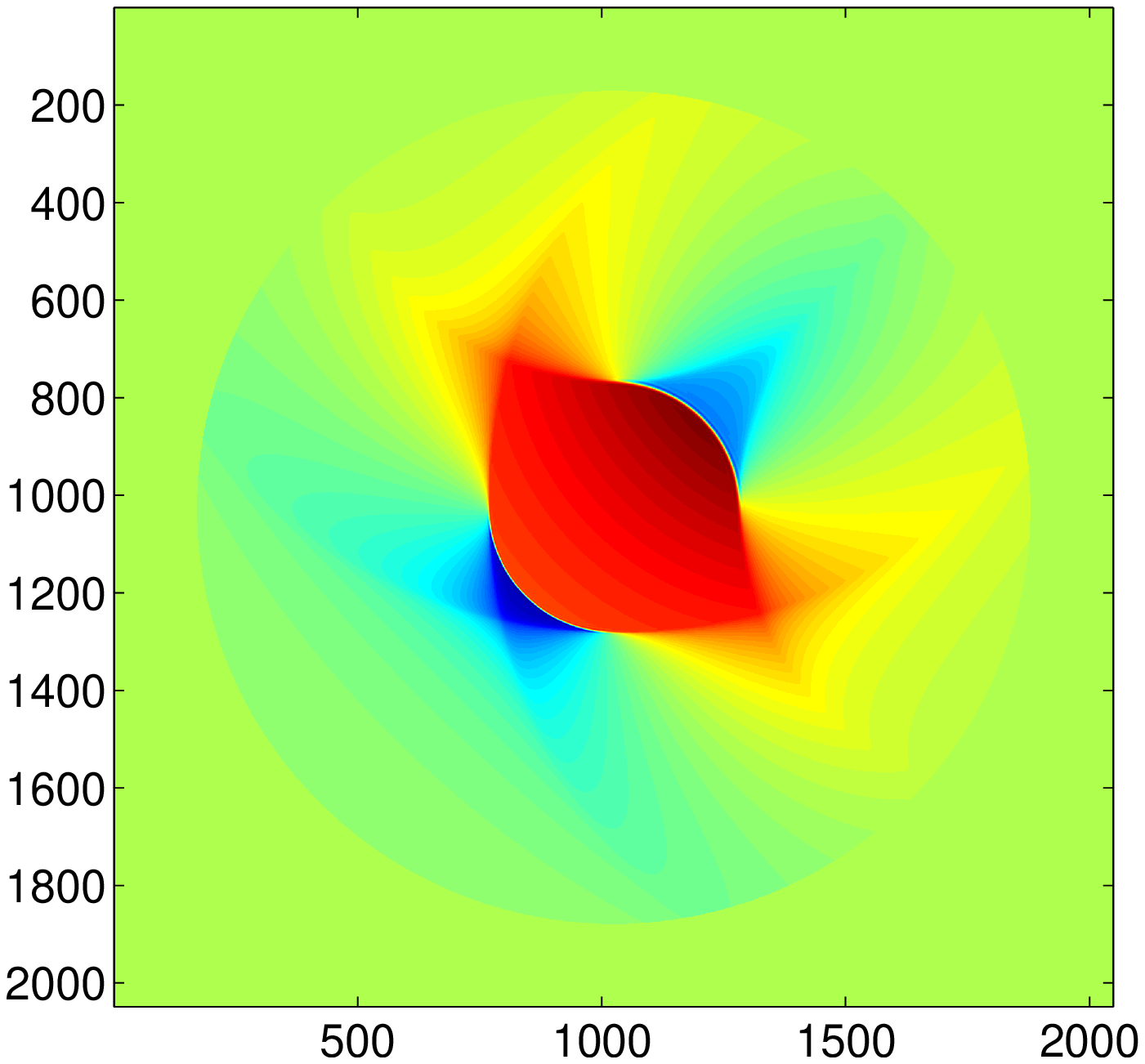}
   \label{subfig:eps3}
}  
%  \hspace{30pt}
\hfill
 \subfloat[short for lof][$k=1$, $\epsilon = 0.2$]{
   \includegraphics[width=0.3\linewidth]{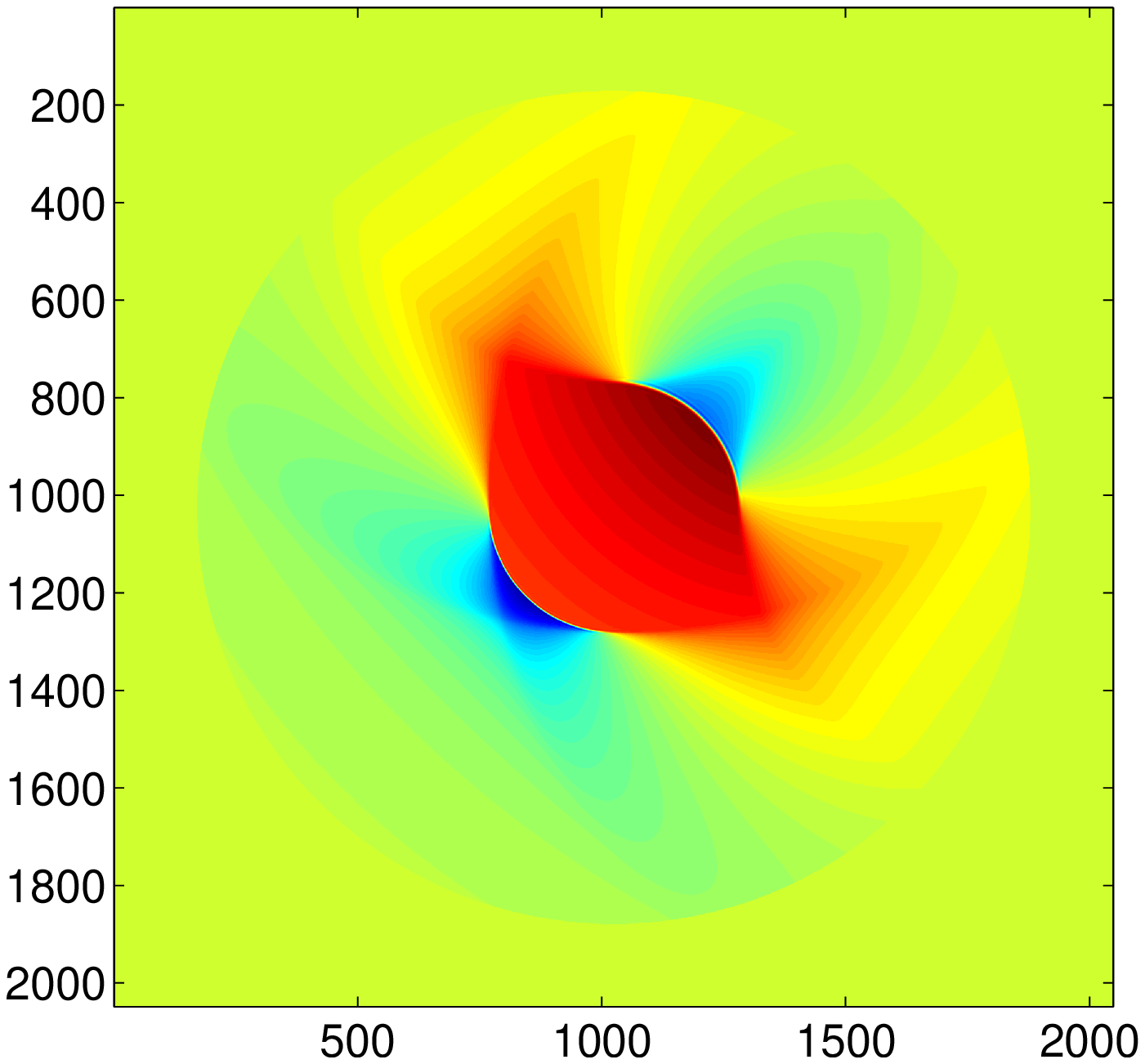}
   \label{subfig:eps5}
}
%\hspace{30pt}
\hfill
 \subfloat[short for lof][$k=1$, $\epsilon= 1$]{
   \includegraphics[width=0.3\linewidth]{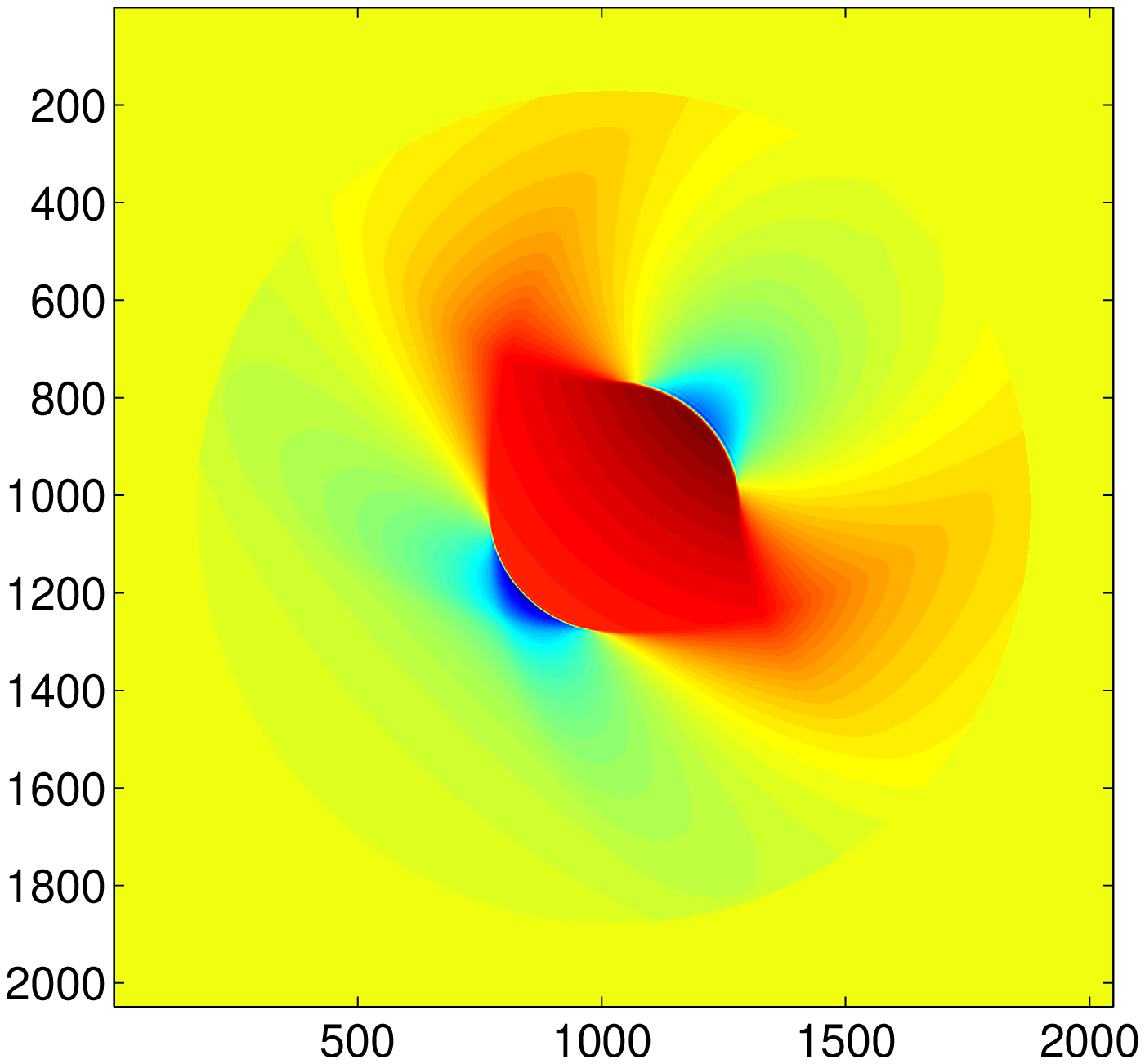}
   \label{subfig:eps8}
}%\\[10pt]
%\hspace*{5pt}
% \subfloat[short for lof][$\epsilon=0.05$]{
%%   \includegraphics[width=0.18\linewidth]{Figures_paper/quart-eps1}
%   \includegraphics[width=0.225\linewidth]{Figures_paper/reconstr_order1_one_circle_grey_N2048_na8192_nr2048_p25_pi2_eps3}  
%   \label{subfig:eps3-bw}
%}  
%\hspace{40pt}
%\subfloat[short for lof][$\epsilon=0.2$]{
%%   \includegraphics[width=0.18\linewidth]{Figures_paper/quart-eps2}
%   \includegraphics[width=0.225\linewidth]{Figures_paper/reconstr_order1_one_circle_grey_N2048_na8192_nr2048_p25_pi2_eps5}
%   \label{subfig:eps5-bw}
%}
%\hspace{40pt}
% \subfloat[short for lof][$\epsilon=1$]{
%%   \includegraphics[width=0.18\linewidth]{Figures_paper/quart-eps3}
%   \includegraphics[width=0.225\linewidth]{Figures_paper/reconstr_order1_one_circle_grey_N2048_na8192_nr2048_p25_pi2_eps8}
%   \label{subfig:eps8-bw}
%}
\caption[short for lof]{Reconstruction from limited view data, acquired on $\Ga_{\frac{\pi}{2}}$ (quarter circle), using the modified reconstruction operator $\mT$ (cf. \eqref{E:mT}) with the smoothing function $\chi^{1}$ defined in \eqref{eq:smooth cutoff}-\eqref{eq:smooth cutoff chi}. The figures illustrate the influence of the parameter $\epsilon$ on artifact reduction.}
\label{fig:1q-simple}
\end{figure}

% -------------------------------------------------------------
\begin{figure}[ht]
\centering
 \subfloat[short for lof][$k=1$, $\epsilon = 0.2$]{
   \includegraphics[width=0.3\linewidth]{reconstr_fl_N2048_na8192_nr2048_p25_eps5}
%   \label{subfig:fig2}
}  
%  \hspace{30pt}
\hfill
 \subfloat[short for lof][$k=2$, $\epsilon = 0.2$]{
   \includegraphics[width=0.3\linewidth]{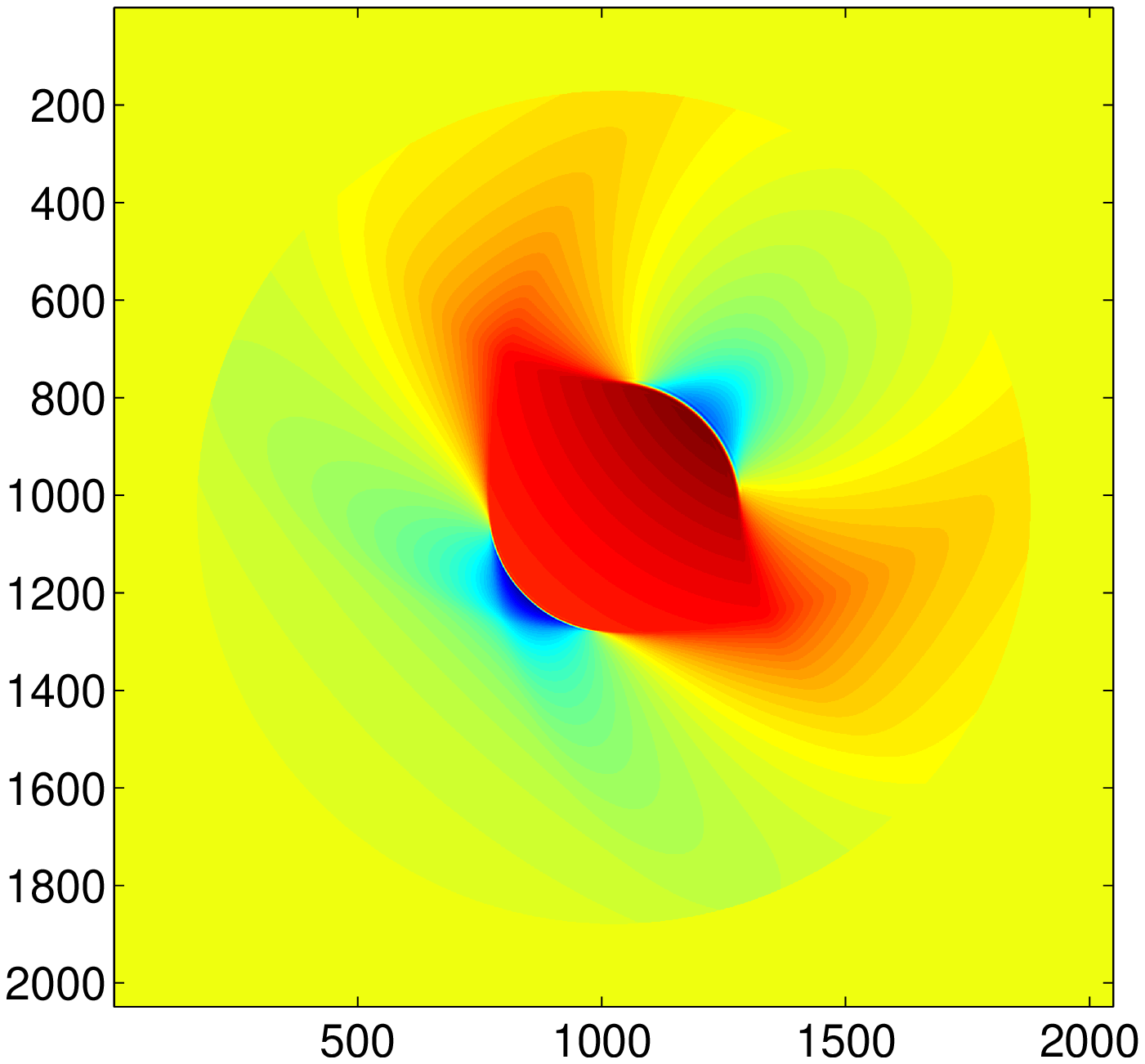}
%   \label{subfig:fig2}
}
%\hspace{30pt}
\hfill
 \subfloat[short for lof][$k=3$, $\epsilon = 0.2$]{
   \includegraphics[width=0.3\linewidth]{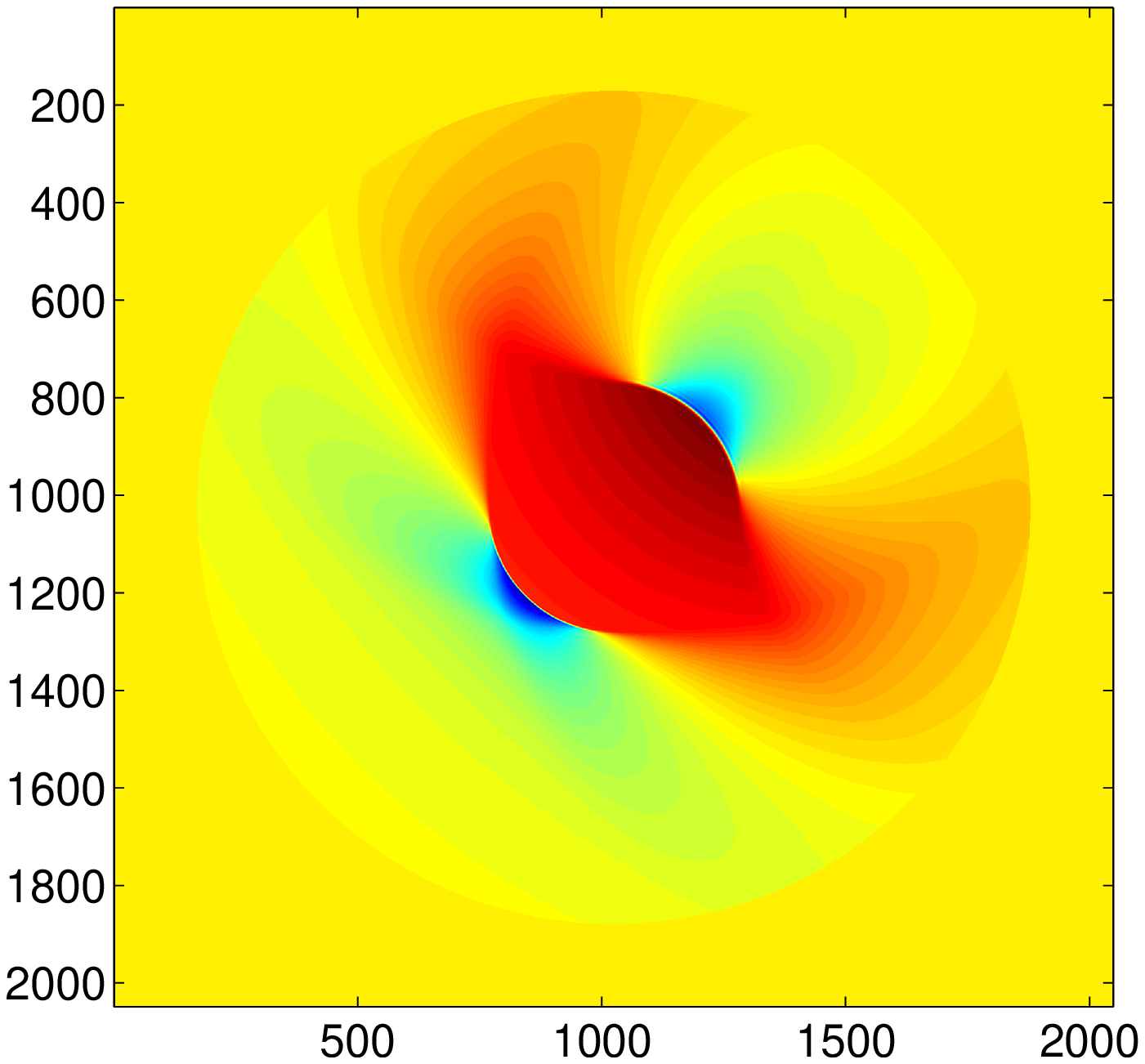}
   %\label{subfig:fig2}
}
%\\[10pt]
%\hspace*{5pt}
% \subfloat[short for lof][Order 1 smoothing]{
%%   \includegraphics[width=0.18\linewidth]{Figures_paper/quart-eps1}
%   \includegraphics[width=0.225\linewidth]{Figures_paper/reconstr_order1_one_circle_grey_N2048_na8192_nr2048_p25_pi2_eps5}  
%%   \label{subfig:fig2}
%}  
%\hspace{40pt}
%\subfloat[short for lof][Order 2 smoothing]{
%%   \includegraphics[width=0.18\linewidth]{Figures_paper/quart-eps2}
%   \includegraphics[width=0.225\linewidth]{Figures_paper/reconstr_order2_one_circle_grey_N2048_na8192_nr2048_p25_pi2_eps5}
%   %\label{subfig:fig2}
%}
%\hspace{40pt}
% \subfloat[short for lof][Order 3 smoothing]{
%%   \includegraphics[width=0.18\linewidth]{Figures_paper/quart-eps3}
%   \includegraphics[width=0.225\linewidth]{Figures_paper/reconstr_order3_one_circle_grey_N2048_na8192_nr2048_p25_pi2_eps5}
%%   \label{subfig:fig2}
%}
%

\caption[short for lof]{Reconstruction from limited view data, acquired on $\Ga_{\frac{\pi}{2}}$ (quarter circle), using the modified reconstruction operator $\mT$ (cf. \eqref{E:mT}) with the smoothing function $\chi^{k}$ defined in \eqref{eq:smooth cutoff}-\eqref{eq:smooth cutoff chi}. The figures illustrate the influence of the smoothing order $k$ on artifact reduction for a fixed parameter $\epsilon=0.2$.}
\label{fig:1q-order_varies}
\end{figure}

%\medskip

\paragraph{Experiment 2.} Our second experiment follows the lead of our first experiment. Here, we only consider a larger angular range where we use limited view data collected on $\Ga_{\frac{3\pi}{2}}$ (three quarters of the unit circle). Again, we compute a series of reconstructions using the standard as well as the modified reconstruction operators, $\mT_{0}$ and $\mT$, respectively. The results of this experiment are shown in Fig.~\ref{fig:3q-no_smoothing} - \ref{fig:3q-order_original_smoothing}. 

Before we start, let us first remark that in this example all singularities of the phantom image are visible (they are located on the circle centered at the origin of radius $0.3$), and the locations of all boundary singularities are the same as in Experiment~1, namely $(\pm 0.3,0)$ and $(0,\pm 0.3)$. In contrast to Experiment 1 where all of the visible singularities were singly visible, we now have both types of visible singularities, doubly and singly visible ones.  Those on the first and third quarters are doubly visible, while those are on the second and fourth quarters are singly visible. See Section \ref{S:2D} for theoretical explanation. 

\begin{figure}[ht]
\centering
 \subfloat[short for lof][Original phantom]{
   \includegraphics[width=0.3\linewidth]{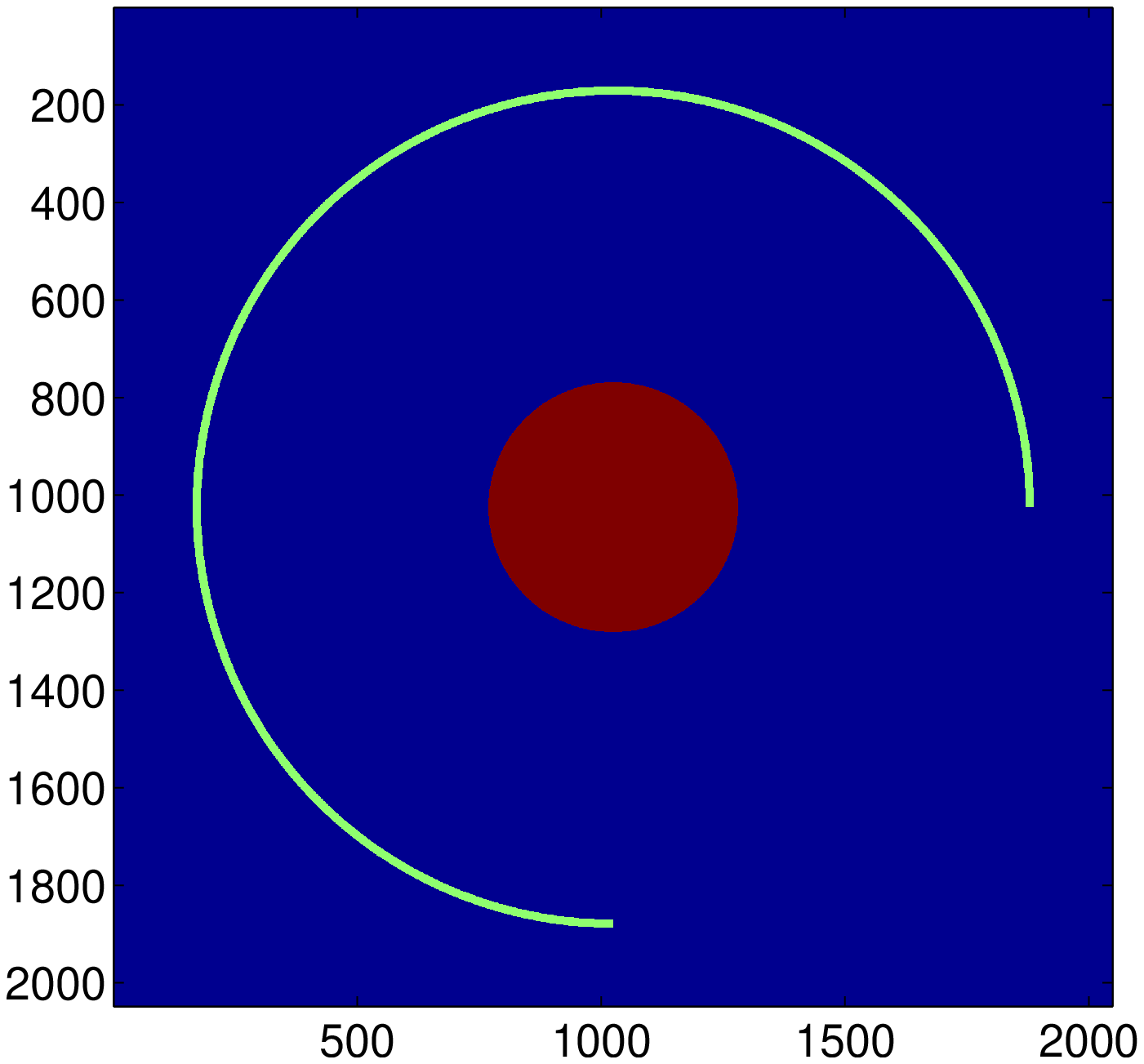}
   \label{subfig:3q-fig1}
 }
   \hspace{15pt}
 \subfloat[short for lof][Reconstruction]{
   \includegraphics[width=0.3\linewidth]{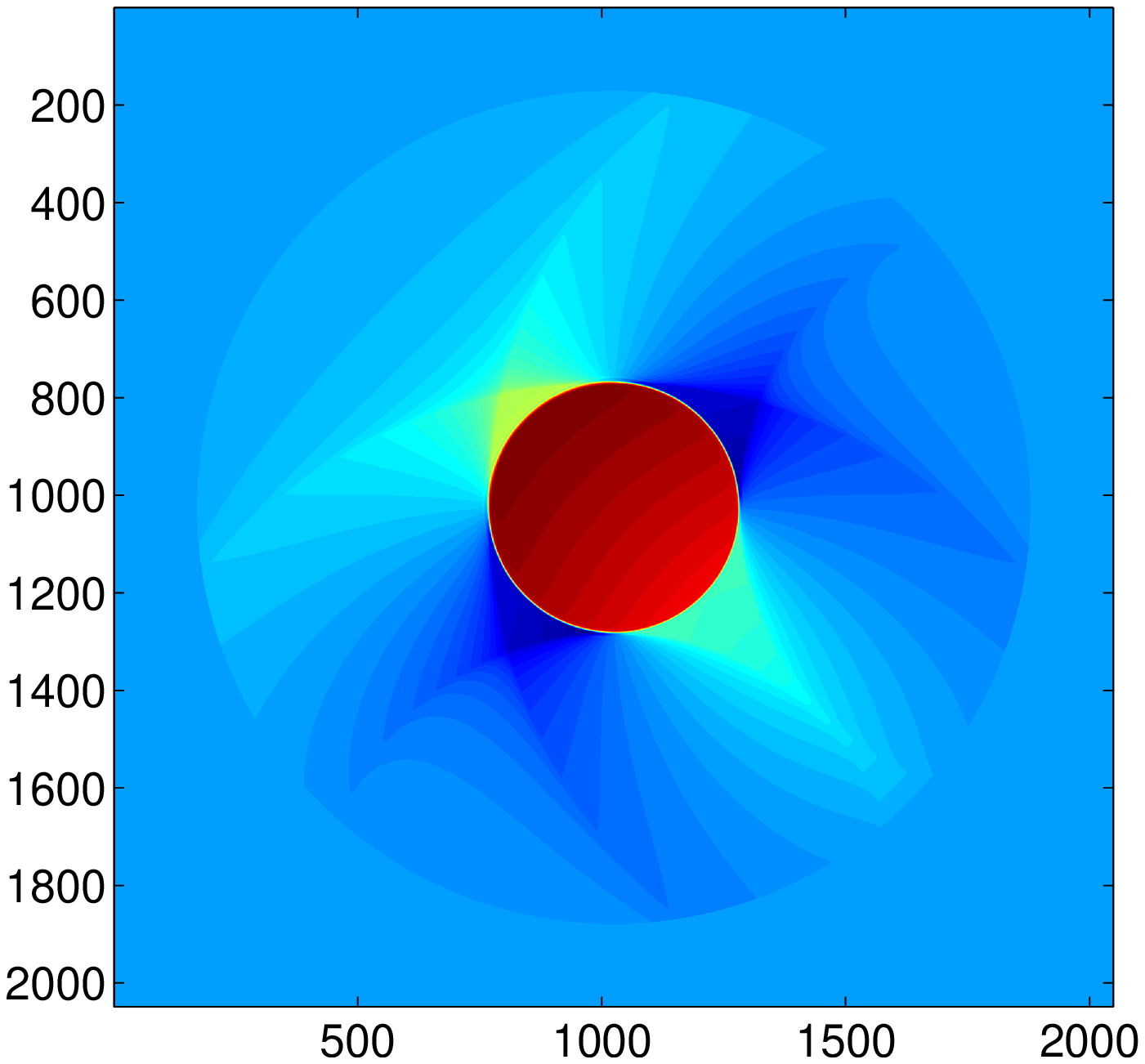}
   \label{subfig:3q-fig2}
}   
%\hspace{15pt}
% \subfloat[short for lof][Black \& white  reconstruction without smoothing ]{
%%   \includegraphics[width=0.28\linewidth]{Figures_paper/3-quart-1st-non}
%\raisebox{5pt} 
%{\includegraphics[width=0.26\linewidth]{Figures_paper/reconstr_one_circle_grey_N2048_na2731_nr2048_p75_ps10_3pi2_s0}  
%   \label{subfig:3q-fig3}
%   }}
\caption[short for lof]{Reconstruction from limited view data, collected on $\Ga_{\frac{3\pi}{2}}$ (three quarters of the unit circle), using the standard reconstruction operator $\mT_{0}$ with no artifact reduction. The acquisition surface  $\Ga_{\frac{3\pi}{2}}$ is illustrated by the green line in the phantom image \subref{subfig:3q-fig1}.}
\label{fig:3q-no_smoothing}
\end{figure}

\begin{figure}[ht]
\centering 
 \subfloat[short for lof][$k=1$, $\epsilon=0.05$]{
   \includegraphics[width=0.3\linewidth]{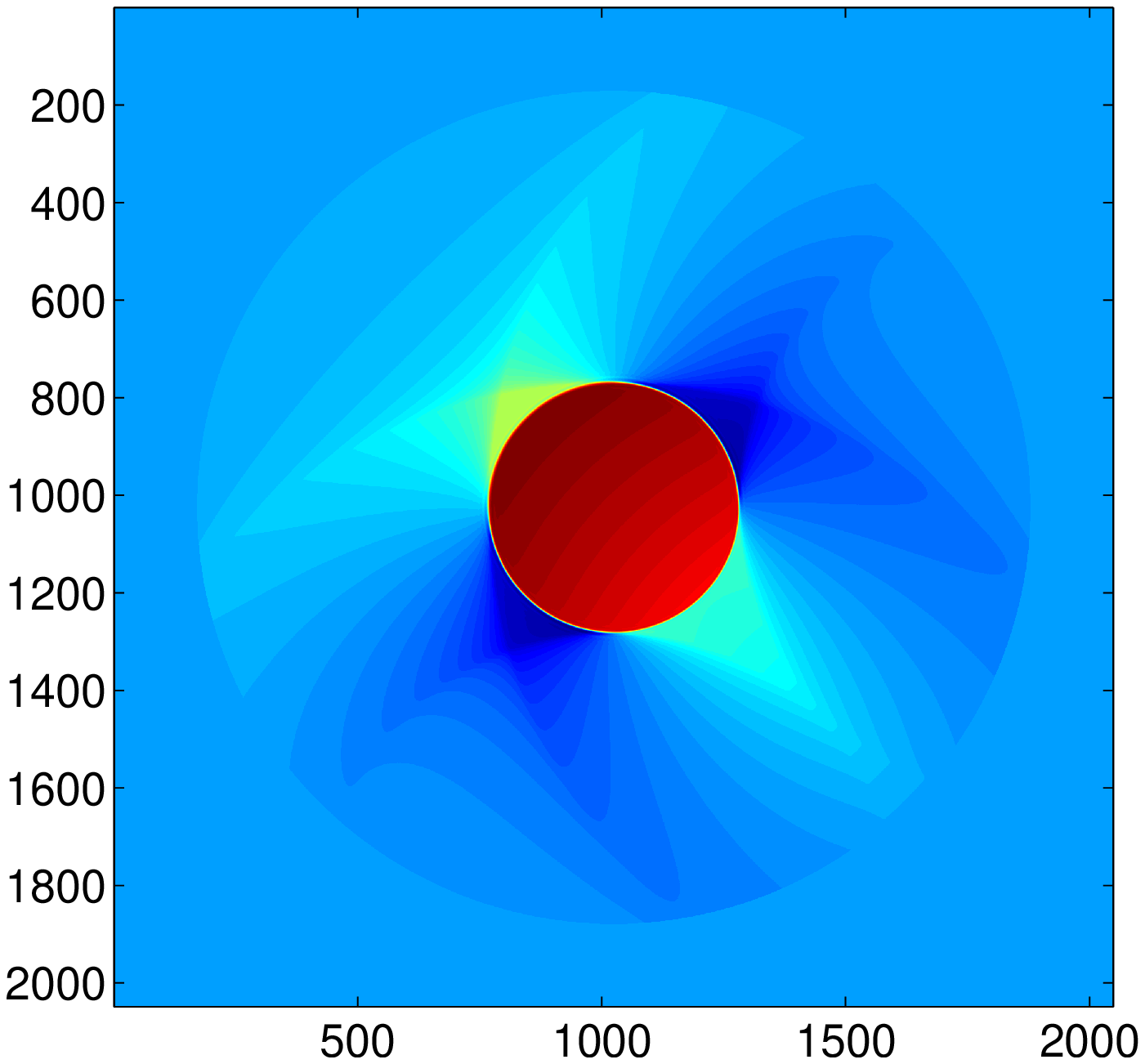}
   \label{subfig:fig2}
}
%  \hspace{20pt}
\hfill
 \subfloat[short for lof][$k=1$, $\epsilon=0.2$]{
   \includegraphics[width=0.3\linewidth]{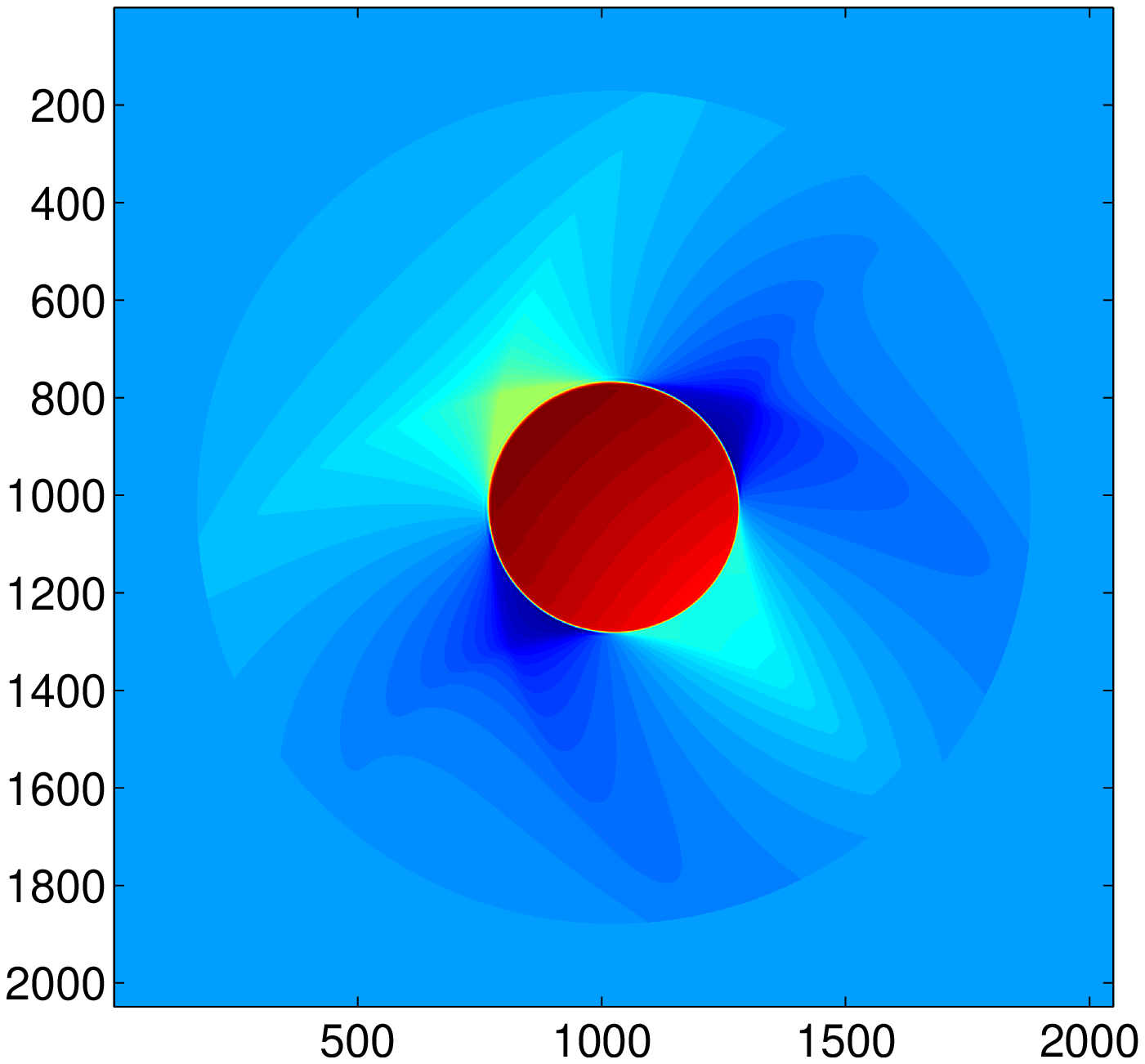}
   \label{subfig:fig2}
}
%  \hspace{20pt}
\hfill
 \subfloat[short for lof][$k=1$, $\epsilon=1$]{
   \includegraphics[width=0.3\linewidth]{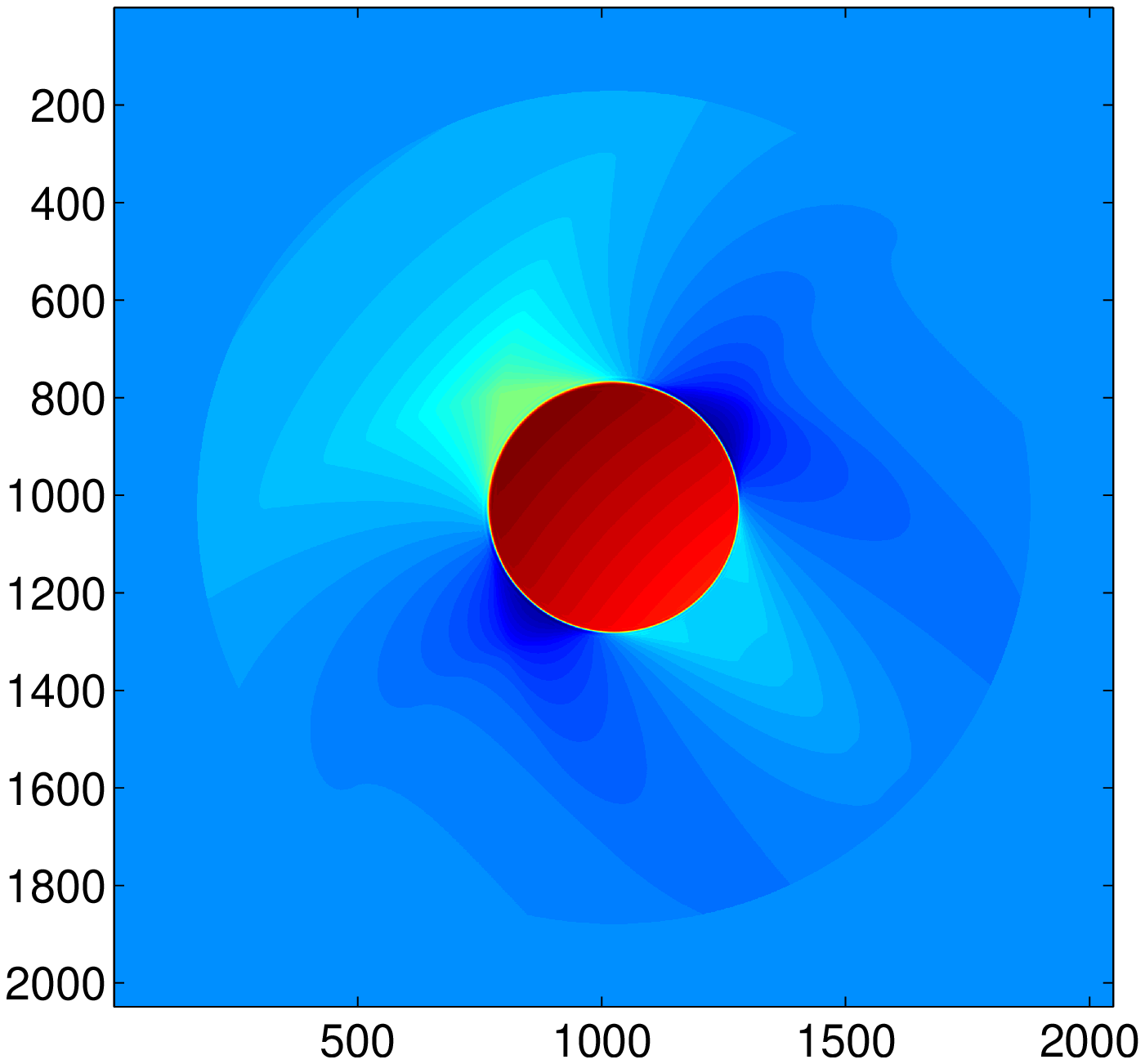}
   \label{subfig:fig2}
}
%\\[10pt]
%\hspace*{5pt}
% \subfloat[short for lof][Reconstruction with $\epsilon=0.05$]{
%%   \includegraphics[width=0.25\linewidth]{Figures_paper/3-quart-1st-eps1}
%   \includegraphics[width=0.23\linewidth]{Figures_paper/reconstr_order1_one_circle_grey_N2048_na2731_nr2048_p75_3pi2_eps3}   
%   \label{subfig:fig2}
%   }
%  \hspace{30pt}
% \subfloat[short for lof][Reconstruction with $\epsilon=0.2$]{
%%   \includegraphics[width=0.25\linewidth]{Figures_paper/3-quart-1st-eps2}
%   \includegraphics[width=0.23\linewidth]{Figures_paper/reconstr_order1_one_circle_grey_N2048_na2731_nr2048_p75_3pi2_eps5}
%   \label{subfig:fig2}
%}
%  \hspace{30pt}
% \subfloat[short for lof][Reconstruction with $\epsilon=1$]{
%%   \includegraphics[width=0.25\linewidth]{Figures_paper/3-quart-1st-eps3}
%   \includegraphics[width=0.23\linewidth]{Figures_paper/reconstr_order1_one_circle_grey_N2048_na2731_nr2048_p75_3pi2_eps8}
%   \label{subfig:fig2}
%}
\caption[short for lof]{Reconstruction from limited view data, acquired on $\Ga_{\frac{3\pi}{2}}$ (three quarters of the unit circle), using the modified reconstruction operator $\mT$ (cf. \eqref{E:mT}) with the smoothing function $\chi^{1}$ defined in \eqref{eq:smooth cutoff}-\eqref{eq:smooth cutoff chi}. The figures illustrate the influence of the parameter $\epsilon$ on artifact reduction.}
\label{fig:3q-simple_original_smoothing}
\end{figure}

\begin{figure}[ht]
\centering 
 \subfloat[short for lof][$k=1$, $\epsilon=0.2$]{
   \includegraphics[width=0.3\linewidth]{reconstr_fl_N2048_na2731_nr2048_p75_ps10_eps5}
   \label{subfig:fig2}
}
%  \hspace{20pt}
\hfill
 \subfloat[short for lof][$k=2$, $\epsilon=0.2$]{
   \includegraphics[width=0.3\linewidth]{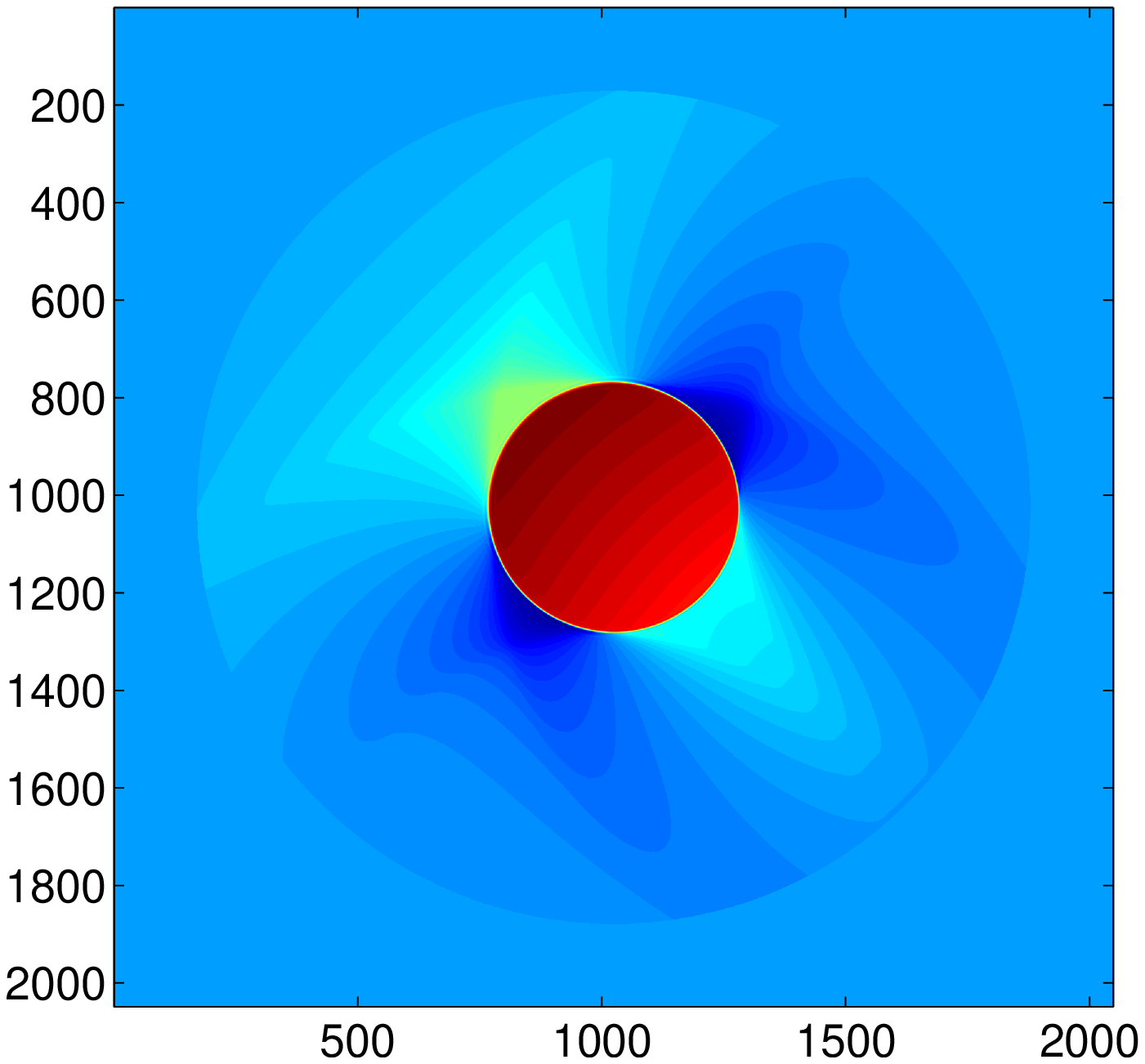}
   \label{subfig:fig2}
}
%  \hspace{20pt}
\hfill
 \subfloat[short for lof][$k=3$, $\epsilon=0.2$]{
   \includegraphics[width=0.3\linewidth]{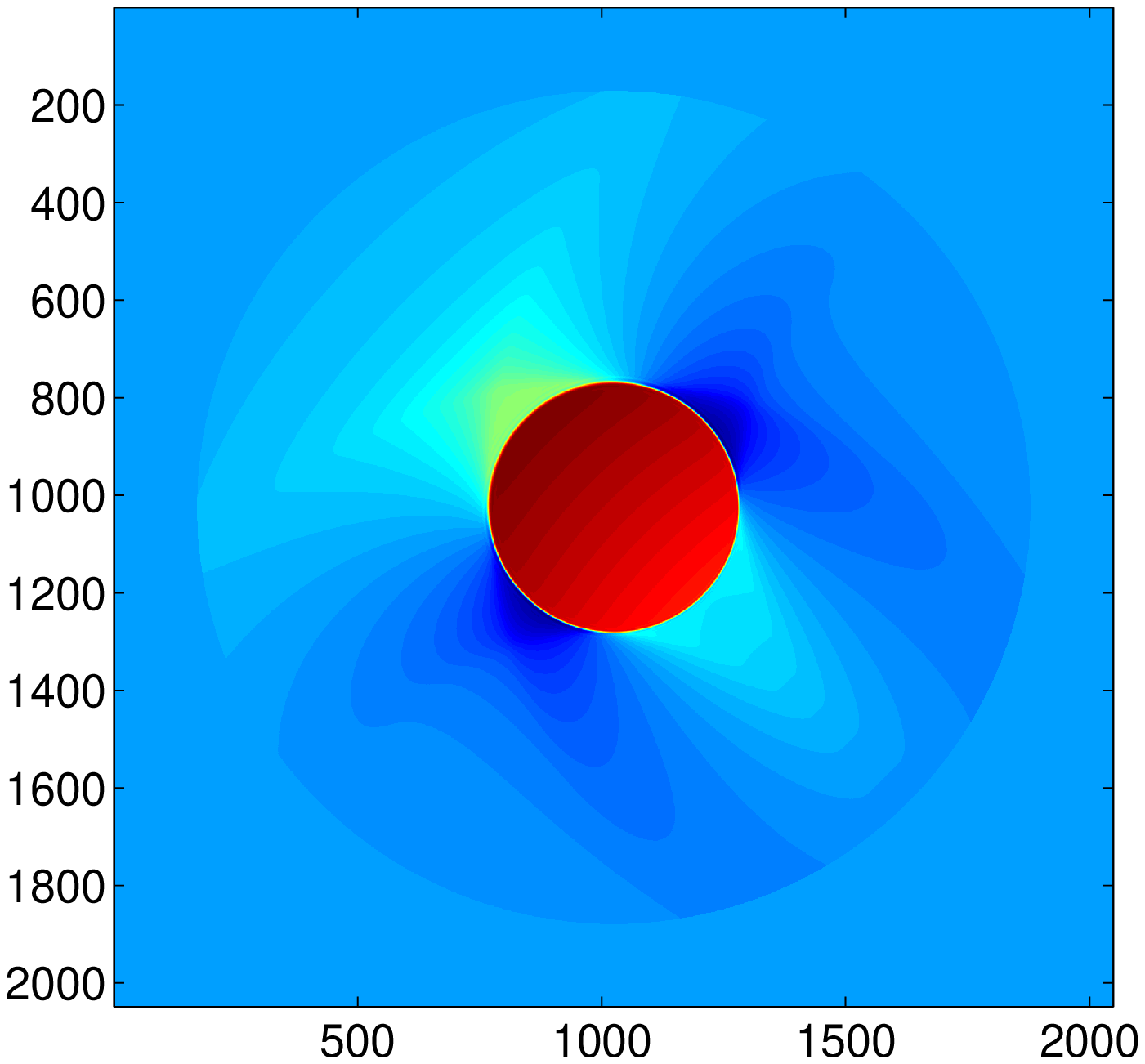}
   \label{subfig:fig2}
}
%\\[10pt]
%\hspace{5pt}
% \subfloat[short for lof][Order 1 smoothing]{
%%   \includegraphics[width=0.25\linewidth]{Figures_paper/3-quart-1st-eps1}
%   \includegraphics[width=0.23\linewidth]{Figures_paper/reconstr_order1_one_circle_grey_N2048_na2731_nr2048_p75_3pi2_eps5}
%   \label{subfig:fig2}
%   }
%  \hspace{30pt}
% \subfloat[short for lof][Order 2 smoothing]{
%%   \includegraphics[width=0.25\linewidth]{Figures_paper/3-quart-1st-eps2}
%   \includegraphics[width=0.23\linewidth]{Figures_paper/reconstr_order2_one_circle_grey_N2048_na2731_nr2048_p75_3pi2_eps5}
%   \label{subfig:fig2}
%}
%  \hspace{30pt}
% \subfloat[short for lof][Order 3 smoothing]{
%%   \includegraphics[width=0.25\linewidth]{Figures_paper/3-quart-1st-eps3}
%   \includegraphics[width=0.23\linewidth]{Figures_paper/reconstr_order3_one_circle_grey_N2048_na2731_nr2048_p75_3pi2_eps5}
%   \label{subfig:fig2}
%}
\caption[short for lof]{Reconstruction from limited view data, acquired on $\Ga_{\frac{3\pi}{2}}$ (three quarters of the unit circle), using the modified reconstruction operator $\mT$ (cf. \eqref{E:mT}) with the smoothing function $\chi^{k}$ defined in \eqref{eq:smooth cutoff}-\eqref{eq:smooth cutoff chi}. The figures illustrate the influence of the smoothing order $k$ on artifact reduction for a fixed parameter $\epsilon=0.2$.}
\label{fig:3q-order_original_smoothing}
\end{figure}

\medskip

By examining the reconstructions using the standard reconstruction formula $\mT_0$ (i.e. $\chi =\chi_\Ga$) in Fig.~\ref{fig:3q-no_smoothing} we easily observe that indeed all singularities of the phantom are reconstructed reliably. Here, the doubly visible singularities have more contrast than the singly visible ones. This is due to the fact that for each doubly visible singularity there are two positions on the acquisition arc  $\Ga_{\frac{3\pi}{2}}$ from which the singularity is visible, whereas there is only one position on this arc for a singly visible singularity. Mathematically, this is reflected by the different values of the principal symbol $\sg_0(x,\xi)$ of the reconstructions operator $\mT$, cf. Theorem~\ref{T:Main1}, where we can see that the principal symbol $\sg_0(x,\xi)=1$ if $(x,\xi)$ is doubly visible and $\sg(x,\xi) = \frac{1}{2}$ if $(x,\xi)$ is singly visible. In Fig.~\ref{fig:3q-simple_original_smoothing}, we further observe that added artifacts are generated on circles that are centered at the boundary points of $\Ga_{\frac{3\pi}{2}}$ and tangent to the boundary singularities. These artifacts, however, are not as strong as the reconstructed singularities, which is again in accordance with our theoretical results, see Theorem~\ref{T:Main1} and the discussion below. We again studied the performance of artifact reduction by using the modified reconstruction operator $\mT$ with the smoothing function $\chi^{k}$ for $b=\frac{3\pi}{2}$ (cf. \eqref{eq:smooth cutoff}-\eqref{eq:smooth cutoff chi}). The reconstruction results for varying parameters $\epsilon$ and for varying smoothing orders are shown in Fig.~\ref{fig:3q-simple_original_smoothing} and Fig.~\ref{fig:3q-order_original_smoothing}, respectively. Not surprisingly, we observe here the same behavior as in Experiment 1.

%Moreover, we also the smoothing process as in Experiment~1 with
%$$H(s) = \frac{s(\frac{3 \pi}{2} -s)}{s(\frac{2 \pi}{2} -s) + \epsilon}, \quad h(s) = \frac{H(s)}{H(\frac{3 \pi}{4})}. $$
%The reconstruction of order one smoothing with various choices of $\epsilon$ are shown in Fig.~\ref{fig:3q-simple_original_smoothing}. The reconstruction with various smoothing order and $\epsilon=0.2$ are shown in Fig.~\ref{fig:3q-order_original_smoothing}. We notice that the artifacts get weaker as we increase $\epsilon$ or the order.

%\medskip

\paragraph{Experiment 3.} In our last experiment we investigate how the choice of the smooth cutoff function for $\mT$ influences the artifact reduction. To that end, we consider the same limited view situation as in the Experiment 2 where the data are collected on $\Ga_{\frac{3\pi}{2}}$ and define a new smoothing function $\chi$ which is equal to $1$ in the interior of $\Ga_{\frac{3\pi}{2}}$ and smoothly decreases to $0$ in transition regions of length $\epsilon$ at the boundary of $\Ga_{\frac{3\pi}{2}}$. To that end, we let \[
h_0(s)=\frac{1}{\epsilon^2} s(2\epsilon-s), \quad 0\leq s \leq \epsilon
\]
and
\begin{equation}
\label{eq:new cutoff}
	h_{\mathrm{new}}\left(\frac{3 \pi}{2} s\right)=\left\{
		\begin{array}{l}
			h_0(s), \hspace{40pt} 0\leq s \leq \, \epsilon, \\[5pt]
			1, \hspace{55pt}  \epsilon< s <1-\epsilon, \\[5pt]
			h_0(-s+1), \hspace{10pt}   1 -\epsilon  \leq s \leq  1.
		\end{array}
	\right.
\end{equation}
The parameter $\epsilon>0$ again controls how close the function $\chi$ is to the constant function $1$. Moreover, this function $h$ vanishes to order 1 at the endpoints of $\Ga_{\frac{3\pi}{2}}$. To obtain higher order smoothness at the endpoints we again consider integer powers of $h$ and set 
\begin{equation}
\label{eq:new cutoff chi}
	\chi_{\mathrm{new}}^{k}(z(s))=h_{\mathrm{new}}^{k}(s).
\end{equation}
A plot of the functions $h^{k}(s)$ is depicted in Fig. \ref{fig:function_h} for different values of $\epsilon$ and $k$. The corresponding limited view reconstructions are presented in Fig. \ref{fig:3q-2sm_perc_varies} and \ref{fig:3q-2sm_order_varies}.

The advantage of such a choice of the function $\chi$ lies in the fact that it is exactly (not approximately) equal to $1$ in the range $\frac{3 \pi}{2} \epsilon \leq s \leq \frac{3 \pi}{2}(1-\epsilon)$. Therefore, if $\epsilon$ is very small, most of the visible singularities are reconstructed up to the factor $1$ (if it is doubly visible) or $\frac{1}{2}$ (if it is singly visible). However, the (theoretical) disadvantage of such a choice is due to singularities of $\chi$ at the interior points $z(\frac{3 \pi}{2} \epsilon)$ and $z(\frac{3 \pi}{2}(1-\epsilon))$. According to our analysis in Section \ref{S:2D}, this may lead to the generation of added artifacts (located on circles that are centered around these points). However, since $\chi^{k}$ is $k$ order smoother at these inner points than at the endpoints, those new artifacts will be weaker than those rotating around the endpoints $z(0)$ and $z(\frac{3 \pi}{2})$. Indeed, as can be seen in Fig. \ref{fig:3q-2sm_perc_varies} and \ref{fig:3q-2sm_order_varies}, the new added artifacts are too weak to be recognized in the reconstructions. 

%Let us now consider reconstructions in Fig. \ref{fig:3q-2sm_perc_varies} and \ref{fig:3q-2sm_order_varies}.   
Concerning the influence of parameters $\epsilon$ and $k$, we arrive at similar observations as in Experiments 1 and 2. Comparing the reconstructions that were computed with different smoothing functions in Fig. \ref{fig:3q-order_original_smoothing} and Fig. \ref{fig:3q-2sm_order_varies} we observe that the new smoothing function leads almost always to a significantly better artifact reduction. For example, in Fig. \ref{fig:3q-2sm_order_varies}\subref{subfig:fig4}, the artifacts almost completely vanish and the phantom as well as the background are reconstructed very well. This examples shows that the choice of the smoothing function might influence artifact reduction performance significantly.

%We also experiment with another choice of the function $\chi$ as follows. Let \[
%h_0(s)=\frac{1}{\epsilon^2} s(2\epsilon-s), \quad 0\leq s \leq \epsilon.
%\]
%and
%\[
%h(s)=\left\{
%\begin{array}{l}
%h_0(s), \hspace{40pt} 0\leq s \leq \epsilon, \\[5pt]
%1, \hspace{55pt}  \epsilon< s <\frac{3\pi}{ 2} -\epsilon, \\[5pt]
%h_0(-s+\frac{3\pi}{2}), \hspace{10pt}   \frac{3\pi}{ 2}  -\epsilon \leq s \leq   \frac{3\pi}{ 2} .
%\end{array}
%\right.
%\]
%

% 
\begin{figure}[ht]
\centering
 \subfloat[short for lof][$h_{\mathrm{new}}^{1}$  for $\epsilon\in\{0.1,0.25,0.4\}$]{
   \includegraphics[width=0.33\linewidth]{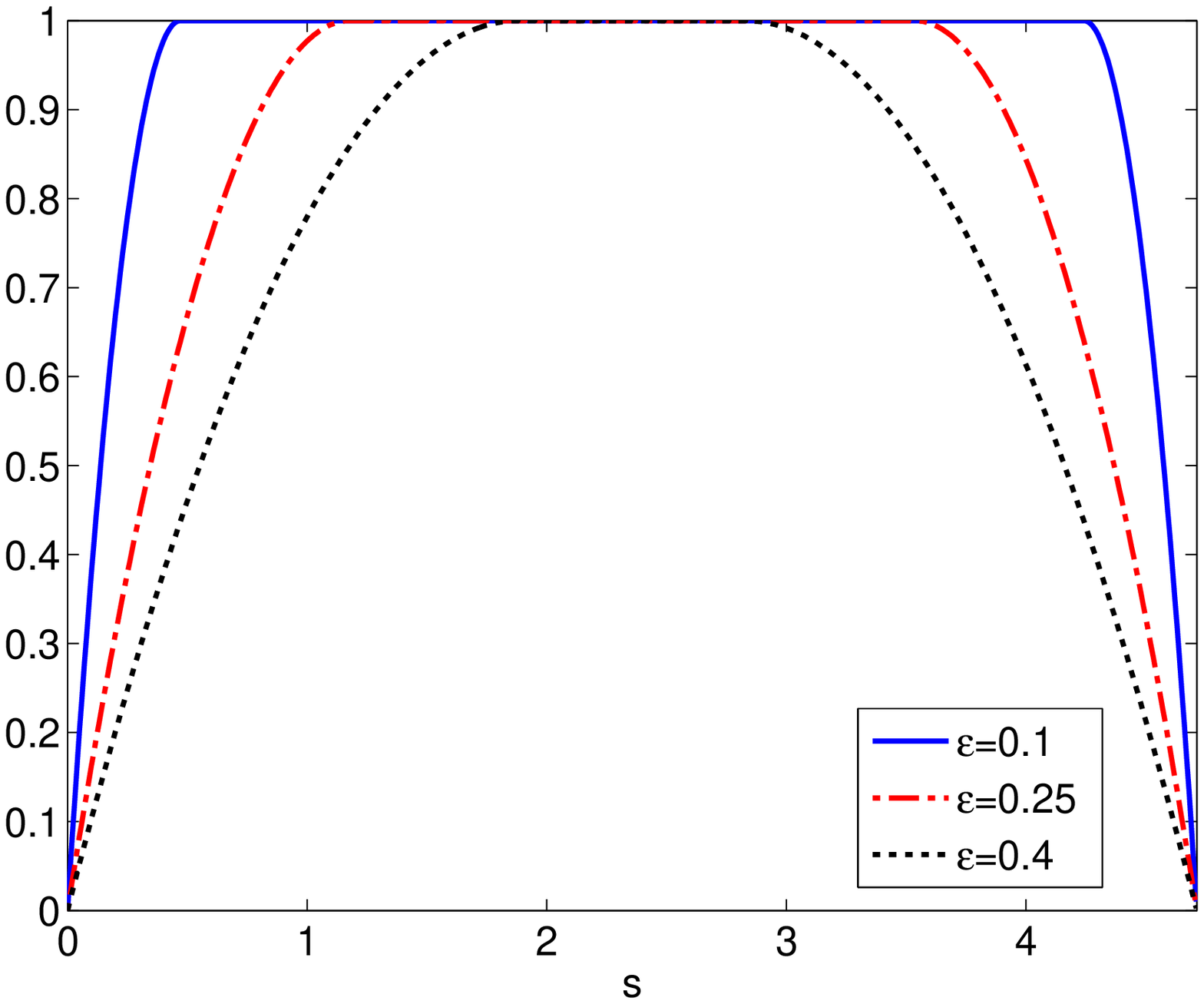}  
   \label{subfig:fig1}
 }
  \hspace{15pt}
 \subfloat[short for lof][$h_{\mathrm{new}}^{k}$  for $\epsilon=0.25$ and $k=1,2,3$]{
   \includegraphics[width=0.33\linewidth]{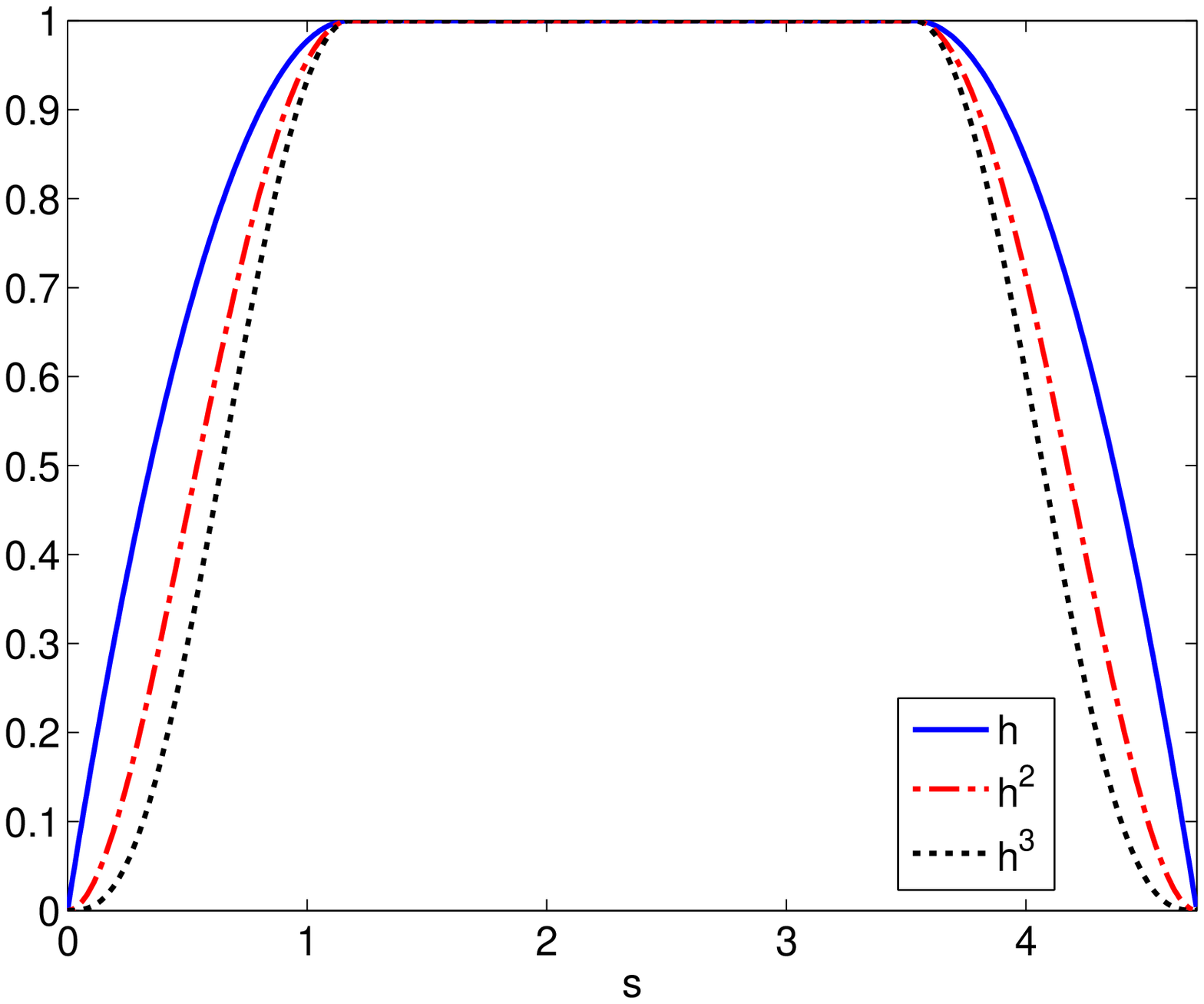}
   \label{subfig:fig2}
}

\caption[short for lof]{Alternative smoothing function $h_{\mathrm{new}}^{k}$ for different values of $\epsilon$ (length of smooth transition region at the boundary of the angular range) and for different smoothing orders $k$, cf. \eqref{eq:new cutoff}-\eqref{eq:new cutoff chi}.}
\label{fig:function_h}

\end{figure}

\begin{figure}[ht]
\centering

\subfloat[short for lof][$k=1$, $\epsilon=0.1$]{%[Reconstruction with $10\%$ smoothing]
   \includegraphics[width=0.3\linewidth]{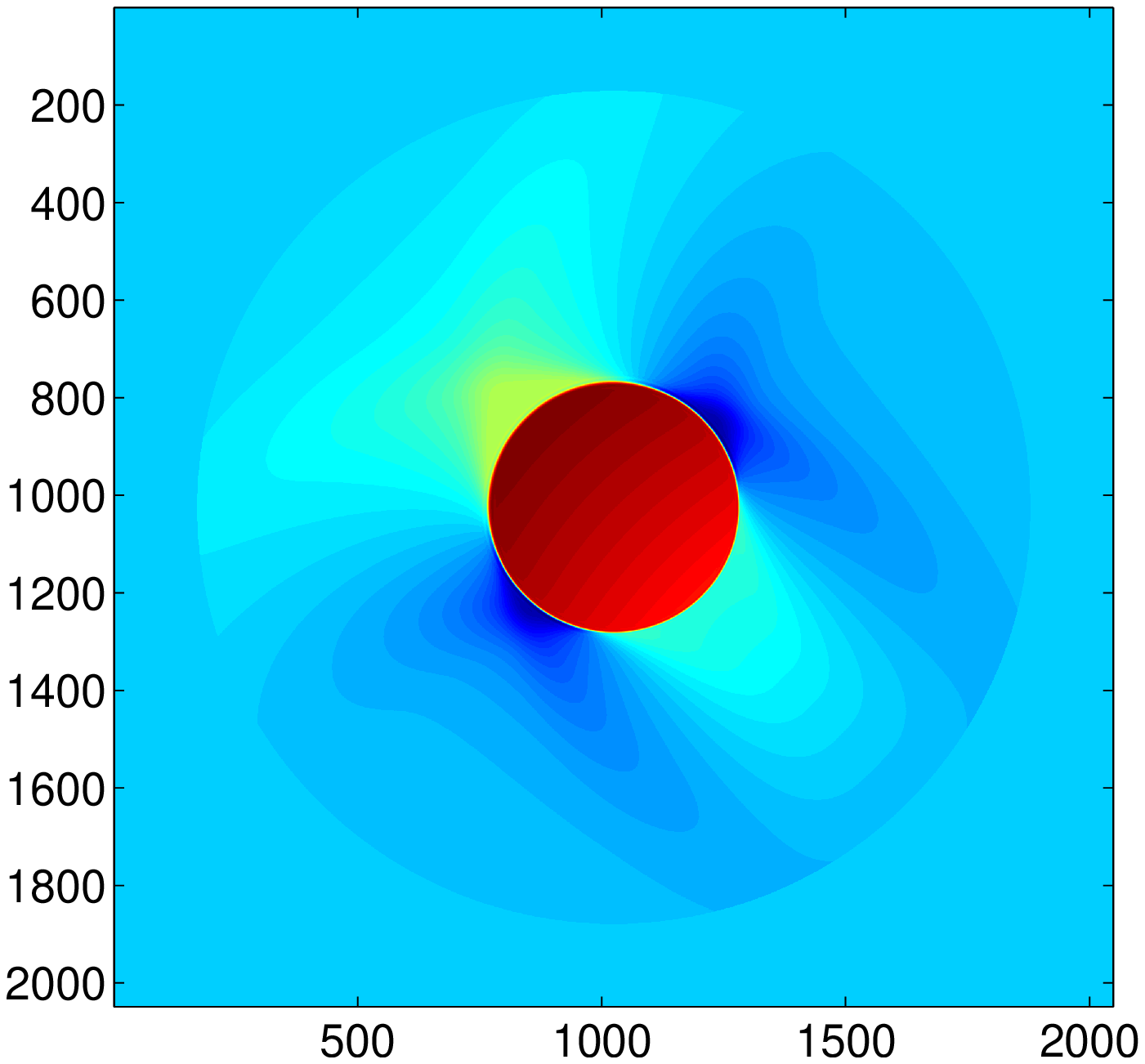}
   \label{subfig:fig1}
 }
%  \hspace{20pt}
\hfill
 \subfloat[short for lof][$k=1$, $\epsilon=0.25$]{%[Reconstruction with $25\%$ smoothing]
   \includegraphics[width=0.3\linewidth]{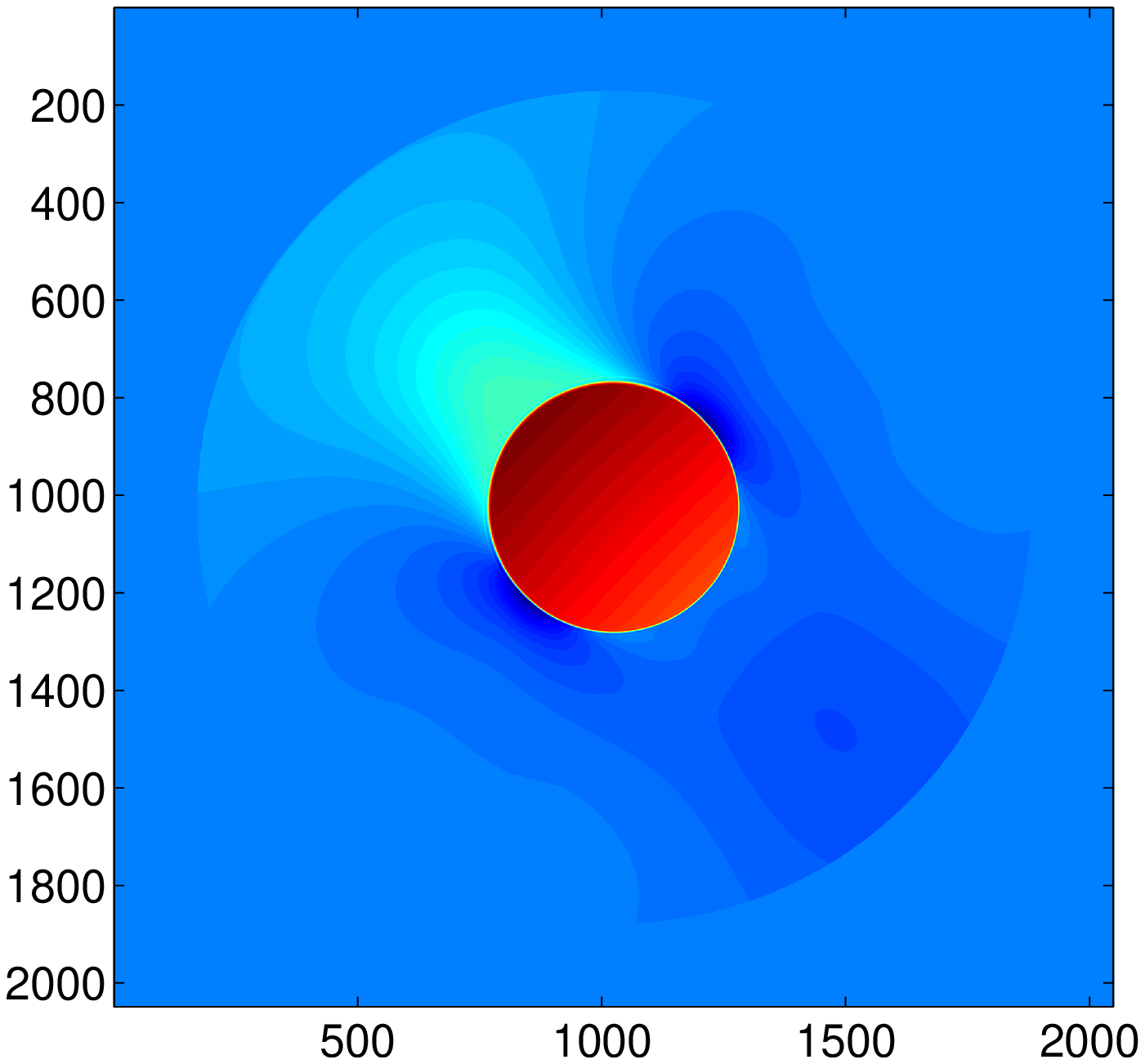}
   \label{subfig:fig2}
}
% \hspace{20pt}
\hfill
 \subfloat[short for lof][$k=1$, $\epsilon=0.4$]{%[Reconstruction with $40\%$ smoothing]
   \includegraphics[width=0.3\linewidth]{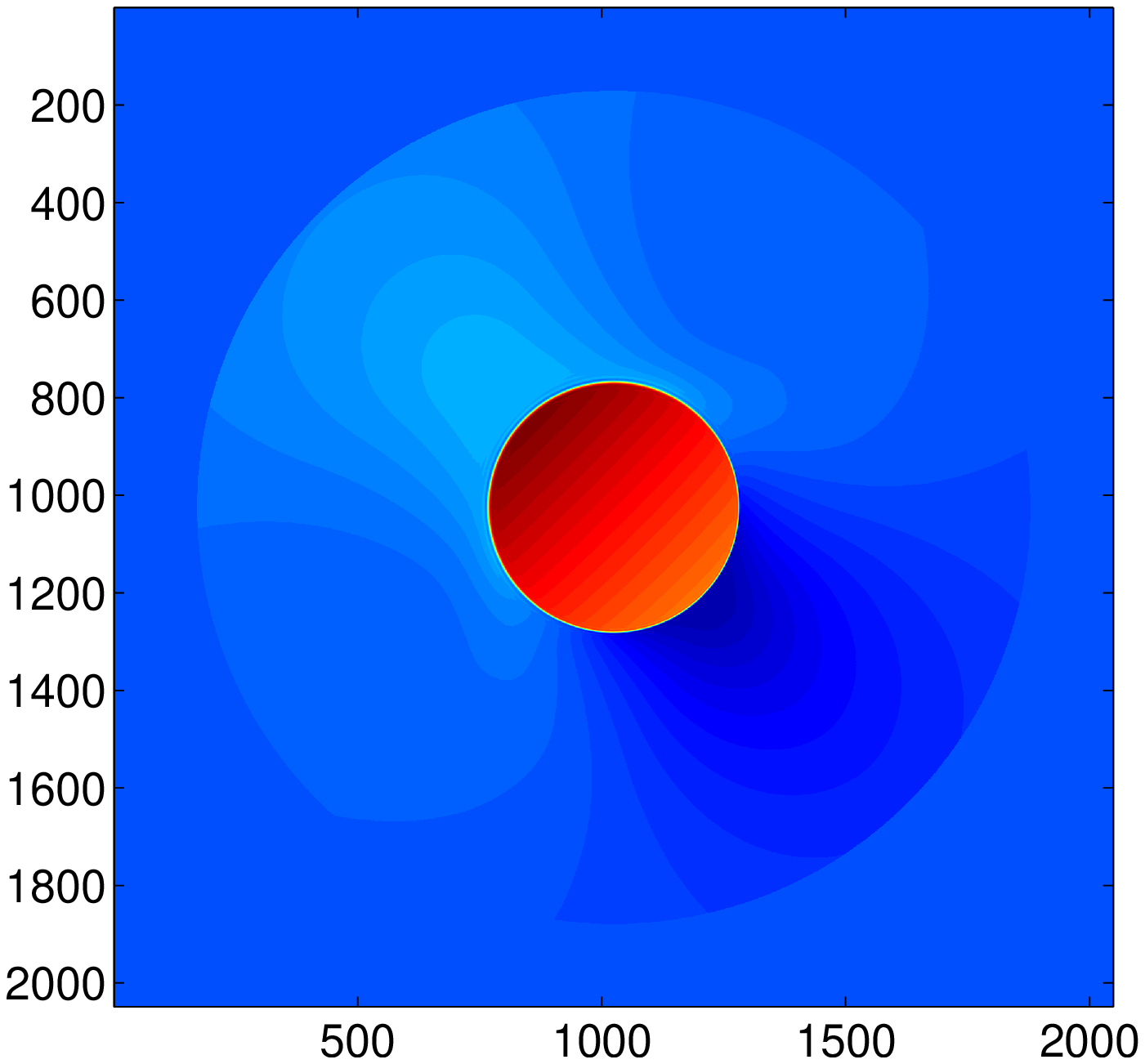}
   \label{subfig:fig2}
} 
%\\[10pt]
%\hspace*{3pt}
%\subfloat[short for lof][Reconstruction with $10\%$ smoothing]{
%   \includegraphics[width=0.23\linewidth]{Figures_paper/reconstr_order2_one_circle_grey_N2048_na2731_nr2048_p75_ps10_3pi2_s2}
%   \label{subfig:fig1}
% }
%  \hspace{30pt}
% \subfloat[short for lof][Reconstruction with $25\%$ smoothing]{
%   \includegraphics[width=0.23\linewidth]{Figures_paper/reconstr_order2_one_circle_grey_N2048_na2731_nr2048_p75_ps25_3pi2_s2}
%   \label{subfig:fig2}
%}
% \hspace{30pt}
% \subfloat[short for lof][Reconstruction with $40\%$ smoothing]{
%   \includegraphics[width=0.23\linewidth]{Figures_paper/reconstr_order2_one_circle_grey_N2048_na2731_nr2048_p75_ps40_3pi2_s2}
%   \label{subfig:fig2}
%}

\caption[short for lof]{Reconstruction from limited view data, acquired on $\Ga_{\frac{3\pi}{2}}$ (three quarters of the unit circle), using the modified reconstruction operator $\mT$ (cf. \eqref{E:mT}) with the new smoothing function $\chi_{\mathrm{new}}^{k}$ defined in \eqref{eq:new cutoff}-\eqref{eq:new cutoff chi}. The figures illustrate the influence of the length of the smooth transition region (parameter $\epsilon$) on artifact reduction for a fixed smoothing order $1$.} %Experiment~2 -  New smoothing function - Order $1$ smoothing.
\label{fig:3q-2sm_perc_varies}
\end{figure}

\begin{figure}[ht]
\centering
% \subfloat[short for lof][Original phantom]{
%   \includegraphics[width=0.20\linewidth]{Figures_paper/original_phantom_N2048_Rmax4_3pi2}
 %  \label{subfig:fig1}
% }
%  \hspace{10pt}
 \subfloat[short for lof][no smoothing]{
   \includegraphics[width=0.3\linewidth]{reconstr_fl_N2048_na2731_nr2048_p75_s0}
   \label{subfig:fig2}   
}
%  \hspace{10pt}
\hfill
 \subfloat[short for lof][$k=1$, $\epsilon=0.4$]{
   \includegraphics[width=0.3\linewidth]{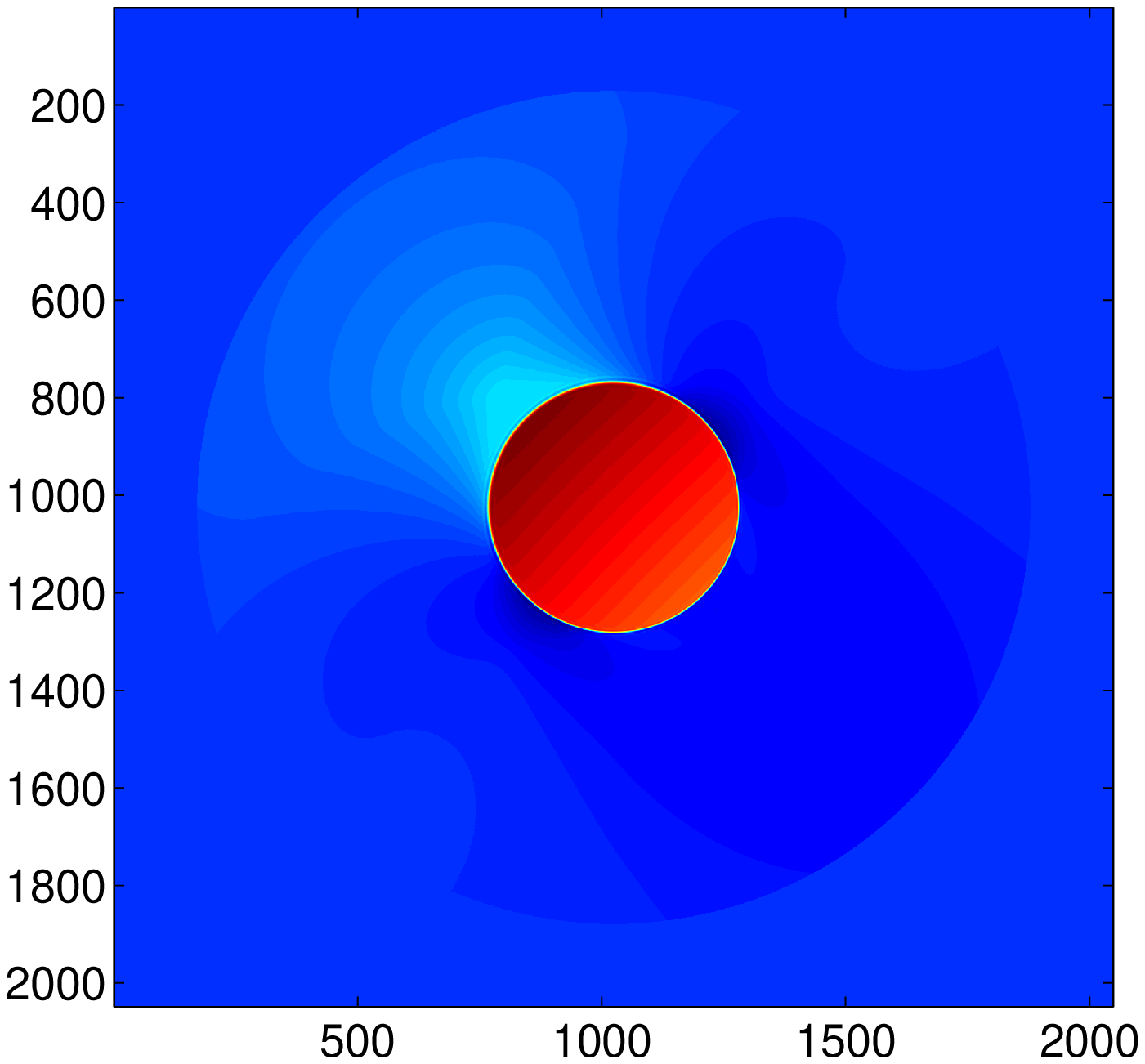}
   \label{subfig:fig3}   
}
%  \hspace{10pt}
\hfill
 \subfloat[short for lof][$k=2$, $\epsilon=0.4$]{
   \includegraphics[width=0.3\linewidth]{reconstr_fl_N2048_na2731_nr2048_p75_ps40_s2}
   \label{subfig:fig4}   
} 
%\\
% \subfloat[short for lof][Reconstruction without smoothing]{
%   \includegraphics[width=0.23\linewidth]{Figures_paper/reconstr_one_circle_grey_N2048_na2731_nr2048_p75_3pi2_s0.eps}
%   \label{subfig:fig2}   
%}
%  \hspace{19pt}
% \subfloat[short for lof][Reconstruction with 1st order smoothing]{
%   \includegraphics[width=0.23\linewidth]{Figures_paper/reconstr_order1_one_circle_grey_N2048_na2731_nr2048_p75_ps40_3pi2_s1}
%   \label{subfig:fig3}   
%}
%  \hspace{19pt}
% \subfloat[short for lof][Reconstruction with 2nd order smoothing]{
%   \includegraphics[width=0.23\linewidth]{Figures_paper/reconstr_order2_one_circle_grey_N2048_na2731_nr2048_p75_ps40_3pi2_s2}
%   \label{subfig:new-ord2}   
%}
\caption[short for lof]{Reconstruction from limited view data, acquired on $\Ga_{\frac{3\pi}{2}}$ (three quarters of the unit circle), using the modified reconstruction operator $\mT$ (cf. \eqref{E:mT}) with the new smoothing function $\chi_{\mathrm{new}}^{k}$ defined in \eqref{eq:new cutoff}-\eqref{eq:new cutoff chi}. The figures illustrate the influence of the order of smoothing $k$ on artifact reduction for a fixed length of the smooth transition region ($\epsilon=0.4$).}%Experiment~2 -  New smoothing function - Smoothing with various orders with $40\%$ smoothing 
\label{fig:3q-2sm_order_varies}

\end{figure}

%\medskip

\paragraph{Conclusion.} The above numerical experiments show that our theoretical results directly translate into practical observations. In particular, the proposed artifact reduction technique can lead to a significant improvement of the reconstruction quality if the smoothing function is chosen appropriately.  We will explore more experiments and report in-depth results in a future publication. 

%we have demonstrated some numerical experiments for the limited data problem in two dimension. We

\section*{Acknowledgement}
The work of L.L.B. was supported by the National Science Foundation Major Research Instrumentation Program, grant 1229766. J. F. thanks Eric Todd Quinto for many enlightening discussions about limited data tomography and microlocal analysis over the years as well as for the warm hospitality during several visits at Tufts University. He also acknowledges support from the HC $\O$rsted Postdoc programme, co-funded by Marie Curie Actions. L.V.N.'s research is partially supported by the NSF grant \# DMS 1212125. He is thankful to Professor G. Uhlmann for introducing him to the theory of pseudo-differential operators with singular symbols presented in \cite{MU,GU,AnUhl}, whose spirit inspires this and the previous works \cite{Streak-Artifacts,Artifact-sphere-flat}. He also thanks Professor T. Quinto for the encouragement and helpful comments/suggestions.

\medskip

The authors would like to thank G. Ambartsoumian for sharing his codes in spherical radon transform, some of them are reused in this article's numerical implementations. They also thank Professor P. Stefanov for pointing out some missing references in the initial version of the article.

\appendix

\section{Appendix - Background in microlocal analysis} \label{S:Micro} %In this Appendix, we recall some background in microlocal analysis that is needed to understand the article.

Let $\Og \subset \rN^n$ be an open set, and $\cT^* \Og$ be the cotangent bundle of $\Og$. For simplicity, we can consider $\cT^* \Og$ as $\Og \times \rN^n$. We also denote
$$\cT^* \Og \setminus 0=\{(x,\xi) \in \cT^* \Og: \xi \neq 0 \}. $$
Let $\mD(\Og)=C_0^\infty(\Og)$ and $\mD'(\Og)$ be the standard spaces of test functions and distributions on $\Og$. In this section, we briefly introduce some basic concepts in microlocal analysis, such as wave front set, pseudo-differential operators ($\Psi$DO), and Fourier Integral Operators (FIOs).  Extensive presentations can be found in \cite{hormander71fourier,Ho1-old,TrPseu,TrFour}. The use of microlocal analysis in geometric integral transforms are pioneered in \cite{Guillemin-Zoll,GrUDuke,GrUCon,QT93}. Its extensive uses in the studies of spherical mean transform can be found in many works, see \cite{louis00local,XWAK08,QuintoSONAR,AMP,FQ14,Artifact-sphere-flat}, just to name a few.

\subsection{Wave Front Sets} 
\begin{defi} [Wave Front Set \cite{hormander71fourier}] Let $f \in \mD'(\Og)$ and $(x_0,\xi_0) \in \cT^* \Og \setminus 0$. Then, $f$ is microlocally smooth at $(x_0,\xi_0)$ if there is a function $\varphi \in C_0^\infty(\Og)$ satisfying $\varphi(x_0) \neq 0$ and an open cone $V$ containing $\xi_0$, such that $\mF (\varphi f)$ is rapidly decreasing in $V$. That is, for any $N>0$  there exists a constant $C_N$ such that
$$|\mF(\varphi f)(\xi)| \leq C_N (1+|\xi|)^{-N}, \mbox{ for all } \xi \in V.$$ The {\bf wavefront set} of $f$, denoted by $\wf(f)$, is the complement of the set of all $(x_0,\xi_0) \in \cT^* \Og$ where $f$ is microlocally smooth.
\end{defi} 

\medskip

An element $(x,\xi) \in \wf(f)$ indicates not only the location $x$ but also the {\bf direction} $\xi$ of a singularity of $f$. For example, if $f$ is the characteristic function of an open set $\mO \Subset \Og$ with smooth boundary $\pdh \mO$, then $(x,\xi) \in \wf(f)$ if and only if $x \in \partial \mO$ and $\xi$ is perpendicular to the tangent plane of $\partial \mO$ at $x$. Detailed discussion can be found in \cite{Petersen-Book} and, more briefly, in \cite{FQ14}.

\medskip

Let $\mT$ be a bounded linear operator from $\mE'(\Og_1)$ to $\mD'(\Og_2)$. The following rule provides an estimate of $\wf(f)$ in terms of $\wf(\mT f)$, see \cite[Theorem 2.5.14] {hormander71fourier}: 
\begin{theorem} \label{T:cal-wave} Let $\mu$ be the Schwartz kernel of $\mT$. Assume that $\wf(\mu) \subset (\cT^* \Og_2 \setminus 0) \times (\cT^* \Og_1 \setminus 0)$, then  \begin{eqnarray*} \wf(\mT f) \subset \wf(\mu)' \circ \wf(f). \end{eqnarray*}
\end{theorem}
\noindent In the above theorem and elsewhere, $\wf(\mu)'$ is the twisted canonical associated to $\wf(\mu)$
$$\wf(\mu)' =\{(x,\xi; y, -\eta): (x,\xi; y, \eta) \in \wf(\mu) \},$$
and
\begin{eqnarray*} \wf(\mu)' \circ \wf(f) :=\{(x,\xi): (x,\xi; y, \eta) \in \wf(\mu)', \mbox{ for some } (y,\eta) \in \wf(f)\}. \end{eqnarray*}

\medskip

The following theorem provides the product rule for wave front set, see \cite[Theorem 2.5.10]{hormander71fourier}:
\begin{theorem} \label{T:product} Let $u,v$ be two distributions on $\Og$. Then the product $uv$ is well defined unless $(x,\xi) \in \wf(u)$ and $(x,-\xi) \in \wf(v)$ for some $(x,\xi)$. Moreover,
$$\wf(uv) \subset \{(x,\xi+\eta): (x,\xi) \in \wf(u) \mbox{ or } \xi=0, (x,\eta) \in \wf(v) \mbox{ or } \eta =0\}. $$
\end{theorem}

The following theorem provides the composition rule for wave front sets, see \cite[Theorem 2.5.15]{hormander71fourier}:

\begin{theorem} \label{T:compose} Let $\mT_1$ and $\mT_2$ be linear transformations whose Schwartz kernels are $\mu_1 \in \mD'(\Og_1 \times \Og_2)$ and $\mu_2 \in \mD'(\Og_2 \times \Og_3)$. We assume that $\wf(\mu_1) \subset (\cT^* \Og_1 \setminus 0) \times (\cT^* \Og_2 \setminus 0)$ and $\wf(\mu_2) \subset (\cT^* \Og_2 \setminus 0) \times (\cT^* \Og_3 \setminus 0)$. Then, the  Schwartz kernel $\mu$ of $\mT_1 \circ \mT_2$ satisfies:
$$\wf(\mu)' \subset \wf(\mu_1)' \circ \wf(\mu_2)' .$$
\end{theorem}

\medskip

The following definition helps to quantify the strength of a singularity:
\begin{defi}[Sobolev Wave Front Set \cite{Petersen-Book}] Let $f \in \mD'(\Og)$ and $(x_0,\xi_0) \in \cT^*\Og \setminus 0$. Then $f$ is in the space $H^s$ microlocally at $(x_0,\xi_0)$ if there is a function $\varphi \in C_0^\infty(\Og)$ satisfying $\varphi(x_0) \neq 0$ and a function $u(\xi)$ homogeneous of degree zero and smooth on $\rN^n \setminus 0$ with $u(\xi_0) \neq 0$, such that $$u(\xi) \, \mF (\varphi f)(\xi) \in L^2(\rN^n, (1+|\xi|^2)^s).$$
The {\bf $H^s$-wave front set} of $u$, denoted by $WF_s(u)$, is the complement of the set of all $(x_0,\xi_0) \in \cT^* \Og$ where $u$ is not microlocally in the space $H^s$. 
\end{defi} 
\noindent One can use the Sobolev orders to compare the singularities $(x_1,\xi_1) \in \wf(f_1)$ and $(x_2,\xi_2) \in \wf(f_2)$, where $f_1,f_2$ are two distributions, not necessarily defined on the same set. For example, $(x_1,\xi_1)$ is stronger than $(x_2,\xi_2)$ if there is $s$ such that  $(x_1,\xi_1) \in \wf_s(f_1)$ but  $(x_2,\xi_2)  \not \in \wf_s(f_2)$. 

\medskip

We also introduce the definition of conormal distribution (e.g., \cite{hormander71fourier,GrU-Functional,FLU,Suresh}):
\begin{defi}
Assume that $S \subset \Og$ is a smooth surface of co-dimension $k$. Let $h \in C^\infty(\Og,\rN^k)$ be a defining function for $S$ with $rank(d h) = k$ on $S$. The class $I^r(S)$ consists of the distributions which locally can be written down as a finite sum of oscillatory integrals of the form
$$u(x) = \intl_{\rN^k} e^{i h(x) \cdot \theta} a(x,\theta) \, d\theta, $$
where $a \in S^r(\Og \times \rN^k)$.
\end{defi}
In the above definition and elsewhere in this article, we use the following definition of a symbol:
\begin{defi}[\cite{hormander71fourier}] Let $\Og \subset \rN^n$ be an open set. The space $S^m(\Og \times \rN^N)$ consists of all functions $a \in C^\infty(\Og \times (\rN^N \setminus 0))$ such that for any multi-indices $\ag,\bg$ and $K \Subset \Og$, there is a positive constant $C_{\ag,\bg,K}$ such that
\begin{equation} \label{E:sym-in} |\pdh_x^\ag \pdh_\theta^\bg a(x,\xi)| \leq C_{\ag,\bg,K} (1+ |\xi|)^{m -|\ag|}, \quad \mbox{ for all } (x,\xi)  \in K \times (\rN^N \setminus 0).\end{equation}
The elements of $S^m(\Og \times \rN^N)$ are called symbols of order $m$. 
\end{defi}

We note that if $u \in I^r(S)$, then $\wf(u) \subset N^* S$ (see, e.g, \cite{hormander71fourier}), where $N^*S$ is the conormal bundle of $S$.

\begin{defi} Let $f \in \mD'(\rN^2)$. We say that $(x_0,\xi_0) \in \wf(f)$ is a conormal singularity of order $r$ to the surface $S$ if there is $u \in I^r(S)$ such that
$$(x_0,\xi_0) \not \in \wf(f-u). $$
\end{defi}
One can use the order $r$ to compare two conormal singularities $(x,\xi) \in \wf(f_1)$ (along the surface $S_1$) and $(y,\eta) \in \wf(f_2)$ (along the surface $S_2$), where $f_1,f_2$ are two distributions on $\Og$. For example, $(x,\xi)$ is {\bf weaker} than $(y,\eta)$, if there is $r \in \rN$ such that $(x,\xi)$ is of order $r$ while $(y,\eta)$ is not.

\subsection{Pseudo-Differential Operators ($\Psi$DOs)} \label{A:PDO} Given $a \in S^m(\Og \times \rN^n)$, the operator $\mT: C_0^\infty(\Og) \to C^\infty(\Og)$ defined by the oscillatory integral
\begin{equation} \label{E:PDO} \mT f(x) = \frac{1}{(2\pi)^n} \intl_{\Og} \intl_{\rN^n}  e^{i (x-y) \cdot \xi} \, a(x,\xi) \, f(y) \, d\xi \, dy, \end{equation}
is called a pseudo-differential operator ($\Psi$DO) of order $m$ with the (full) symbol $a(x,\xi)$. 

\medskip

Since the above integral may not converge in the classical sense, the expression in \eqref{E:PDO} needs to be properly defined, see, e.g., \cite[Proposition 1.1.2 ]{hormander71fourier}. Given this proper definition, $\mT$ extends continuously to $\mE'(\Og) \to \mD'(\Og)$. In particular, it can be shown that a   pseudo-differential operator $\mT$ does not generate new singularities. That is, \cite[Page 131]{hormander71fourier} \begin{equation} \label{E:wave-front-PDO} \wf(\mT f) \subset \wf(f).\end{equation}
Moreover, if $f$ is in the space $H^s$ microlocally at $(x_0,\xi_0)$ then $\mT f$ is in space $H^{s-m}$ microlocally at the same element $(x_0,\xi_0)$, see \cite{Petersen-Book,TrPseu}.

\medskip

We will denote $\mu \in I^m(\Delta)$ if $\mu$ is the Schwartz kernel of a pseudo-differential operator of order $m$. Let us define a technical term that is used in the statement of Theorem~\ref{T:Main1}~a):
\begin{defi} \label{D:PDO}
Let $A \subset \Delta$ be a conic set that is open in the induced topology of $\Delta$. We say that {\bf near $A$, $\mu$ is microlocally in the space $I^m(\Delta)$ with the symbol $\sg$} if the following holds: for each element $(x^*,\xi^*;x^*,\xi^*) \in A$ there exist $\mu_* \in I^m(\Delta)$ such that $$(x^*,\xi^*; x^*, \xi^*) \not \in \wf(\mu-\mu_*)',$$ and the symbol of $\mu_*$ is equal to $\sg(x,\xi)$ in a conic neighborhood if $(x^*,\xi^*)$. 
\end{defi}

The following result was proved in \cite{Streak-Artifacts}, which is used to explain our result in Sections~\ref{S:2D}~\&~\ref{S:3D}:
\begin{lemma} \label{L:Pet}
Let $\mT: \mE'(\Og) \to \mD'(\Og)$ be a linear operator whose Schwartz kernel $\mu \in \mD'(\Og \times \Og)$ satisfies $\wf(\mu)' \subset (\cT^*\Og\setminus 0)\times (\cT^*\Og\setminus 0)$.  Assume that $\mu$ is microlocally in $I^m(\Delta)$ near $A$ with the symbol $\sg(x,\xi)$. Let $(x^*,\xi^*) \in \cT^*\Og \setminus 0$ such that $(x^*,\xi^*;x^*,\xi^*) \in A$ and in a conic neighborhood of $(x^*,\xi^*)$
$$|\sg(x,\xi)| \geq C(1+|\xi|)^m, \quad |\xi| \geq 1.$$
Assume further that \begin{equation} \label{E:wave-x*} \{(x^*,\xi^*;y,\eta) \in \wf(\mu)': (y,\eta) \in \wf(f)\} \subset \Delta.\end{equation} Then, for any $s \in \rN$, \begin{equation*} (x^*,\xi^*) \in WF_s(f) \mbox{ if and only if } (x^*,\xi^*) \in WF_{s-m}(\mT f).\end{equation*}
\end{lemma}

\subsection{Fourier Integral Operators (FIOs)}  \label{A:FIOs} %In this section, we introduce two special types of FIOs that are needed in this article, one to deal with the two dimensional problem and one with three dimensional problem. The interested reader is referred to \cite{hormander71fourier,Ho1-old,TrFour,Duistermaat} for a comprehensive presentation on FIOs.

In this section, we introduce some special Fourier distributions which are needed in this article. The reader is referred to, e.g., \cite{hormander71fourier,TrFour,Duistermaat} for the general theory of the topic. Let $X$ and $Y$ be two manifolds of dimension $n_X$ and $n_Y$, respectively, and $\Llg$ be a homogeneous canonical relation in $(\cT^*X \setminus 0) \times (\cT^* Y \setminus 0)$. Then, there is  an associated class of Fourier distributions of order $m$, denoted by $I^m(\Llg)$. Each element of $I^m(\Llg)$, called a Fourier integral distribution of order $m$, is a distribution $\mu \in \mD'(\Og \times \Og)$ such that it can be locally written down in the form
$$\mu(x,y) = \intl_{\rN^N} e^{i \phi(x,y,\llg)} \, a(x,y,\llg) \,d\llg.$$
Here, \footnote{The order $m+(n_X+n_Y-2N)/4$ of $a$ specified here is due to, e.g., \cite[pp. 115]{hormander71fourier}.}  \begin{equation} \label{E:Ord} a(x,y,\llg) \in S^{m+(n_X+n_Y-2N)/4} (X \times Y\times \rN^N), \end{equation} and  $\phi$ is a phase function associated to $\Llg$. That is,
$\phi=\phi(x,y,\llg) \in C^\infty(X \times Y \times (\rN^N \setminus 0))$ satisfies \begin{itemize} \item[1)] $\phi$ is homogeneous of degree $1$ in $\llg$, \item[2)] $\phi_x \neq 0$ and $\phi_y\neq 0$ on the set $$\mC=\{(x,y,\llg) \in X \times Y \times (\rN^N \setminus 0): d_\llg\phi =0\},$$
\end{itemize}
and \begin{equation*} \{(x,d_x \phi; y, - d_y \phi): (x,y,\llg) \in \mC \} \subset  \Llg. \end{equation*}
The following result gives us a rule for the wave front set of a Fourier integral distribution (see, e.g., \cite[Theorem 3.2.6]{hormander71fourier}):
\begin{theorem} \label{T:wave-FIO} Let $\mu \in I^m(\Llg)$ then $$\wf(\mu)' \subset \Llg.$$ \end{theorem}
\noindent The linear operator $\mT: \mD(Y) \to \mD'(X)$ whose Schwartz kernel is $\mu$ is a Fourier integral operator (FIO) of order $m$. With a slight abuse of notation, we also write $\mT \in I^{m}(\Llg)$.

\medskip

The following technical term is used in the statement of Theorem~\ref{T:Main1}~b):
\begin{defi} \label{D:FIO} Let $A \subset \Llg$ be an open conic set in the induced topology of $\Llg$. We say that {\bf near $A$, $\mu$ is microlocally in the space $I^m(\Llg)$} if the following holds: for each element $(x^*,\xi^*;y^*,\eta^*) \in A$ there exists $\mu_* \in I^m(\Llg)$ such that $$(x^*,\xi^*; y^*, \eta^*) \not \in \wf(\mu-\mu_*)'.$$ 
\end{defi}

%We also recall that a operator $\mF: \mE'(\Og) \to \mD'(\Og)$ whose Schwartz kernel is $\mu \in I^{m}(\Llg)$ is called a Fourier integral operator of order $m$. For the sake of simplicity, we also denote $\mF \in I^{m}(\Llg)$. 

\medskip

The following result (see \cite[Theorem 4.3.2]{hormander71fourier}) is used to analyze the mapping properties of the FIOs discussed below: 

\begin{theorem} \label{T:Ho}
Let $\Og \subset \rN^n$ and $\Llg \subset (\cT^* \Og \setminus 0) \times (\cT^* \Og \setminus 0)$ be a homogeneous canonical relation such that both of its left and right projections on $\Og$ have surjective differentials. Assume that the differentials of the left and right projections $\Llg \to \cT^* \Og$ have rank at least $l + n$. Then, every $\mT \in I^m(\Llg)$ maps continuously from $H^s_{comp}(\Og)$ to $H_{loc}^{s- m - \frac{n-l}{2}}(\Og)$. 
\end{theorem}

\medskip

In the discussion below, we introduce two classes of Fourier distributions whose canonical relation is defined by rotations around a point or around tangent lines of a smooth curve, respectively.

\medskip

\subsubsection{\bf Fourier distributions associated to a point}  \label{S:FIO-point} Let us now introduce the class of Fourier distributions whose canonical relation is defined by the rotations around a point. For this type of distribution, we only consider $\Og \subset \rN^2$. 
%This class of Fourier distributions is used in the statement and proof of Theorem~\ref{T:Main1}~b) (namely, the class $I^m(\Llg_\pm)$) and Theorem~\ref{T:Main2}~b) (the class $I^m(\Llg_j)$, $j=1,\dots,4$). 

\medskip

Let $x_0 \in \rN^2$ such that $x_0 \not \in \Og$. We define the following homogeneous canonical relation in $(\cT^* \setminus 0) \times (\cT^* \Og \setminus 0)$ $$\Llg_{x_0}=\{(x,\tau (x-x_0);y, \tau(y-y_0)) \in (\cT^* \Og \setminus 0) \times (\cT^* \Og \setminus 0): |x-x_0| = |y-x_0| \}.$$
That is, $\Llg_{x_0}$ is defined by rotating $(y,\eta=\tau(y-x_0)) \in \cT^*\Og \setminus 0$, which pass through $x_0$, around $x_0$. 

In Section~\ref{S:2D}, we work with the following explicit form of an element $\mu \in I^m(\Llg_{x_0})$:
\begin{equation*} \mu(x,y) =  \intl_{\rN} e^{i (|x-x_0|^2-|y-x_0|^2) \, \llg} a(x,y,\llg)  d\llg,\end{equation*} where $a \in S^{m+\frac{1}{2}}(\Og \times \Og \times \rN)$. 

\medskip

The following result is a direct consequence of Theorem~\ref{T:Ho}:
\begin{lemma} \label{L:x0}
Let $\mT \in I^m(\Llg_{x_0})$. Then, $\mT$ maps continuously from $H^s_{comp}(\Og)$ to $H^{s-m-\frac{1}{2}}_{loc} (\Og)$.
\end{lemma}

%\begin{coro} \label{C:Ho}
%Let $\mT: \mE'(\Og) \to \mD'(\Og)$ be a linear operator whose Schwartz kernel $\mu \in \mD'(\Og \times \Og)$ satisfies $\wf(\mu) \subset (\cT^*\Og \setminus 0)\times (\cT^*\Og\setminus 0)$ and on $A \subset \Llg$, $\mu$ is microlocally in $I^m(\Llg)$. Assume that $(x^*,\xi^*) \in \wf_s(\mT f)$ and $$\{(x^*,\xi^*;y,\eta) \in \wf(\mu)' \cup \Llg: (y,\eta) \in \wf(f)\} \mbox{ is a compact subset of  A}.$$  Then, there is $(y^*,\xi^*) \in WF_{s+m+\frac{1}{2}}(f)$ 
%\end{coro}

\begin{proof}
We only need to apply Theorem~\ref{T:Ho} with $l=1$. 
\end{proof}

The following result is used to analyze the strength of artifacts in Section~\ref{S:2D}. Its proof is almost exactly the same as that of \cite[Corollary 2.15]{Streak-Artifacts}. We skip it for the sake of brevity.
\begin{lemma} \label{L:Ho}
Let $\mT: \mE'(\Og) \to \mD'(\Og)$ be a linear operator whose Schwartz kernel $\mu \in \mD'(\Og \times \Og)$ satisfies $\wf(\mu) \subset (\cT^*\Og\setminus 0)\times (\cT^*\Og\setminus 0)$.  Assume that $\mu$ is microlocally in $I^m(\Llg)$ near an open conic set $A \subset \Llg$. Let $(x^*,\xi^*) \in \wf(\mT f) \cap \pi_L(\Llg)$ such that $$\{(x^*,\xi^*;y,\eta) \in \wf(\mu)' \cup \Llg: (y,\eta) \in \wf(f)\} \mbox{ is a compact subset of  A}.$$ If $(x^*,\xi^*) \in WF_{s}(\mT f)$, then there is $(y^*,\eta^*) \in \cT^*\Og \cap \Llg$ such that $$ (x^*,\xi^*;y^*,\eta^*) \in \Llg \mbox{ and } (y^*, \eta^*) \in WF_{s+m+\frac{1}{2}}(f).$$
\end{lemma}

\medskip

The following result is useful to analyze the artifacts when the original singularities are conormal. Its proof is almost exactly the same as that of \cite[Theorem 2.16]{Streak-Artifacts}. We skip it for the sake of brevity.
\begin{lemma} \label{L:Spread}
Suppose that all the assumptions in Lemma~\ref{L:Ho} hold. Assume further that:
\begin{itemize}
\item [1)] There are at most finitely many $y^* \in \Og$ such that $$(x^*,\xi^*; y^*, \eta^*= \tau(y^* - x_0)) \in \Llg \mbox{ and } (y^*,\eta^*) \in \wf(f).$$
\item [2)] Each such $(y^*,\eta^*)$ is a conormal singularity of order $r$ along a curve $\mC$ whose contact order with $\uS(x_0,|x-x_0|=|y-x_0|)$ is exactly $1$.\footnote{We note here that the contact order is always at least $1$, since both curves are perpendicular to $\eta^*$ at $y^*$. Therefore, the condition on the contact order is quite generic.}
 \end{itemize}
Then, $(x^*,\xi^*)$ is a conormal singularity of order at most $m+r$ along the circle $\uS(x_0, |x^*-x_0|)$.
\end{lemma}

\medskip

\subsubsection{\bf Fourier distributions associated to a smooth curve} Let us consider $\Og \subset \rN^3$. We introduce a class of Fourier distributions, whose canonical relation is defined by the rotations around tangent lines of a smooth curve. This class of Fourier distributions appears in the statement and proof of Theorem~\ref{T:Main2}~b). Let $\ga$ be a closed smooth curve in $\rN^3$ parametrized by the parameter $s$. Assume that $\ga \cap \Og = \emptyset$. We define the following homogeneous canonical relation in $(\cT^* \Og \setminus 0) \times (\cT^* \Og \setminus 0)$
\begin{multline*}\Llg_\ga = \big \{(x,\tau \, (x-\ga(s)); \, y,\tau \, (y-\ga(s))): |x-\ga(s)| = |y-\ga(s)|,\\ \left<x-\ga(s),\ga'(s) \right> = \left<y-\ga(s),\ga'(s) \right>, x \in \Og, y \in \Og, s \in \rN, 0 \neq \tau \in \rN \big \}.\end{multline*}
That is, $\Llg_\ga$ is defined by rotating an element $(y,\tau \, (y-\ga(s)))$, that passes through $z=\ga(s)$, around the the tangent line of $\ga$ at $z$. In Section~\ref{S:3D}, we make use of this class $I^m(\Llg_\ga)$. We state here some needed basic facts of this class. 

\medskip

The following property is a direct consequence of Theorem~\ref{T:Ho} (for $l=1$):
\begin{lemma}
Assume that $\mF \in I^{m}(\Llg)$. Then, $\mF$ is a continuous map from $H^{s}_{com}(\Og) \to H_{loc}^{s-m-\frac{1}{2}}(\Og)$. 
\end{lemma}
We note that:
$$\pi_L(\Llg) = \pi_{R}(\Llg) =  \big \{(x,\tau \, (x-\ga(s))): x \in \Og, s \in \rN, 0 \neq \tau \in \rN \big \}.$$
The following result is a microlocal version of the above result, which is used in Section~\ref{S:3D} to analyze the strength of artifacts. Its proof is almost exactly the same as that of \cite[Corollary 2.15]{Streak-Artifacts}. We skip it for the sake of brevity.
\begin{lemma} \label{L:Ho3D}
Let $\mT: \mE'(\Og) \to \mD'(\Og)$ be a linear operator whose Schwartz kernel $\mu \in \mD'(\Og \times \Og)$ satisfies $\wf(\mu) \subset (\cT^*\Og\setminus 0)\times (\cT^*\Og\setminus 0)$.  Assume that $\mu$ is microlocally in $I^m(\Llg)$ near an open conic set $A \subset \Llg$. Let $(x^*,\xi^*) \in \wf(\mT f) \cap \pi_L(\Llg)$ such that $$\{(x^*,\xi^*;y,\eta) \in \wf(\mu)' \cup \Llg: (y,\eta) \in \wf(f)\} \mbox{ is a compact subset of  A}.$$ If $(x^*,\xi^*) \in WF_{s}(\mT f)$, then there is $(y^*,\eta^*) \in \pi_R(\Llg)$ such that $$ (x^*,\xi^*;y^*,\eta^*) \in \Llg \mbox{ and } (y^*, \eta^*) \in WF_{s+m+\frac{1}{2}}(f).$$
\end{lemma}

%\bibliography{research}
%\bibliographystyle{alpha}

\def\dbar{\leavevmode\hbox to 0pt{\hskip.2ex \accent"16\hss}d}

\end{document}